\documentclass[12pt]{article}

\usepackage{amssymb, amsmath, amsthm}
\usepackage{mathrsfs}
\usepackage[left=2cm,top=4cm,bottom=2.5cm,right=2cm]{geometry}
\usepackage[utf8]{inputenc} 
\usepackage[T1]{fontenc}    
\usepackage{hyperref}       
\usepackage{url}            
\usepackage{booktabs}       
\usepackage{amsfonts}       
\usepackage{nicefrac}       
\usepackage{microtype}      
\usepackage{xcolor}         
\usepackage{physics}        
\usepackage{amssymb, amsmath, amsthm}
\usepackage{xpatch}
\usepackage{graphicx}
\usepackage{authblk}
\usepackage[symbol]{footmisc}
\usepackage{multirow}       
\usepackage{diagbox}       
\usepackage{makecell}
\usepackage{caption}
\usepackage{enumitem}
\usepackage{comment}
\usepackage{tikz}
\usepackage{algorithm2e}

\makeatletter
\AtBeginDocument{\xpatchcmd{\@thm}{\thm@headpunct{.}}{\thm@headpunct{}}{}{}}
\makeatother

\newcommand{\mathbbm}[1]{\text{\usefont{U}{bbm}{m}{n}#1}}
\newcommand{\floor}[1]{\left\lfloor #1 \right\rfloor}
\DeclareMathOperator\arctanh{arctanh}

\DeclareMathOperator{\diag}{diag}

\DeclareMathOperator{\sinc}{sinc}
\DeclareMathOperator{\sgn}{sgn}

\newtheorem{theorem}{Theorem}
\newtheorem*{theorem*}{Theorem}
\newtheorem{proposition}[theorem]{Proposition}
\newtheorem{corollary}[theorem]{Corollary}
\newtheorem{lemma}{Lemma}
\newtheorem*{lemma*}{Lemma}
\newtheorem{definition}{Definition}
\newtheorem*{definition*}{Definition}
\newtheorem{remark}{Remark}
\newtheorem*{remark*}{Remark}

\newtheorem*{fact*}{Fact}

\newtheoremstyle{theoremnoperiod}
    {\medskipamount}%
    {\medskipamount}%
    {\itshape}%
    {\parindent}%
    {\scshape}%
    {}%
    {1em}%
    {}%
\theoremstyle{theoremnoperiod}
\newtheorem*{Restatementlemma*}{Restatement of Lemma\hspace*{-9pt}}  
\newtheorem*{Restatementtheorem*}{Restatement of Theorem\hspace*{-9pt}}
\newtheorem*{Restatementproposition*}{Restatement of Proposition\hspace*{-9pt}}

\begin{document}

\title{Optimal gentle measurements of finite-dimensional quantum states}
\date{}

\author[1]{\small Cristina Butucea \footnote{\texttt{cristina.butucea@ensae.fr}}}
\author[2]{\small Jan Johannes \footnote{\texttt{johannes@math.uni-heidelberg.de}}}
\author[2]{\small Henning Stein \footnote{\texttt{henning.stein@math.uni-heidelberg.de}}}

\affil[1]{\footnotesize CREST, ENSAE, Institut Polytechnique de Paris, 91120 Palaiseau, France}
\affil[2]{\footnotesize Heidelberg University, 69120 Heidelberg, Germany}

\maketitle

\begin{abstract}
We study the task of estimating a $d-$dimensional quantum state $\rho$ under the constraint that the measurement is $\alpha-$gentle. Such measurements $M$ do not collapse the state; they issue both a random variable $R^M = \omega$ containing statistical information and a post-measurement state $\rho_{M \to \omega}$ such that $\|\rho_{M\to \omega} - \rho\|_{Tr} \leq \alpha$. We describe gentle measurements and their connection to quantum differential privacy. 

Our results show that the optimal minimax estimation rate in Frobenius norm is of order $d^3/(n \alpha^2)$, instead of $d^2/n$ for general measurements. Moreover, for rank $r$ states with $r\leq d$ we prove that the optimal minimax rate is $rd^2/(n \alpha^2)$, instead of $rd/n$. Very surprisingly, the loss for gentleness $d/\alpha^2$ scales with the ambient dimension of the Hilbert space, rather than the number of parameters $rd$, typically seen in classical differential privacy.

We propose optimal gentle measurements and indicate how they can be physically implemented using an ancillary state and a CNOT gate to entangle it with the initial state. We notice that the resulting random variable has a likelihood that satisfies local differential privacy. Lower bounds are proven through a new quantum information-theoretic inequality applied to well chosen families of states in the manifold of (small-rank) quantum states. 
\end{abstract}

\section{Introduction}
One of the most remarkable insights of quantum mechanics is that measurements on a microscopic level behave fundamentally different to those on a macroscopic scale. In contrast to what one would expect from classical physics, any observation of a quantum system is both non-deterministic as well as \emph{destructive}. These properties of quantum systems inherently confront researchers who seek to gain information on the system with a \emph{statistical inverse problem}. One such task is \emph{quantum state estimation}, also called \emph{quantum state tomography}, in which one aims at learning a complete description of the state $\rho$ of a quantum system by repeatedly probing a large quantity of identical individual systems. Our goal is to study the statistical properties of quantum state tomography while quantifying the impact of the destruction of the system by the measurement and keeping it below a prescribed level. 

\smallskip

The earliest procedures for quantum tomography \cite{smithey_measurement_1993}, being mostly physically motivated, lacked rigorous statistical foundations. As experiments showed that the creation and manipulation of quantum systems became ever more feasible \cite{mooij_josephson_1999, hannemann_self-learning_2002, haffner_scalable_2005}, so arose the need for the construction of efficient methods for reconstruction that are founded on statistical guarantees. When this task was brought to the attention of the statistical community \cite{artiles_invitation_2005} researchers took on the task of developing statistically optimal procedures for a variety of different tasks using a variety of different methods. 

In the typical formulation of quantum statistics we assume that we are given access to a quantum system in some unknown state $\rho$ that is part of a quantum statistical model $\mathcal{S} $. By performing a measurement $M$ we obtain a random outcome, \emph{i.e.} a random variable $R^M$, with values in some space $\Omega$ on which we can then perform statistical analysis. The measurement can therefore be seen as transporting the quantum state $\rho$ into some likelihood over $\Omega$. Crucially, in contrast to classical statistical procedures, researchers in quantum state tomography are confronted with the task of minimizing over both the quantum measurement and the subsequent classical post-processing. The difficulty lies therefore not only in the optimal treatment of information given by a measurement but also in the optimal obtainment of statistical information by choosing the most informationally rich measurement of the quantum state. 

Another important purely quantum phenomenon is entanglement, which, greatly simplified, allows for some part of a quantum system to influence another part without the two being connected in the classical sense. For a measurement to harness the power of this property however, any experimentalist needs to be able to hold all states in memory and manipulate them simultaneously which renders entangled (also referred to as coherent) measurements unfeasible with current technology \cite{cotler_quantum_2020}. As such, a lot of research focuses on the study of the subclass of unentangled (also independent/product) measurements which have weaker statistical properties but allow for physical implementation. 

One way of measuring the goodness of the chosen quantum statistical method is by considering its minimax risk. For a class $\mathcal{S}$ of quantum states and a class of $\mathcal{M}$ of quantum measurements, the minimax estimation risk with respect to a loss function $l$ is given by
\begin{equation}
\label{eqn::def_minimax_risk_tomography}
    \mathcal{R}(\mathcal{S}, \mathcal{M}, l) = \inf_{\hat{\rho}} \inf_{M \in \mathcal{M}} \sup_{\rho \in \mathcal{S}} \mathbb{E}_{\rho}\left[ l(\hat{\rho}, \rho) \right]
\end{equation}
where $\hat{\rho}$ is a statistical estimator defined on the random outcomes of the measurement $M$. The choice of a suitable measurement is an inherently quantum phenomenon for which there lacks a classical statistical analogue. 

The most basic class of quantum states is that of \emph{qubits}, a portmanteau combining \emph{quantum} and \emph{bit}, which only consists of two levels. In \cite{bagan_optimal_2006} it was shown that the optimal measurement scheme for qubit estimation is given by an entangled measurement and attains a minimax rate that scales like $1/n$. They show the optimality by comparing their result to the quantum Cramér-Rao-bound (QCRB) based on the Fisher-Helstrom-information that extends the classical Fisher-information into the quantum setting \cite{helstrom_quantum_1969, braunstein_statistical_1994, holevo_statistical_2001}. Interestingly, as shown in \cite{romero-isart_optimal_2005}, the minimax rate for this particular task does not change when the measurements are restricted to be unentangled.

While the task of estimation for qubits had been solved, the question of the optimal estimation strategy for general $d$-dimensional quantum system, a \emph{qudit}, remained unanswered for a longer time. While \cite{kahn_local_2009} showed that the risk scales with $1/n$, exact dependence on the dimensional $d$ remained unknown. In the following years, quantum state tomography became the topic of study for many statisticians covering the topic from various angles \cite{butucea_minimax_2007, koltchinskii_von_2011, koltchinskii_optimal_2015, cai_optimal_2016}. A particular noteworthy result was the discovery that quantum state tomography is intricately linked to the problem of recovering a high dimensional matrix from noisy observations \cite{wang_asymptotic_2013}. This allowed for various known methods in the study of low-rank matrix recovery \cite{candes_power_2010, rohde_estimation_2011, koltchinskii_nuclear-norm_2011, koltchinskii_von_2011} to be used for quantum state estimation. These insights eventually led to the discovery of sample optimal measurement schemes \cite{flammia_quantum_2012, kueng_low_2017, guta2018faststatetomographyoptimal} using techniques from compressed sensing. The corresponding lower bound was proved in \cite{haah_sample-optimal_2017} showing a minimax rate of $dr^2/n$ for state tomography in trace norm using unentangled measurements.

For general entangled measurements, it has been shown \cite{odonnell_efficient_2016, haah_sample-optimal_2017} that the minimax rate for state tomography is $dr/n$, which significantly improves on the result for unentangled measurements, especially for full rank states. However, as stated above, it is currently physically not possible to implements the measurements needed for these results.

Another typical topic in quantum statistics is that of quantum state certification, \emph{i.e.} goodness-of-fit testing. As in the classical case, testing a hypothesis state is easier than estimating an unknown state. \cite{odonnell_quantum_2021} have shown a squared minimax separation rate for quantum state certification scaling as $d/n$ for unentangled measurements. As with estimation, it has been shown \cite{bubeck_entanglement_2020, liu_role_2024, yu_sample_2021} that entanglement is necessary in order to achieve this rate and that the additional constraint of unentangled measurements introduces a penalty. Notably, this penalty is further dependent on another property of the measurement which is randomness. It was shown by \cite{liu_role_2024} that random measurements outperform fixed measurements for state certification, lowering the rate from $d^{2}/n$ to $d^{3/2}/n$. This phenomenon is, again, inherently quantum and limited to the task of state certification rather than tomography.

\smallskip

Most of the research, including the works and techniques mentioned until now, have been focused on only one of the two differences between quantum mechanical and classical observations, the randomness, while ignoring the other, the destruction. While it is well known that no information can be extracted from a quantum mechanical system without disturbing it \cite{myrvold_quantum_2009}, almost no results exist quantifying the amount of disturbance necessary in order to gain information. Such measurements that limit the amount of destruction are known as \emph{gentle measurements}. A quantum measurement $M$ is said to be gentle if it controls the trace-distance between the state before, $\rho$, and after, $\rho_{M \to \omega}$, by some small value $\alpha$, \emph{i.e.}
\begin{equation*}
    \norm{\rho - \rho_{M \to \omega}}_{Tr} \leq \alpha,
\end{equation*}
for all possible states $\rho$ and outcomes $\omega$. Given the complicated non-linear relation between a state before and after a measurement, for a long time, not much more than the rudimentary "gentle measurement Lemma" \cite{winter_coding_1999} was known about this class of measurements. 

Their importance was then recently highlighted when it was shown that gentle measurements are instrumental in the development of some quantum machine learning algorithms, especially quantum backpropagation for neural networks \cite{abbas_quantum_2023}. Quantum algorithms have already been shown to provide substantial speed-ups compared to classical computers for several tasks such as factorizing numbers \cite{shor_algorithms_1994}, the simulation of quantum systems \cite{lloyd_universal_1996} and solving linear equations \cite{harrow_quantum_2009}. With the growing capabilities of physical quantum computers \cite{arute_quantum_2019, zhong_quantum_2020, google_quantum_ai_and_collaboratorstors_quantum_2025} it is ever more important to understand the capabilities and limitations of the measurements and manipulations of quantum states.

The development of gentle measurements requires a lot more care than the treatment of general measurements. The laws of quantum mechanics dictate that any projection measurement operator completely collapses any state perpendicular to its smaller rank image space which prohibits these types of measurements from being gentle. However, most optimal quantum measurement schemes rely on such projector valued measurements \cite{bagan_optimal_2006, bubeck_entanglement_2020, haah_sample-optimal_2017, liu_role_2024,  odonnell_efficient_2016, odonnell_quantum_2021} leading us to consider an almost completely unknown type of measurement operator on which very little is known in the literature.

A significant leap forward in the study of gentle measurements came with the relation of gentle measurements to quantum differential privacy \cite{aaronson2019gentle}. Quantum differential privacy (qDP) is an extension of classical differential privacy \cite{dwork_calibrating_2006} to quantum measurements. While qDP as we define it does currently not have any statistically practical interpretation itself, it allows for the usage of the well-understood field of differential privacy \cite{duchi2014localprivacydataprocessing, butucea_interactive_2023, steinberger_efficiency_2024} to develop results for gentle measurements. While \cite{aaronson2019gentle} focuses on \emph{global gentleness}, such an approach is a again physically unfeasible as it requires access to and manipulation of a large quantity of states at once. In this work we focus on \emph{local gentleness} in which we are given $n$ identical quantum states $\rho$, which we measure separately, limiting the amount of disturbance on each state separately. 

The connection between gentleness and differential privacy was used by \cite{aaronson2019gentle} to give a global measurement procedure based on the Laplace mechanism to estimate the Hamming-weight for a multi-qubit system. Recently, \cite{butucea_sample-optimal_2025} further formalized the concept and gave a procedure for general qubits based on \emph{quantum label switching}, which generalized the optimal differentially private procedure for Bernoulli random variables \cite{steinberger_efficiency_2024} to the quantum setting. Furthermore, they developed a lower bound framework based on Le Cam's method using a newly developed quantum data-processing-inequality for gentle measurements, showing an increase in the minimax rate from $1/n$ to $1/n\alpha^2$. Their results prove and quantify the unavoidable trade-off between information gain and disturbance mirroring the trade-off in differential privacy between information gain and user privacy. Recently, \cite{butucea_locally_2026} then considered the impact of local gentleness on state certification for high dimensional quantum states and showed an optimal gentleness penalty of $d/\alpha^2$ for the goodness-of-fit test for full-rank quantum states.

\smallskip

In this paper we study the effect of local gentleness on high dimensional quantum state estimation and give minimax-optimal rates for states of arbitrary rank. \textbf{Remarkably, we demonstrate that for the corresponding quantum setting, $\alpha$-gentleness incurs a penalty of order $d/\alpha^2$ that scales only with the square root of the number of parameters for full rank states}. This is in stark contrast to the results for classical differential privacy where the additional factor incurred scales linearly with the number of parameters of the model \cite{duchi_minimax_2018}, which is of order $d^2$ for a qudit. This phenomenon suggests that the rich geometry of the quantum state space permits measurement manipulations that are significantly more efficient than their classical counterparts. This is a striking result that offers more insights into the properties of quantum systems and differential privacy. Indeed, we introduce here new quantum algorithms and methodologies that upgrade classical privacy mechanisms to the quantum world, opening the road to new quantum gentle algorithms that may result in new privacy mechanisms. 

This paper is structured as follows. In Section~\ref{sec::quantum_mechanics_background} we give a short recapitulation of the fundamental notions of quantum statistics necessary for the understanding of the background of our work. We then develop the necessary central results of gentle measurements in~\ref{sec::gentle_measurements}. In particular, we improve on the relation between gentleness and quantum differential privacy from \cite{aaronson2019gentle} and show that in certain cases, we can view the two concepts to be equivalent and our results are sharp. Section~\ref{sec::qLS_pure_states} deals with a particular gentle measurement for the probability vector of a pure state. It serves as an introductory example into the development of gentle measurements as it covers a measurement from different angles. In~\ref{sec::mathematical_formulation_of_gentle_measurement_for_probability_vector}, we start by giving its mathematical formulation with which we will be working with when proving statistical guarantees of the subsequent estimator in~\ref{sec::statistical_properties_of_probability_vector_estimator}. In~\ref{sec::relation_to_privacy_of_probability_vector_estimator}, we then demonstrate how the measurement we described earlier relates to the standard privacy kernel for multinomial distributions. Finally, in~\ref{sec::Implementation_of_measurement_of_probability_estimator}, we demonstrate how this measurement can be viewed as a basis measurement in an enlarged space as a consequence of Naimark's dilation theorem, allowing for a more feasible implementation. 
Section~\ref{sec::qLS_pure_states} serves as a prelude for Section~\ref{sec::upper_bound_qudit_estimation} in which we build upon our previous gentle measurement to construct a more sophisticated version that serves as the backbone for the gentle state estimator. By combining gentle measurements with standing state tomography methods, we build a gentle projected least squares estimator that attains a minimax rate of $\frac{d^2r}{n\alpha^2}$, where $r$ is the rank of the state to be measured, up to log factors. 

We then prove the optimality of the state estimator in Section~\ref{sec::lower_bounds_qudits}. For this, we first develop a lower bound reduction scheme for gentle measurements based on Assouad's method in~\ref{sec::reduction_scheme}. We demonstrate the power of our reduction scheme thought the construction of suitable set of quantum states in Section~\ref{sec::lower_bound_constructions}, showing tight minimax lower bounds that match the rate of our estimator up to a factor of $\log(d)$. This is achieved by two separate construction for high- and low-rank states in~\ref{sec::lower_bound_construction_full_rank} and~\ref{sec::lower_bound_construction_low_rank} respectively. For completeness, we then show a construction for pure states in~\ref{sec::lower_bound_construction_pure_states}. Finally, in~\ref{sec::lower_bound_construction_probability_vector}, we finish of the section by demonstrating that our method not only applies to state estimation itself, but can also be applied to the estimation of derived parameters of the state such as the vector of probabilities for a pure state, proving minimax optimality of the estimator constructed in Section~\ref{sec::qLS_pure_states}. We believe that the these constructions may be used for several other tasks in quantum estimation.

\section{Fundamentals of quantum statistics}
\label{sec::quantum_mechanics_background}

\subsection{Notation}
Throughout this paper we will work with the standard physical notation of quantum mechanics. A vector of the complex Hilbert space $\mathcal{H}$ (we will be almost exclusively working with $\mathcal{H} = \mathbb{C}^d$) are denoted by a "ket" $\ket{\psi}$. The corresponding "bra" is the linear operator $\bra{\psi}$ that maps a ket $\ket{\varphi}$ onto the inner product $\bra{\psi}\ket{\varphi}$ between the two. As such, $P = \ket{\psi}\bra{\psi}$ is the rank-one operator given by $P \ket{\varphi} = \ket{\psi}\bra{\psi}\ket{\varphi}$. For a matrix $A \in \mathbb{C}^{d \times d}$ we denote by $A^*$ the adjoint matrix. $A \in \mathbb{C}^{d \times d}$ is said to be positive ($A \geq 0$ in short) if $A$ is self-adjoint and positive semi-definite. If A is positive, it has a square root, i.e. there exists a unique positive matrix $B$ such that $B^2 = A$. Furthermore, by $|A| = \sqrt{A^* A}$ we denote the matrix absolute value of $A$. Using the bra-ket notation, we can write any self-adjoint matrix $A$ as $A = \sum_{j = 1}^d \lambda_j \ket{\psi_j}\bra{\psi_j}$ for an orthonormal basis $(\ket{\psi_j})_{j = 1}^d$ of $\mathbb{C}^d$. We denote operator norm of $A$ as $\norm{A}_{op}$.

\subsection{Quantum mechanics background}
In this section we recapitulate the basic notions of quantum mechanics necessary for understanding the contents of this work. The formalization we use is based on complex separable Hilbert spaces. Since we will are working with finite dimensional objects almost exclusively, the underlying Hilbert space in our case is $\mathbb{C}^d$. 

\begin{definition}[Quantum states]
\label{defn::quantum_states}
    Let $\mathbb{C}^{d}$ be the Hilbert space describing a quantum system. A state of the system is a a positive matrix $\rho \in \mathbb{C}^{d \times d}$ such that $\Tr[\rho] = 1$.
\end{definition}

If $d=2$ we say that $\rho$ is a "qubit". Qubits form the most fundamental quantum system with the lowest possible non-trivial dimension. Contrary, $d$-dimensional quantum states are often referred to as "qudits". The space of qudits (that is the space of  quantum states on $\mathbb{C}^{d}$) is denoted by $\mathcal{S}(\mathbb{C}^d)$. The subspace of quantum states of rank $r$ is denoted by $\mathcal{S}_r(\mathbb{C}^d)$. We say that a state is pure if and only if it has rank one. In that case in can be written as the projector $\rho = \ket{\psi}\bra{\psi}$ for a normalized vector $\ket{\psi} \in \mathbb{C}^d$. As such, we often identify a pure state with its representational vector $\ket{\psi}$. Note that for a $\lambda = e^{i \theta} \in \mathbb{C}$ with $|\lambda| = 1$, $\ket{\psi}$ and $\ket{\psi'} = \lambda \ket{\psi}$, we have $\ket{\psi}\bra{\psi} = \ket{\psi'}\bra{\psi'}$. As such, $\ket{\psi}$ and $\ket{\psi'}$ which only differ by a phase parameter $\theta$ represent the same pure state. We can therefore identify the space of pure states $\mathcal{S}_{pure}(\mathbb{C}^d)$ with the rays of the Hilbert space $\mathbb{C}^d$. For a fixed orthonormal basis $(\ket{e_j})_{j = 1}^d$, which we will sometimes write as $(\ket{j})_{j = 1}^d$, we can write any pure state $\ket{\psi}$ as $\ket{\psi} = \sum_{j = 1}^d \gamma_j \ket{e_j}$ with $\sum_{j = 1}^d |\gamma_j|^2 = 1$. We say that $\ket{\psi}$ is in a superposition of the basis states $\ket{e_j}$ if for all $j = 1,...,d$ it holds $|\gamma_j| \neq 1$. Furthermore, the vector $p = [p_1,...,p_d]^t := \left[ |\gamma_1|^2,...,|\gamma_d|^2 \right]^t$ is the vector of probability amplitudes of $\ket{\psi}$. \medskip

States of a quantum system cannot be observed directly. Instead the observation of the state is limited to quantum measurements. Even if the state of the system is completely known, in general, the outcome of the measurement will be random and cannot be predicted beforehand. Furthermore, a fundamental result of quantum mechanics is that a measurement irreversible alters the state of the system it measures. This alteration is know as the "collapse of the wave function".
\begin{definition}[Quantum measurement]
\label{defn::quantum_measurement}
    A measurement $M$ on a quantum system $\mathbb{C}^d$ is given by a countable set of measurement operators (matrices) $M = (M_{\omega})_{\omega \in \Omega}$ such that
    \begin{equation}
    \label{eqn::completeness_relation}
        \sum_{\omega \in \Omega} E_{\omega} := \sum_{\omega \in \Omega} M_{\omega}^* M_{\omega} = \mathbbm{1}.
    \end{equation}
    for $E_{\omega} := M_{\omega}^*M_{\omega}$. Equation~\eqref{eqn::completeness_relation} is called the \emph{completeness relation}. Each measurement $M$ induces a random variable $R^M$ when measuring a quantum system. If the state of the system is $\rho$, the probability of the measurement outcome being $\omega \in \Omega$ is
    \begin{equation}
    \label{eqn::measurement_prob}
        \mathbb{P}_{\rho}^{R^M}(\{\omega\}) = \mathbb{P}_{\rho}(R^M = \omega) = \Tr\left[ \rho E_{\omega} \right].
    \end{equation}
    After the measurement the new state of the system is given by
    \begin{equation}
    \label{eqn::post_measurement_state}
        \rho_{M \to \omega} = \frac{1}{\mathbb{P}_{\rho}(R^M = \omega)} M_{\omega} \rho M_{\omega}^*.
    \end{equation}
    If $\rho = \ket{\psi}\bra{\psi}$ the post measurement state $\rho_{M \to \omega} = \ket{\psi_{M \to \omega}}\bra{\psi_{M \to \omega}}$ is also pure and equations~\eqref{eqn::measurement_prob} and~\eqref{eqn::post_measurement_state} can be written as
    \begin{equation}
    \label{eqn::post_measurement_state_pure}
    \mathbb{P}_{\rho}( R^M = \omega) = \bra{\psi}M_{\omega}^*M_{\omega}\ket{\psi} = \norm{M_{\omega} \ket{\psi}}^2 \hspace{10pt} \text{and} \hspace{10pt} \ket{\psi_{M \to \omega}} = \frac{1}{\norm{M_{\omega}}} M_{\omega} \ket{\psi}.
    \end{equation}
    respectively.
\end{definition}
There are several notions of measurements in quantum mechanics. We are working with the definition of \cite{NielsenChuang} which is the most general definition that still allows for a description of post-measurement states which is essential for working with gentle measurements. Another often used description of measurements is that of positive operator valued measurements (POVMs). $M: \Sigma \to \mathbb{C}^{d \times d}$ for a measurable space $(\Omega, \Sigma)$ such that 
\begin{enumerate}
    \item[(i)] $M(E) \geq 0$ for all $E \in \Sigma$;

    \item[(ii)] $M(\Omega) = \mathbbm{1}$;

    \item[(iii)] $M(\bigcup_{n \in \mathbb{N}} E_n) = \sum_{n \in \mathbb{N}} M(E_n)$ for a countable sequence of pairwise disjoint $E_n$.
\end{enumerate}
Note that for a measurement as we define it, the mapping
\begin{align*}
    M:& \; \mathcal{P}(\Omega) \; \to \; \mathbb{C}^{d \times d}
    \\
    & \quad E \quad \mapsto \; \sum_{\omega \in E} E_{\omega}
\end{align*}
defines a POVM on $(\Omega, \mathcal{P}(\Omega))$. While POVMs in general capture a wider class of statistical properties of measurements, their general formulation contains no description of the post-measurement state. A measurement is said to be a projector valued measurement  (PVM), if each measurement operator $M_{\omega}$ is an orthogonal projection. A special case of PVMs are basis measurements. We say that we measure a state in the basis $(\ket{e_j})_{j = 1}^d$ if we measure using $M = (\ket{e_j}\bra{e_j})_{j = 1,...,d}$. All the aforementioned descriptions have in common that they transport a quantum model to a statistical model. For a quantum model $\mathcal{S} $ on $\mathbb{C}^{d}$ a measurement $M =(M_{\omega})_{\omega \in \Omega}$ induces a statistical model
\begin{equation*}
    \left\{ \mathbb{P}_{\rho}^{R^M} \, \middle| \, \rho \in \mathcal{S} \right\} \hspace{10pt} \text{on} \hspace{10pt} (\Omega, \mathcal{P}(\Omega)).
\end{equation*}
This "quantum-to-classical randomization" allows for the definition of quantum estimators and tests. A quantum estimator of the state $\rho$ is a pair $(M, \hat{\rho})$ consisting of a measurement $M = (M_{\omega})_{\omega \in \Omega}$ and a classical estimator $\hat{\rho}: \Omega \to \mathcal{S}(\mathbb{C}^d)$. The statistical properties of quantum estimator for a state $\rho$ are then given by the properties of the classical estimator $\hat{\rho}$ on the probability space $(\Omega, \mathcal{P}(\Omega), \mathbb{P}_{\rho}^{R^M})$. Similarly, a quantum test is given by the pair $(M, \Delta)$ for a measurement $M$ and a classical test $\Delta: \Omega \to \{0,1\}$. In~\ref{eqn::types_of_measurements} we further explore different types of particular measurement schemes with properties needed for optimal gentle quantum state tomography.

\begin{definition}[Joint quantum systems]
    Let $\mathbb{C}^{d_1}$ and $\mathbb{C}^{d_2}$ be the Hilbert spaces describing two quantum systems in the states $\rho_1$ and $\rho_2$ respectively. Then the joint system is given by the tensor product $\mathbb{C}^{d_1} \otimes \mathbb{C}^{d_2}$ and the state of the joint system is $\rho_1 \otimes \rho_2$. 
\end{definition}

The state $\rho = \rho_1 \otimes \rho_2$ is called the product of the states $\rho_1$ and $\rho_2$. Conversely, any state $\rho \in \mathcal{S}(\mathbb{C}^{d_1} \otimes \mathbb{C}^{d_2})$ that can be written in the form $\rho = \rho_1 \otimes \rho_2$ for a $\rho_1 \in \mathcal{H}_1$ and $\rho_2 \in \mathcal{H}_2$ is said to be a product state. If $\rho_1$ and $\rho_2$ are pure, so is the product state. Similarly, a measurement $M = (M_{\omega})_{\omega \in \Omega}$ on $\mathbb{C}^{d_1} \otimes \mathbb{C}^{d_2}$ is said to be product if $\Omega = \Omega_1 \times \Omega_2$ and $M = M_1 \otimes M_2 = (M_{\omega_1} \otimes M_{\omega_2})_{\omega_1 \times \omega_2 \in \Omega_1 \times \Omega_2}$ for all $\omega \in \Omega$. Product systems are important to study as they correspond to independent systems. In particular, for a product state $\rho = \rho_1 \otimes \rho_2$ and a product measurement $M = M_1 \otimes M_2$, the induced random variable $R^M$ is given by the vector $R^M = (R^{M_1}, R^{M_2})$ with independent entries, where $R^{M_1}$ and $R^{M_2}$ are the outcome random variables of $M_1$ and $M_2$ on $\rho_1$ and $\rho_2$ respectively. As such, the outcome of applying the product of $n$ identical measurements $M^{\otimes n} := M \otimes ... \otimes M$ on the product state $\rho^{\otimes n} := \rho \otimes ... \otimes \rho$ are $n$ i.i.d. random variables $R^M$. \medskip

The partial trace can be viewed as the opposite operation to pairing a quantum system. Let $\mathbb{C}^{d_1} \otimes \mathbb{C}^{d_2}$ be the Hilbert spaces describing a joint quantum system in the state $\rho$. The partial trace $\Tr_2[\rho] \in \mathbb{C}^{d_1 \times d_1}$ of $\rho$ with respect to the second subsystem is the unique operator defined by
\begin{equation*}
    \Tr\left[ \Tr_2[\rho] B \right] = \Tr\left[ \rho (B \otimes \mathbbm{1}) \right] \hspace{20pt} \text{for all } B \in \mathbb{C}^{d_1 \times d_1}.
\end{equation*}
The partial trace $\Tr_2[\rho]$ describes the properties of the state $\rho$ when we are only given access to the subsystem $\mathbb{C}^{d_1}$. For a product state $\rho = \rho_1 \otimes \rho_2$ we have $\Tr_2[\rho] = \rho_1$.

\subsection{Distances on the space of states}
In order to asses the properties of an estimator as well as as the gentleness of our measurements we introduce metrics on the space on quantum states. 

\begin{definition}
    Let $\rho_0, \rho_1 \in \mathcal{S}(\mathbb{C}^d)$ be two quantum states. The trace-norm between the two states is given by 
    \begin{equation*}
        \norm{\rho_0 - \rho_1}_{Tr} = \frac{1}{2} \Tr[|\rho_0 - \rho_1|] = \frac{1}{2} \sum_{i = 1}^d |\lambda_i|,
    \end{equation*}
    where the $(\lambda_i)_{i = 1}^d$ denote the eigenvalues of $\rho_0-\rho_1$. The Frobenius-norm between the two states is given by
    \begin{equation*}
        \norm{\rho_0 - \rho_1}_F = \Tr\left[ (\rho_0 - \rho_1)^2 \right]^{\frac{1}{2}} = \left(\sum_{i = 1}^d \lambda_i^2 \right)^{\frac{1}{2}}.
    \end{equation*}
\end{definition}
The trace- and Frobenius-norm of $A$ are equivalent to the Schatten-$p$-norms of $A$ for $p = 1,2$ where the trace norm is normalized in such a way that for two quantum states $\rho_0, \rho_1 \in \mathcal{S}(\mathbb{C}^d)$ we have $\norm{\rho_0 - \rho_1} \in [0,1]$ making it the quantum generalization of the total-variation distance. As such, it may alternatively be defined via the minimal testing error between the states $\rho_0$ and $\rho_1$ as
\begin{equation*}
    \inf_{(M, \Delta)} \mathbb{P}_{\rho_0}(\Delta = 1) + \mathbb{P}_{\rho_1}(\Delta = 0) = 1 - \norm{\rho_0 - \rho_0}_{Tr}.
\end{equation*}
where the infimum is taken over all measurements $M$ and subsequent test functions $\Delta$. Furthermore, we will make use of the fact that the Frobenius-norm is induced by the inner product on $\mathbb{C}^{d \times d}$ defined $\langle A, B \rangle = \Tr\left[A^* B \right]$ for $A, B \in \mathbb{C}^{d \times d}$.
The calculation of the trace- and Frobenius-norm is difficult in general as we need to calculate the spectral decomposition of $\rho_0 - \rho_1$ but simplifies a lot if $\rho_0$ and $\rho_1$ are pure states. 

\begin{lemma}[\cite{kargin2003chernoffboundefficiencyquantum}] \label{lemmaTraceNorm}
    Let $\rho_0 = \ket{\psi_0}\bra{\psi_0}, \rho_1 = \ket{\psi_1}\bra{\psi_1}$ be two pure states. Then the trace-norm distance between the two is given by 
    \begin{equation*}
        \norm{\rho_0 - \rho_1}_{Tr} = \sqrt{1 - |\bra{\psi_0}\ket{\psi_1}|^2} \hspace{10pt} \text{and} \hspace{10pt} \norm{\rho_0 - \rho_1}_{F} = \sqrt{2} \sqrt{1 - |\bra{\psi_0}\ket{\psi_1}|^2}.
    \end{equation*}
\end{lemma}

Note that in general, while the trace- and Frobenius-norms are equivalent, this equivalence is generally dependent on the dimension of the underlying quantum system as it holds
\begin{equation*}
    \frac{1}{\sqrt{2}} \norm{\rho_0 - \rho_1}_F \leq \norm{\rho_0 - \rho_1}_{Tr} \leq \frac{\sqrt{r}}{2} \norm{\rho_0 - \rho_1}_F
\end{equation*}
for $r = \rank(\rho_0 - \rho_1)$, where equality on the left holds iff $\rho_0$ and $\rho_1$ are both pure states.

\section{Gentle measurements}
\label{sec::gentle_measurements}
As the name "collapse of the wave function" suggests, quantum states before and after a measurement are widely different. A simple example of this is the impact of the basis measurement in the $\ket{0}, \ket{1}$ basis on the pure state $\rho = \ket{\psi}\bra{\psi}$ for $\ket{\psi} = \sqrt{1 - |\eta|^2} \ket{0} + \eta \ket{1}$ for some $\eta \in \mathbb{C}$ with $|\eta| < 1$. The laws of quantum measurements dictate that the post-measurement state under after this measurement is given by $\rho_{M \to \omega} = \ket{\psi_{M \to \omega}}\bra{\psi_{M \to \omega}}$ for $\omega = 0,1$, where $\ket{\psi_{M \to 0}} = \ket{0}$ and $\ket{\psi_{M \to 1}} = \ket{1}$. The distance between the pre- and post-measurement states in trace norm is then given by
\begin{equation*}
    \norm{\rho - \rho_{M \to 0}}_{Tr} = |\eta| \hspace{20pt} \text{and} \hspace{20pt} \norm{\rho - \rho_{M \to 1}}_{Tr} = \sqrt{1 - |\eta|^2}.
\end{equation*}
For $|\eta|$ close to 0 or 1, we see that the measurement collapses the state into an almost orthogonal post-measurement state that is almost maximally far away from the initial pre-measurement state. Furthermore, the laws of quantum mechanics state that any subsequent measurement of the state will result in a certain outcome $0$ or $1$, depending on the collapse, from which we cannot infer any further on the initial state $\rho$. Let us now assume in contrast that we had performed a measurement $M$ that had the property that
\begin{equation}
\label{eqn::gentleness_property_rewritten}
    \rho_{M \to \omega} = \rho + \Delta_{\omega, \alpha} \hspace{20pt} \text{for some } \Delta_{\omega, \alpha} \text{ with } \norm{\Delta_{\omega, \alpha}}_{Tr} \leq \alpha.
\end{equation}
In that case, for any subsequent measurement $\tilde{M}$, it would hold
\begin{align*}
    \mathbb{P}_{\rho_{M \to \omega}}\left( R^{\tilde{M}} = \tilde{\omega} \right) = \mathbb{P}_{\rho}\left( R^{\tilde{M}} = \tilde{\omega} \right) + \Tr\left[ \Delta_{\omega, \alpha} M_{\tilde{\omega}}^* M_{\tilde{\omega}} \right].
\end{align*}
Since we assumed $\norm{\Delta_{\omega, \alpha}}_{Tr} \leq \alpha$, we have $\left|  \Tr\left[ \Delta_{\omega, \alpha} M_{\tilde{\omega}}^* M_{\tilde{\omega}} \right] \right| \leq \alpha$. In particular, for small $\alpha$, the outcomes of the measurement $\tilde{M}$ on the post-measurement state $\rho_{M \to \omega}$ are almost identical to the outcomes of $\tilde{M}$ on the initial state $\rho$. Property~\eqref{eqn::gentleness_property_rewritten} therefore allows for further inference on the initial state $\rho$ by measurements of the post-measurement state $\rho_{M \to y}$. This motivates the following definition of gentle measurements.

\begin{definition}
\label{defn::gentleness}
    A measurement $M = (M_{\omega})_{\omega \in \Omega}$ is called $\alpha$-gentle on a  quantum statistical model $\mathcal{S} \subseteq \mathcal{S}(\mathbb{C}^d)$ if
    \begin{equation}
        \sup_{\rho \in \mathcal{S}} \sup_{\omega \in \Omega} \norm{\rho - {\rho}_{M \to \omega}} \leq \alpha.
    \end{equation}
    If $\rho = \rho_{ 1} \otimes ... \otimes \rho_{ n} $ is a product state belonging to $ \mathcal{S}_1 \otimes ... \otimes \mathcal{S}_n =: \mathcal{S}^n$, we say that a measurement $M$ is locally-$\alpha$-gentle if it is a product measurement $M = M_1 \otimes ... \otimes M_n$ and $M_i$ is $\alpha$-gentle on $\mathcal{S}_i$ for all $i$.
\end{definition}

Gentleness has been shown by \cite{aaronson2019gentle} to be intimately related to quantum differential privacy which is a generalization of classical differential privacy to the quantum setting. Heuristically, a measurement is $\delta$-quantum differentially-private if for all possible outcomes $\omega \in \Omega$ the probabilities of obtaining the same outcome $\omega$ under a state $\rho_1$ and under a different state $\rho_2$ are close.

\begin{definition}
\label{defn::quantm_differential_pricacy}
    A measurement $M = (M_{\omega})_{\omega \in \Omega}$ is called $\delta$-quantum-differentially-private on a quantum statistical model $\mathcal{S}  \subseteq \mathcal{S}(\mathbb{C}^d)$ if
    \begin{equation}
        \mathbb{P}_{\rho_1}(R^M = \omega ) \leq e^{\delta} \mathbb{P}_{\rho_2}(R^M = \omega) \hspace{20pt} \text{ for all } \rho_1, \rho_2 \in \mathcal{S}, \; \omega \in \Omega.
    \end{equation}
    Similarly as with gentleness, we say that a product measure $M = M_1 \otimes ... \otimes M_n$ is locally-$\delta$-quantum differentially-private on $\mathcal{S} = \mathcal{S}_1 \otimes ... \otimes \mathcal{S}_n$ if each $M_i$ is $\delta$-quantum differentially-private on $\mathcal{S}_i$ for all $i$.
\end{definition}

As with classical differential privacy, we verify local gentleness and local quantum differential privacy on each register separately. As such, we often consider only the action of a single $M_i$ rather than the whole product measurement $M_1 \otimes ... \otimes M_n$. 

\medskip

The first result we proof is allows us to verify the quantum differential privacy of a measurement more easily. It shows that it suffices to check quantum differential privacy on the subset of pure states of a quantum system in order to guarantee it for all states. As pure states are often much easier to work with, this heavily simplifies the study of such measurements. Even more, we can show that verifying quantum differential privacy can be done independently of the states at hand by proving an eigenvalue property of the measurement operators.

\begin{proposition}
\label{prop::equalities_qDP}
    Let $M = (M_{\omega})_{\omega \in \Omega}$ be a quantum measurement. Then the following are equivalent:
    \begin{enumerate}
        \item[i)] $M$ is a $\delta$-quantum-differentially-private measurement on $\mathcal{S}(\mathbb{C}^d)$;

        \item[ii)] $M$ is a $\delta$-quantum-differentially-private measurement on $\mathcal{S}_{pure}(\mathbb{C}^d)$;

        \item[iii)] $\lambda_{max}(E_{\omega}) \leq e^{\delta} \lambda_{min}(E_{\omega})$ for all $\omega \in \Omega$, where $E_{\omega} = M_{\omega}^*M_{\omega}$. 
    \end{enumerate}
\end{proposition} 

\begin{proof}
\noindent

\begin{enumerate}[left=0pt, leftmargin=20pt, align=left, itemsep=0pt]

\item[i) $\implies$ ii)] Since $\mathcal{S}_{pure}(\mathbb{C}^d) \subseteq \mathcal{S}(\mathbb{C}^d)$.

\item[ii) $\implies$ iii)] Let $M$ be $\delta$-quantum-differentially-private on pure states. For any $\omega \in \Omega$ we define $E_{\omega} = M_{\omega}^*M_{\omega}$. Then for arbitrary pure states $\rho = \ket{\psi}\bra{\psi}, \rho' = \ket{\psi'}\bra{\psi'}$ it holds by definition that
    \begin{equation*}
        \bra{\psi}E_{\omega}\ket{\psi} = \mathbb{P}_{\ket{\psi}}(R^M = \omega) \leq e^{\delta}\mathbb{P}_{\ket{\psi'}}(R^M = \omega) = e^{\delta} \bra{\psi'}E_{\omega}\ket{\psi'}.
    \end{equation*}
    Now, let us choose in particular $\ket{\psi}$ and $\ket{\psi'}$ to be the eigenvectors of $E_{\omega}$ with eigenvalue $\lambda_{max}(E_{\omega})$ and $\lambda_{min}(E_{\omega})$ respectively. Then we have
    \begin{align*}
         \lambda_{max}(E_{\omega}) = \bra{\psi}E_{\omega}\ket{\psi} \quad\text{and} \quad 
          \lambda_{min}(E_{\omega}) = \bra{\psi'}E_{\omega}\ket{\psi'}.
    \end{align*}
    We conclude that $\lambda_{max}(E_{\omega}) \leq  e^{\delta} \lambda_{min}$.

\item[iii) $\implies$ i)] Note that we can write any quantum state as $\rho = \sum_{k = 1}^{d} \lambda_k \ket{\psi_k}\bra{\psi_k}$. The outcome probability for $\omega \in \Omega$ is then given by
    \begin{equation*}
        \mathbb{P}_{\rho}\left( R^M = \omega \right) = \Tr\left[ \rho E_{\omega} \right] = \sum_{k = 1}^{d} \lambda_k \Tr\left[ \ket{\psi_k}\bra{\psi_k} E_{\omega} \right] = \sum_{k = 1}^{d} \lambda_k \bra{\psi_k}E_{\omega}\ket{\psi_k}.
    \end{equation*}
    By definition of $\lambda_{max}(E_{\omega})$ and $\lambda_{min}(E_{\omega})$, since $\sum_{k = 1}^{d} \lambda_k = 1$, we have
    \begin{equation*}
        \lambda_{min}(E_{\omega}) \leq \mathbb{P}_{\rho}\left( R^M = \omega \right) \leq \lambda_{max}(E_{\omega})
    \end{equation*}
    for any $\rho \in \mathcal{S}(\mathbb{C}^d)$ . Therefore we have
    \begin{equation*}
        \mathbb{P}_{\rho}\left( R^M = \omega \right) \leq \lambda_{max}(E_{\omega}) \leq e^{\delta} \lambda_{min}(E_{\omega}) \leq e^{\delta} \mathbb{P}_{\rho'}\left( R^M = \omega \right)
    \end{equation*}
    for all $\rho, \rho' \in \mathcal{S}(\mathbb{C}^d), \omega \in \Omega$.
\end{enumerate}
\end{proof}

As quantum differential privacy is a condition purely determined by the outcome probabilities of the measurement, in order to assess the gentleness of the measurement we must fix an implementation. More precisely, for a quantum measurement $M = (M_{\omega})_{\omega \in \Omega}$ and a family of unitary matrices $(U_{\omega})_{\omega \in \Omega}$, the measurement $\tilde{M} = (U_{\omega}M_{\omega})_{\omega \in \Omega}$ also defines a quantum measurement with the same outcome probabilities as $M$. The gentleness of the measurement $\tilde{M}$ is dependent on the particular choice of the $U_{\omega}$. The most natural choice of unitary is such that $U_{\omega} M_{\omega} = |M_{\omega}|$, where $|M_{\omega}|$ is the matrix absolute value of $M_{\omega}$. The following proposition assures that such a choice leads in fact to a gentle measurement. Its proof can be found in Appendix~\ref{sec::additional_proofs_gentleness_results}.

\begin{proposition}
\label{prop::Privacy_implies_gentleness_pure}
    Let $M = (M_{\omega})_{\omega \in \Omega}$ be a quantum measurement and $E_{\omega} = M_{\omega}^*M_{\omega}$. If $M$ is $\delta$-quantum-differentially-private on $\mathcal{S}(\mathbb{C}^{d})$, then there exists an implementation $\Tilde{M}$ of $M$, given by $\tilde{M} = (|M_{\omega}|)_{\omega \in \Omega}$, such that $\Tilde{M}$ is $\alpha$-gentle on $\mathcal{S}_{pure}(\mathbb{C}^{d})$ for
    \begin{equation*}
        \alpha = \frac{e^{\frac{\delta}{2}} -1}{e^{\frac{\delta}{2}} +1} = \tanh\left( \frac{\delta}{4} \right).
    \end{equation*}
\end{proposition}

Proposition~\ref{prop::Privacy_implies_gentleness_pure} only assures gentleness on pure states. However, as with quantum differential privacy, we can show that gentleness on pure states implies gentleness on all states as long as we assume the measurement operators to be positive and self-adjoint. This assertion is the content of Proposition~\ref{prop::pure_states_imply_all_states} whose proof can again be found in Appendix~\ref{sec::additional_proofs_gentleness_results}.

\begin{proposition}
\label{prop::pure_states_imply_all_states}
    Let $M = (M_{\omega})_{\omega \in \Omega}$ be an $\alpha$-gentle measurement on $\mathcal{S}_{pure}(\mathbb{C}^d)$ such that $M_{\omega}$ is positive and self-adjoint. Then, $M$ is $\alpha$-gentle on $\mathcal{S}(\mathbb{C}^d)$.
\end{proposition}

We can combine the two aforementioned results to obtain the following corollary.

\begin{corollary}
\label{prop::Privacy_implies_gentleness}
    Let $M = (M_{\omega})_{\omega \in \Omega}$ be a quantum measurement and $E_{\omega} = M_{\omega}^*M_{\omega}$. If $M$ is $\delta$-quantum-differentially-private on $\mathcal{S}(\mathbb{C}^{d})$, then there exists an implementation $\Tilde{M}$ of $M$, given by $\tilde{M} = (|M_{\omega}|)_{\omega \in \Omega}$, such that $\Tilde{M}$ is $\alpha$-gentle on $\mathcal{S}(\mathbb{C}^{d})$ for
    \begin{equation*}
        \alpha = \frac{e^{\frac{\delta}{2}} -1}{e^{\frac{\delta}{2}} +1} = \tanh\left( \frac{\delta}{4} \right).
    \end{equation*}
\end{corollary}

\begin{proof}
    For the implementation $\Tilde{M} = (|M_{\omega}|)_{\omega \in \Omega}$ we chose in the proof of Proposition~\ref{prop::Privacy_implies_gentleness_pure}, we have that the measurement operators are positive and self-adjoint. Thus, by Proposition~\ref{prop::pure_states_imply_all_states}, we have that the same implementation $\Tilde{M}$ is $\alpha$-gentle on $\mathcal{S}(\mathbb{C}^d)$.
\end{proof}

This shows that quantum-differentially-private measurements have a gentle implementation. The following two results are now concerned with the opposite direction showing that gentle measurements are always quantum-differentially-private. 

\begin{proposition}
\label{lem::gentleness_implies_privacy}
    Let $\alpha $ in $[0, \frac{1}{2})$ and $M$ be $\alpha$-gentle on $\mathcal{S}(\mathbb{C}^d)$ with measurement operators $M_{\omega}$. Then $M$ is $\delta$ quantum differentially-private on $\mathcal{S}(\mathbb{C}^d)$ for $\delta = 2\log(\frac{1+2\alpha}{1-2\alpha})$. 
\end{proposition}

Note that in this direction we are not concerned about the particular implementation of the measurement. However, while Proposition~\ref{lem::gentleness_implies_privacy} shows that an arbitrary gentle measurement is quantum differentially private, we can improve upon this result under the additional assumption of positive-definite measurement operator. In particular, we can see that Lemma~\ref{lem::improved_constant_positivity} is the direct counterpart to Corollary~\ref{prop::Privacy_implies_gentleness}. It shows that for positive semi-definite measurement operators, gentleness and quantum differential privacy are equivalent and the constants relating the two are optimal.

\begin{lemma}
\label{lem::improved_constant_positivity}
    Let $\alpha $ in $[0, 1)$ and $M$ be $\alpha$-gentle on $\mathcal{S}(\mathbb{C}^d)$ with positive-definite measurement operators $M_{\omega}$. Then $M$ is $\delta$ quantum differentially-private on $\mathcal{S}(\mathbb{C}^d)$ for 
    \begin{equation*}
        \delta = 2\log(\frac{1+\alpha}{1-\alpha}) = 4 \arctanh(\alpha).
    \end{equation*}
\end{lemma}

From the results on gentle measurements we have seen so far, it is straight forward to show that the measurement operators $M_{\omega}$ of a gentle measurement $M$ are necessarily full rank. Since the post-measurement state $\rho_{M \to \omega}$ of a quantum state $\rho$ is a multiple of $M_{\omega} \rho M_{\omega}$, $\rho_{M \to \omega}$ has the same rank as $\rho$. As such, we see that gentle measurements leave the rank of the state invariant.

\section{Gentle estimation of a probability vector}
\label{sec::qLS_pure_states}
Before we derive a gentle quantum state estimator, let us shortly demonstrate how to construct gentle measurements using the properties derived in the previous section. We do this by considering the task of estimating the probability vector of a given pure quantum state. Let us recall that for a pure quantum state $\ket{\psi} = \sum_{j = 1}^d \gamma_j \ket{j} \in \mathcal{S}_{pure}(\mathbb{C}^d)$ its probability vector is given by $p = [p_1,...,p_d]^t =  [|\gamma_1|^2,...,|\gamma_d|^2]^t$. Note that this vector of probabilities is a fundamental property of the quantum system and an integral part in many optimal quantum estimators \cite{butucea_local_2018, guta2018faststatetomographyoptimal, odonnell_efficient_2016}. Measuring the state $\ket{\psi}$ in the computational basis $\ket{j}$ results in an outcome random variable $R^M$ that is multinomially distributed on $\{1,...,d\}$ with parameter vector $p$. In order to construct a gentle estimator of $p$ we will therefore leverage the Markov kernel for multinomials, see e.g. \cite{duchi_minimax_2018}, to the quantum level. In the following we will define a quantum measurement based on this optimal privacy kernel, which we will then prove to be gentle. After deriving minimax convergence rates on the subsequent gentle quantum estimator $\hat{p}_n$ of $p$ we will again take a step back and describe in more detail how exactly this relation to the privacy kernel is given. This section shall serve as both an instructive example on gentle measurements as well as a baseline on which to expand in order to construct a more sophisticated estimator of the high-dimensional matrices characterizing quantum states.

\subsection{Quantum gentle measurement}
\label{sec::mathematical_formulation_of_gentle_measurement_for_probability_vector}
For $\Omega = \{0,1\}^d$ and $\omega$ in $\Omega$ we define the measurement operators $M_{\delta, \omega}$ by
\begin{eqnarray}
\label{eqn::measurement_operators_gentle}
    M_{\delta, \omega} &= &\frac{1}{(e^{\frac{\delta}{2}} +1)^{\frac{d}{2}}} \begin{bmatrix}
            (e^{\frac{\delta}{4}})^{d- \norm{\omega - e_1}_1}&  & \multicolumn{2}{c}{\text{\kern0.5em\smash{\raisebox{-1ex}{\Large 0}}}}
            \\
            & \ddots & 
            \\
            \multicolumn{2}{c}{\text{\kern-0.5em\smash{\raisebox{0.75ex}{\Large 0}}}} & & (e^{\frac{\delta}{4}})^{d- \norm{\omega - e_d}_1}
    \end{bmatrix} \nonumber\\
    &=& \hspace*{-3pt}\left(\frac{e^{\frac{\delta}{2}}}{e^{\frac{\delta}{2}} + 1}\right)^{\frac{d}{2}} \sum_{k = 1}^d e^{-\frac{\delta}{4}\norm{\omega - e_k}_1} \ket{k}\bra{k},
\end{eqnarray}
where $e_k \in \{0,1\}$ denotes the vector that is zero everywhere and 1 in the $k$-th entry and $\norm{\cdot}_1$ denotes the $l_1$-norm in $\mathbb{C}^d$. The sum and the factor $e^{-\delta/4 \norm{\omega - e_k}_1}$ assure that the state does not collapse into any singular basis and assures the gentleness of the operators $M_{\delta, \omega}$. Mathematically, we can see that the eigenvalues of the operators are close together, which, as we have seen in the previous section, is characteristic of gentle measurements. Formally, we have the following result.
\begin{lemma}
\label{lem::gentleness_of_qudit_measurement}
    Let $M_{\delta, \omega}$ be the measurement operator  in~\eqref{eqn::measurement_operators_gentle}. Then the family $M_\delta$ of operators $ (M_{\delta,\omega})_{\omega \in \{0,1\}^d}$ defines a quantum measurement that is $\delta$-quantum differentially private and  $\alpha$-gentle on $\mathcal{S}(\mathbb{C}^{d})$ for $\alpha = \tanh\left( \frac{\delta}{4} \right)$. Moreover, the resulting random variable $R^{M_{\delta}}$ has likelihood
        \begin{equation}
        \label{eqn::probability_outcome_pure_state}
        \mathbb{P}_{\rho} \left( R^{M_{\delta}} = \omega  \right) = \left(\frac{e^{\frac{\delta}{2}}}{e^{\frac{\delta}{2}} + 1} \right)^d \sum_{j = 1}^d |\gamma_j|^2 e^{-\frac{\delta}{2} \norm{\omega - e_j}_1}
    \end{equation}
    and the post measurement states are given by
    \begin{equation}
        \label{eqn::post_measuremenent_state_pure_state}
        \rho_{M_{\delta} \to \omega} = \frac{1}{\mathbb{P}_{\rho} \left( R^{M_{\delta}} = \omega  \right)} M_{\delta, \omega} \rho M_{\delta, \omega}^*.
    \end{equation}
\end{lemma}
The proof of this result is a direct application of Lemma~\ref{lem::gentleness_of_gentleized_general_measurement} in Appendix~\ref{sec::additional_proofs_gentleness_results} and the Definition~\ref{defn::quantum_measurement}.
 
\subsection{Statistical properties of the estimator based on the gentle measurement}
\label{sec::statistical_properties_of_probability_vector_estimator}
Let us now demonstrate how we may use the measurement~\eqref{eqn::measurement_operators_gentle} in order to construct an $\alpha$-gentle estimator of $p = [|\gamma_1|^2,...,|\gamma_d|^2]^t$ for $\ket{\psi} = \sum_{j = 1}^d \gamma_j \ket{j} \in \mathcal{S}_{pure}(\mathbb{C}^d)$. 
Let us start by calculating the expectation of the random variable $R^{M_{\delta}}$ in \eqref{eqn::probability_outcome_pure_state}.
Here we have
\begin{align*}
    \mathbb{E}_{\rho}\left[ R^{M_{\delta}} \right]_\ell &= \sum_{\left\{ \omega \in \{0,1\}^d \middle| {\omega}_\ell = 1 \right\}} \mathbb{P}_{\rho}\left( R^{M_{\delta}} = \omega \right)
    \\
    &= \left( \frac{e^{\frac{\delta}{2}}}{e^{\frac{\delta}{2}} + 1} \right)^d \sum_{\left\{ \omega \in \{0,1\}^d \middle| {\omega}_\ell = 1 \right\}}  \sum_{k = 1}^d e^{-\frac{\delta}{2} \norm{\omega - e_k}_1} |\gamma_k|^2
    \\
    &= \left( \frac{e^{\frac{\delta}{2}}}{e^{\frac{\delta}{2}} + 1} \right)^d \Bigg[ \sum_{\left\{ \omega \in \{0,1\}^d \middle| {\omega}_\ell = 1 \right\}} e^{-\frac{\delta}{2} \norm{\omega - e_\ell}_1} |\gamma_\ell|^2 
    \\
    &\qquad \qquad \qquad \quad+ \sum_{k \neq \ell} \sum_{\left\{ \omega \in \{0,1\}^d \middle| {\omega}_\ell = 1 \right\}} e^{-\frac{\delta}{2} \norm{{\omega} - e_k}_1} |\gamma_k|^2 \Bigg].
\end{align*}
We see that, for all $k$ we have
\begin{align*}
    \sum_{\left\{ \omega \in \{0,1\}^d \middle| {\omega}_\ell = 1 \right\}} e^{-\frac{\delta}{2} \norm{\omega - e_k}_1} &= \sum_{j = 0}^{d-1} \binom{d-1}{j} e^{-\frac{\delta}{2}j}  = \left( e^{-\frac{\delta}{2}} + 1 \right)^{d-1}.
\end{align*}
Thus, we get,
\begin{align*}
    \mathbb{E}_{\rho}\left[ R^{M_{\delta}} \right]_\ell   
    &= \left( \frac{e^{\frac{\delta}{2}}}{e^{\frac{\delta}{2}} + 1} \right)^d \left( e^{-\frac{\delta}{2}} + 1 \right)^{d-1} \left( |\gamma_\ell|^2 + e^{-\frac{\delta}{2}} \sum_{k \neq \ell} |\gamma_k|^2 \right)
    \\
    &
    = \frac{e^{\frac{\delta}{2}} - 1}{e^{\frac{\delta}{2}} + 1} |\gamma_\ell|^2 + \frac{1}{e^{\frac{\delta}{2}} + 1} ,
\end{align*}
where we used the fact that $\sum_{k \neq l} |\gamma_k|^2 = 1- |\gamma_l|^2$.
This shows that the expectation of the vector $R^{M_{\delta}}$ is given by
\begin{equation*}
    \mathbb{E}_{\rho}\left[ R^{M_{\delta}} \right] =  \frac{e^{\frac{\delta}{2}}-1}{e^{\frac{\delta}{2}}+1} \begin{bmatrix}
        |\gamma_1|^2
        \\
        \vdots
        \\
        |\gamma_d|^2
    \end{bmatrix} + \frac{1}{e^{\frac{\delta}{2}}+1}.
\end{equation*}
If we perform $n$ identical measurements, we obtain $n$ i.i.d. random variables $R^{M_{\delta}}_i$.
We then consider the estimator
\begin{equation}
\label{eqn::estimator_probability_vector}
    \hat{p}_n = \left( \frac{1}{n} \sum_{i = 1}^n R^{M_{\delta}}_i - \frac{1}{e^{\frac{\delta}{2}}+1} \right) \frac{e^{\frac{\delta}{2}}+1}{e^{\frac{\delta}{2}}-1}.
\end{equation}
It is straightforward to see that $\hat{p}_n$ is unbiased, i.e. $\mathbb{E}_{\rho}\left[ \hat{p}_n \right] = [|\gamma_1|^2, ..., |\gamma_d|^2]^{\top} = p$, and that for the MSE, we have with our choice of $\delta = \delta(\alpha) = 4 \arctanh(\alpha)$:
\begin{align*}
    \mathbb{E}_{\rho}\left[ \norm{\hat{p}_n - p}_2^2 \right] 
    &= \sum_{k = 1}^d \text{Var}_{\rho}\left[ \hat{p}_k \right] 
    = \frac{1}{n} \left( \frac{e^{\frac{\delta}{2}}+1}{e^{\frac{\delta}{2}}-1} \right)^2 \sum_{k = 1}^d \text{Var}_{\rho}\left[ (R^{M_{\delta}})_k^2 \right]
    \leq 4 \frac{d}{n \alpha^2}
\end{align*}
since all coordinates of the vector $R^{M_{\delta}}$ are 0 or 1. We summarize the results of the preceding section in the following theorem.
\begin{theorem}
\label{thm::upper_bound_prob_amplitude}
    Let $\rho = \ket{\psi}\bra{\psi} \in \mathcal{S}_{pure}(\mathbb{C}^d)$ a pure qudit for $\ket{\psi} = \sum_{k = 1}^d \gamma_k \ket{k}$. Then, there exists a locally-$\alpha$-gentle measurement $M_{\delta(\alpha)}$ and subsequent estimator $\hat{p}_n$ of the probability vector $p = [|\gamma_1|^2, ..., |\gamma_d|^2]^{\top}$ of $\rho$ such that
    \begin{equation*}
        \sup_{\rho \in \mathcal{S}_{pure}(\mathbb{C}^d)} \mathbb{E}_{\rho}\left[ \norm{\hat{p}_n - p}_2^2 \right] \leq 4 \frac{d}{n \alpha^2}.
    \end{equation*}
\end{theorem}

\subsection{Relation to privacy for multinomial distributions}
\label{sec::relation_to_privacy_of_probability_vector_estimator}
Let us shortly illustrate how the measurement~\eqref{eqn::measurement_operators_gentle} relates to results from differential privacy. 
Recall that the parameter of interest is the vector of probabilities $p = [|\gamma_1|^2,...,|\gamma_d|^2]^t$. 
The task of estimating this property gently is therefore related to the task of estimating the parameter vector of a multinomial distribution privately. 
To see this, consider the basis measurement $M_B = (\ket{k}\bra{k})_{k = 1,...,d}$ with outcome probabilities
\begin{equation*}
    \mathbb{P}_{\rho}(R^{M_B} = k) = |\gamma_k|^2
\end{equation*}
for $\rho = \ket{\psi}\bra{\psi}$. 
A privatized version of the random variable $R^{M_B}$ is given by applying the generalized label switch kernel for multinomial random variables \cite{duchi_minimax_2018}. 
This mechanism is defined on the outcomes of $R^{M_B}$ by drawing $d$ independent random variables $W_j \in \{0,1\}$ according to
\begin{align}
    Q(W_j = \omega_j | R^{M_B} = k) 
    &= \frac{1}{e^{\frac{\delta}{2}}+1} \begin{cases}
        \omega_j e^{\frac{\delta}{2}} + (1-\omega_j) \hspace{10pt}& j = k
        \\
        (1-\omega_j) e^{\frac{\delta}{2}} + \omega_j \hspace{10pt}& j \neq k
    \end{cases} \label{eqn::privacy_kernel_multinomial}
\end{align}
and combining these into one vector $W \in \{0,1\}^d$.
Given an outcome $k$, the probability of obtaining an outcome $\omega \in \{0,1\}^d$ is therefore
\begin{align*}
    Q(W = w | R^{M_B} = k) &= \prod_{j = 1}^d Q(W_j = \omega_j | R^{M_B} = k)
    = \left( \frac{e^{\frac{\delta}{2}}}{e^{\frac{\delta}{2}} +1} \right)^d e^{-\frac{\delta}{2} \norm{\omega - e_k}_1},
\end{align*}
where $e_k$ denotes the vector that is zero everywhere except the $k$-th entry. We then have
\begin{align*}
    \mathbb{P}_{\rho}\left( W = \omega \right) &= \sum_{k = 1}^d Q(W = w |R^{M_B} = k) \mathbb{P}_{\rho}(R^{M_B} = k)
    \\
    &= \left( \frac{e^{\frac{\delta}{2}}}{e^{\frac{\delta}{2}} +1} \right)^d \sum_{k = 1}^d  e^{-\frac{\delta}{2} \norm{\omega - e_k}_1} |\gamma_k|^2 
    = \Tr \left[ \ket{\psi}\bra{\psi} E_{\delta, \omega} \right] = \mathbb{P}_{\rho}\left( R^{M_{\delta}} = \omega \right)
\end{align*}
This shows that measuring the state $\rho = \ket{\psi}\bra{\psi}$ with the measurement $M_\delta$ yields the same results as measuring the same state in the computational basis $\ket{k}$ to get samples from the multinomial random variable $R^{M_B}$ and then applying the classical differentially privacy mechanism for multinomial random variables~\eqref{eqn::privacy_kernel_multinomial} to these samples. However, the second method destroys the quantum information of the original state and it is therefore not gentle.
 The main difficulty in designing the gentle measurement is to manipulate the initial quantum state under gentleness constraints such that the resulting random variables are as informative as the classical differentially private random variables.

\subsection{Physical implementation}
\label{sec::Implementation_of_measurement_of_probability_estimator}
Let us finish this section by showing that the implementation of the measurement we proposed in~\eqref{eqn::measurement_operators_gentle} is feasible in practice. We start by viewing it as a basis measurement on an extended Hilbert space. It can be shown that the measurement we defined can be implemented using an enlarged quantum system. In particular, we prepare an ancillary state, i.e. a second quantum system consisting of $d$ independent qubits, \emph{i.e.} $2$-dimensional quantum states:  $\sigma = \ket{\phi_0}\bra{\phi_0}$ on $(\mathbb{C}^2)^d$, for $\ket{\phi_0} = \ket{0} \otimes ... \otimes \ket{0}$. Next, we apply a unitary transform $\tilde U_{\delta}$, which is composed of qubit rotations $R_{\delta}$ and qudit controlled NOT gates $U_1,\ldots,U_d$, on the coupled state $\rho \otimes \sigma$ entangling the two individual states. The whole system is then in the state $\rho_{\delta,\tilde{U}} = \tilde U_{\delta}(\rho \otimes \sigma)\tilde U_{\delta}^* = |\Psi_{\delta,\tilde{U}}\rangle \langle\Psi_{\delta,\tilde{U}}|$. Finally, measuring the second system only using a basis measurement $\tilde{M}_{\tilde{B}}$ gives the outcome of the measurement $M_{\delta}$ defined by~\eqref{eqn::measurement_operators_gentle}. Furthermore, considering now only the first system and ignoring the second gives exactly the post measurement state $\rho_{M \to \omega}$, showing that this process is gentle on the initial system. Formally, we describe this in the following proposition, which we prove in the Appendix~\ref{sec::additional_proofs_gentleness_introduction_section}. The algorithm scheme is shown in Figure~\ref{alg::Quantum_Label_Switch}.

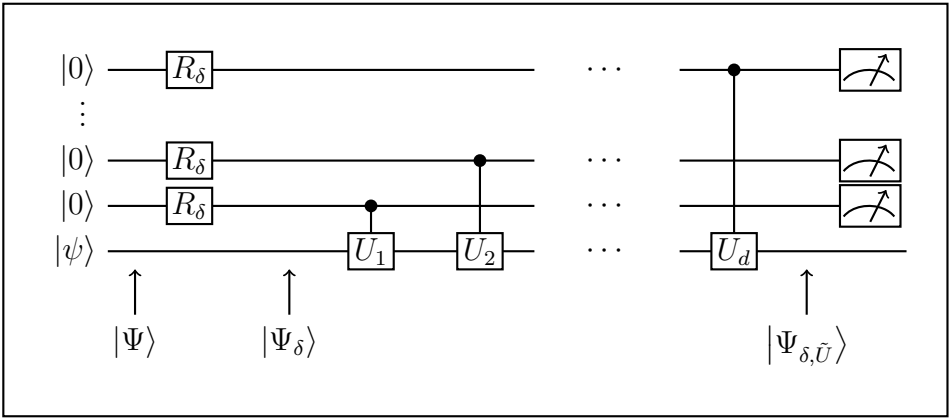
\begin{figure}[ht]


\tikzset{
    measurebox/.style={
        draw, rectangle, fill=white, minimum width=0.8cm, minimum height=0.55cm, inner sep=0pt
    }
}

\begin{tikzpicture}[x=1.2cm, y=-1.2cm, line width=0.8pt]


\node[anchor=east] at (0, 0) {$\ket{0}$};
\node[anchor=east] at (0, 1) {$\ket{0}$};
\node[anchor=east] at (0, 1.5) {$\ket{0}$};
\node[anchor=east] at (0, 2) {$\ket{\psi}$}; 

\node[anchor=east] at (-0.1, 0.4) {$\vdots$};

\draw (0,0) -- (4.7,0);
\draw[fill=white] (0.65, -0.2) rectangle (1.15, 0.2); \node at (0.9, 0) {$R_{\delta}$};
\node at (5.5,0) {$\cdots$}; 
\draw (6.3,0) -- (8.2,0);

\draw (0,1) -- (4.7,1); 
\draw[fill=white] (0.65, 0.8) rectangle (1.15, 1.2); \node at (0.9, 1) {$R_{\delta}$};
\node at (5.5,1) {$\cdots$}; 
\draw (6.3,1) -- (8.2,1);

\draw (0,1.5) -- (4.7,1.5); 
\draw[fill=white] (0.65, 1.3) rectangle (1.15, 1.7); \node at (0.9, 1.5) {$R_{\delta}$};
\node at (5.5,1.5) {$\cdots$}; 
\draw (6.3,1.5) -- (8.2,1.5);

\draw (0,2) -- (2.65,2);
\draw (3.15,2) -- (3.85,2);
\draw (4.35,2) -- (4.7,2);
\node at (5.5,2) {$\cdots$};
\draw (6.3,2) -- (8.8,2);

\draw[fill=white] (2.65, 1.8) rectangle (3.15, 2.2); \node at (2.9, 2) {$U_1$};
\draw[fill=white] (3.85, 1.8) rectangle (4.35, 2.2); \node at (4.1, 2) {$U_2$};
\draw[fill=white] (6.65, 1.8) rectangle (7.15, 2.2); \node at (6.9, 2) {$U_d$};

\draw[fill=black] (2.9, 1.5) circle (2pt);
\draw (2.9, 1.5) -- (2.9, 1.8);

\draw[fill=black] (4.1, 1) circle (2pt);
\draw (4.1, 1) -- (4.1, 1.8);

\draw[fill=black] (6.9, 0) circle (2pt);
\draw (6.9, 0) -- (6.9, 1.8);

\node[measurebox] (M1) at (8.4, 0) {};
\draw ([xshift=-10pt, yshift=-4pt]M1.center) arc (150:30:11pt and 8pt);
\draw[->] ([yshift=-6pt]M1.center) -- ([xshift=6pt, yshift=6pt]M1.center);

\node[measurebox] (M2) at (8.4, 1) {};
\draw ([xshift=-10pt, yshift=-4pt]M2.center) arc (150:30:11pt and 8pt);
\draw[->] ([yshift=-6pt]M2.center) -- ([xshift=6pt, yshift=6pt]M2.center);

\node[measurebox] (M3) at (8.4, 1.5) {};
\draw ([xshift=-10pt, yshift=-4pt]M3.center) arc (150:30:11pt and 8pt);
\draw[->] ([yshift=-6pt]M3.center) -- ([xshift=6pt, yshift=6pt]M3.center);

\draw[->] (0.3, 2.7) -- (0.3, 2.2);
\node[anchor=north] at (0.3, 2.7) {$\ket{\Psi}$};

\draw[->] (2, 2.7) -- (2, 2.2);
\node[anchor=north] at (2, 2.7) {$\ket{\Psi_{\delta}}$};

\draw[->] (7.7, 2.7) -- (7.7, 2.2);
\node[anchor=north] at (7.7, 2.7) {$\big|\Psi_{\delta, \tilde{U}}\big\rangle$};

\draw[black, thick] ([xshift=-15pt, yshift=-15pt]current bounding box.south west) rectangle ([xshift=15pt, yshift=15pt]current bounding box.north east);

\end{tikzpicture}
    \caption{Schematic representation of the $d$-dimensional Label Switch mechanism.}
    \label{alg::Quantum_Label_Switch}
\end{figure}

\begin{proposition}
\label{prop::physical_implementation}
    Let $\delta > 0$, $\ket{\psi} \in \mathcal{S}_{pure}(\mathbb{C}^d)$ and $\rho = \ket{\psi}\bra{\psi}$. For $\ket{\phi_0} = \ket{0} \otimes ... \otimes \ket{0}$ define $\sigma = \ket{\phi_0}\bra{\phi_0}$ on $(\mathbb{C}^2)^d$. Then there exists a unitary operation $\tilde U_{\delta}$, comprised of qubit rotations $R_{\delta}$ and qudit controlled NOT-gates $U_1,\ldots,U_d$, and a basis measurement $\tilde{M}_{\tilde{B}} =(\mathbbm{1} \otimes \ket{\omega}\bra{\omega})_{\omega \in 
    \Omega}$ on the second system alone such that for $\rho_{\delta,\tilde{U}} = \tilde U_{\delta}(\rho \otimes \sigma) \tilde U_{\delta}^*$ it holds
    \begin{equation*}
        \mathbb{P}_{\rho_{\delta,\tilde{U}}}\left( R^{\tilde{M}_{\tilde{B}}} = \omega \right) = \mathbb{P}_{\rho}\left( R^{M_{\delta}} = \omega \right) \hspace{10pt} \text{and} \hspace{10pt} \Tr_2\left[(\rho_{\delta,\tilde{U}})_{\tilde{M}_{\tilde{B}} \to \omega}\right] = \rho_{M_{\delta} \to \omega}
    \end{equation*}
    for all $\omega \in \Omega$, where $M_{\delta}$ is the measurement given by the measurement operators~\eqref{eqn::measurement_operators_gentle} and $\Tr_2$ denotes the partial trace over the second system.
\end{proposition}

\section{Gentle estimation of quantum states}
\label{sec::upper_bound_qudit_estimation}

\subsection{Types of quantum measurements}
\label{eqn::types_of_measurements}
The previous section gave an implementable gentle measurement procedure for the estimation of the probability vector corresponding to a pure quantum state. The main ideas that assure gentleness in that task can be applied to the more complex problem of quantum state tomography. However, as is the case with non-gentle tomography, measuring a state in a single basis in insufficient in order to obtain a full description of it. In the non-gentle case, the optimal measurements for state tomography and state certification are given by so-called $2$-designs. A (proper) 2-design is a finite set of $D$ normalized-vectors $(\ket{v_m})_{m = 1}^D$ such that
\begin{equation}
\label{eqn::2_design_definition}
    \frac{1}{D} \sum_{m = 1}^D (\ket{v_m}\bra{v_m})^{\otimes 2} = \int \ket{\psi}\bra{\psi}^{\otimes 2} d\mu(\psi)
\end{equation}
where $\mu$ is the Haar-measure on the unit sphere in $\mathbb{C}^d$. For a proper 2-design $(\ket{v_m})_{m = 1}^D$, the operators $(\sqrt{d/D} \ket{v_m}\bra{v_m})_{m = 1}^D$ define a quantum measurement that preserves the second moment of the Haar-measure \cite{ambainis_quantum_2007}. $2$-designs have been shown to define optimal measurement schemes for both state tomography \cite{guta2018faststatetomographyoptimal} and state certification \cite{liu_role_2024}. However, \cite{butucea_locally_2026} have shown that for optimal gentle state certification, we must consider a sub-class of $2$-design given by mutually unbiased bases (MUBs) as a direct gentle-ization of the $2$-design leads to a sub-optimal rate. A similar thing occurs in state tomography. Gentle-izing the optimal measurement from \cite{guta2018faststatetomographyoptimal} directly leads to the sub-optimal estimation rate of $d^4/n\alpha^2$. 

In order to achieve optimality, we must consider a complete set of MUBs. A set of $n_B$ MUBs is a collection $(|e_k^{(b)}\rangle)_{k = 1,...,d, b = 1,...,n_B}$, where, for fixed $b$, the vectors $(|e_k^{(b)}\rangle)_{k = 1,...,d}$ form an orthonormal basis of $\mathbb{C}^d$ and for $b \neq b'$ it holds $|\langle e_k^{(b)} | e_l^{(b')}\rangle|^2 = \frac{1}{d}$. For any given dimension $d$, there exist a maximum of $n_B = d+1$ MUBs and such a complete set of $d+1$ MUBs forms a $2$-design \cite{klappenecker_mutually_2005}. Instead of gentle-izing this $2$-design as a whole, the optimal gentle measurement is instead given by gentle-izing each basis individually and combining the gentle-ized basis measurements into one whole, which reduces the variance of the subsequent estimator by a factor of $d$ compared to gentle-izing the whole $2$-design.

\subsection{The qLS mechanism for quantum state tomography}
Let us now adapt the estimation procedure given in Section~\ref{sec::qLS_pure_states} to define a locally-gentle estimator for general $d$-dimensional states. Set $D = d(d+1)$ and let us denote the $d+1$ elements of the complete set of MUBs by
\begin{equation}
\label{eqn::definition_v_m}
    (\ket{v_m})_{m = 1,...,D} = \left( \ket{e_{k}^{(b)}} \right)_{_{k = 1,...,d}^{b = 1,...,d+1}}.
\end{equation}
Since maximal sets of mutually unbiased bases form a $2$-design, we have (see \cite{guta2018faststatetomographyoptimal})
\begin{align}
    \label{eqn::2design_property}
    \sum_{m = 1}^D \Tr\left[\ket{v_m}\bra{v_m} X \right] \ket{v_m}\bra{v_m} &= X + \Tr[X] \mathbbm{1}
    \\
    \label{eqn::ONB_property}
    \sum_{m = 1}^D \ket{v_m}\bra{v_m} &= (d+1) \mathbbm{1}
\end{align}
For $\delta > 0$ and for all $d+1$ bases, we define measurement operators $M_{\delta, \omega, b} = \sqrt{E_{\delta, \omega, b}}$ as in equation~\eqref{eqn::measurement_operators_gentle}. That is
\begin{align}
    \label{defn::measurement_operators_qls_mixed_MUB}
    E_{\delta, \omega, b} &= \frac{1}{d+1} \left( \frac{e^{\frac{\delta}{2}}}{e^{\frac{\delta}{2}} + 1} \right)^d \sum_{k  =1}^d e^{-\frac{\delta}{2} \norm{\omega - e_k}_1} \ket{e_k^{(b)}}\bra{e_k^{(b)}} \hspace{10pt} \text{for } \delta = 4 \arctanh(\alpha)
\end{align}
Note that $e_k$ is again the $k$-th standard $d$-dimensional basis vector, \emph{i.e.} the vector that is zero everywhere except in the $k$-th entry, while $|e_k^{(b)} \rangle$ is the $k$-th basis vector in the $b$-th mutually unbiased basis. Since the factor $1/(d+1)$ does not change the eigenvalue ratio of the operators, Lemma~\ref{lem::gentleness_of_gentleized_general_measurement} remains applicable, showing that the measurement $M_{\delta} = (M_{\delta, \omega, b})_{\omega \in \{0,1\}^d, b = 1,...,d+1}$ is $\alpha$-gentle for $\alpha = \tanh(\delta/4)$. The measurement~\eqref{defn::measurement_operators_qls_mixed_MUB} is equivalent to first choosing a basis $b \in \{1,...,d+1\}$ uniformly at random and then performing the measurement defined in~\eqref{eqn::measurement_operators_gentle} in that basis. Formally, the outcome of the measurement is a random tuple $R^{M_{\delta}} = (W, B)$. We transform the outcome into a $D$-dimensional vector $H$. 
Since every $m \in \{1,...,D\}$ may be written as $m = d  b + k$ for unique $b \in \{1,...,d+1\}$ and $k \in \{1,...,d\}$ we have
\begin{equation}\label{eqn::duality_2_desing_MUB}
    H_m = 1 \Longleftrightarrow B = b \text{ and } W_k = 1.
\end{equation}
Thus, the vector $H$ is itself composed of $d+1$ subvectors of length $d$. Each subvector corresponds to one of the possible bases we measure. The resulting $W$ we obtain as an output then replaces the correponding subvector. The remaining $d$ subvectors are $0$. 

By measuring repeatedly, $n$ times, independent copies of $\rho$, we obtain i.i.d. copies $H^{(i)}$ of $H$, for $i$ from 1 to $n$. The vector $\Bar{H} := (\Bar{H}_1,...,\Bar{H}_D) =  \frac{1}{n} \sum_{i = 1}^n H^{(i)}$ is the gentle estimator of the frequencies of possible outcomes $m$ or, equivalently, $(w,b)$. It is not a vector of frequencies as its components do not sum up to one. However, by compensating for this bias introduced by privacy, we are able to define an estimator
\begin{equation}
\label{eqn::qLS_est_before_proj}
    \bar{\rho}_n 
    = \frac{d+1}{\alpha} \sum_{m = 1}^D \Bar{H}_m \ket{v_m}\bra{v_m} - \left(  1 +  \frac{d+1}{\alpha}  \beta \right) \mathbbm{1}
\end{equation}
for $\beta = (e^{\frac{\delta}{2}} + 1)^{-1} \in (0,\frac{1}{2})$. Finally, we define the projection of this matrix on the convex set of quantum states
\begin{equation}
\label{eqn::qLS_est_qudits}
    \check{\rho}_n = \Pi_{\mathcal{S}(\mathbb{C}^d)} (\bar{\rho}_n)
\end{equation}
where $\Pi_{\mathcal{S}(\mathbb{C}^d)}$ denotes the projection onto $\mathcal{S}(\mathbb{C}^d)$. Note that the projection increases the distance to the true state by a factor 2 at most: : $\|\check{\rho}_n - \rho\|_{op} \leq 2 \cdot \|\hat \rho_n - \rho\|_{op}$ (Lemma~\ref{lem::operator_norm_properties_projection}). 

We can further process the estimator in order to assure that it is rank consistent. We do this by spectral thresholding similar to \cite{butucea_spectral_2015, lahiry_minimax_2024}. Let $\check{\rho}_n = \check{U}_n\check{D}_n\check{U}_n^*$ be the spectral decomposition of $\check{\rho}_n$. For $t > 0$, the spectral thresholding algorithm $\mathcal{A}$ (see Algorithm~\ref{alg::spectral_thresholding} in Appendix~\ref{sec::additional_calculations_upper_bound}) with threshold $2t$ then constructs an estimator $\hat{\rho}_n = \mathcal{A}(\check{\rho}_n)$, given by a valid density matrix, such that the minimal eigenvalue of $\hat{\rho}_n$ is bounded from below by $2t$. It assures that, if $\rho$ is low rank, so is the estimator $\hat{\rho}_n$. In fact, we can show that the rank of the estimator $\hat{\rho}_n$ is smaller than the rank of $\rho$ with high probability. We can then show that this estimator is minimax optimal. Furthermore, if we assume the lowest eigenvalue of the true state $\rho$ to be separated far enough from 0 we can show that the estimator $\hat{\rho}_n$ has the same rank as $\rho$ with high probability while incurring only a logarithmic factor in the estimation rate. We summarize this result in the following theorem, which we prove in Appendix~\ref{sec::additional_calculations_upper_bound}. 

\begin{theorem}
\label{thm::upper_bound_qLS_qudits}
    Let $d \geq 2$, $1 \leq r \leq d$ and $\rho \in \mathcal{S}_r(\mathbb{C}^d)$. For $\alpha \in (0, 1)$ set $\delta = 4 \arctanh(\alpha)$ and for $\epsilon \in (0,1]$ set
    \begin{equation*}
        t^2(\epsilon) = 92 \log(d/\epsilon) \frac{d^2}{n\alpha^2} 
    \end{equation*}
    Let $\hat{\rho}_n = \mathcal{A}(\check{\rho}_n)$ be the estimator described previously, where $\mathcal{A}$ is the spectral thresholding algorithm~\ref{alg::spectral_thresholding} with threshold $2t(\epsilon)$. Then $\hat{\rho}_n$ is a locally-$\alpha$-gentle estimator of $\rho$ such that and there exists a constant $C > 0$ independent of $r, d, n, \epsilon$ and $\alpha$ such that
    \begin{equation*}
        \sup_{\rho \in \mathcal{S}_r(\mathbb{C}^d)} \mathbb{E}_{\rho}\left[ \norm{\hat{\rho}_n - \rho}_F^2 \right] \leq C \frac{rd^2}{n \alpha^2} \log\left(\frac{d}{\epsilon}\right).
    \end{equation*}
    If we additionally assume $\lambda_{min}(\rho) \geq 6t(\epsilon)$, then
    \begin{equation*}
        \mathbb{P}_{\rho}\left( \rank(\hat{\rho}_n) = r \right) \geq 1 - \epsilon.
    \end{equation*}
\end{theorem}

\subsection{Quantum state tomography via mutually unbiased measurements}

A complete set of MUBs is only guaranteed to exist in prime power dimensions $d = p^q$. In non prime-power dimensions $d$ a generalization of MUBs is given by mutually unbiased measurements (MUMs, \cite{kalev_mutually_2014}). These are operators $(P_k^{(b)})_{k = 1,...,d, b = 1,...,n_B}$, together with an efficiency parameter $\kappa \in (1/d, 1]$, such that $\Tr \, [P_k^{(b)}] = 1$ and $\Tr \, [P_k^{(b')}P_{k'}^{(b')}] = \delta_{k, k'} \delta_{b,b'} \kappa + (1-\delta_{k,k'})\delta_{b,b'} \frac{1-\kappa}{d-1} + (1-\delta_{b,b'})\frac{1}{d}$. A complete set of $n_B = d+1$ MUMs can be constructed in any dimension from an ONB of the space of hermitian matrices. The efficiency parameter of the MUMs is dependent on the ONB chosed for construction. In the case $\kappa = 1$, the MUMs are MUBs with $P_k^{(b)} = |e_k^{(b)}\rangle\langle e_k^{(b)}|$. In that way we can view MUMs as an extension of MUBs and we can extend our locally gentle tomography mechanism to MUMs to be valid in any dimension. For a complete set of MUMs given by
\begin{equation}
\label{eqn::definition_v_m_MUM}
    (V_m)_{m = 1,...,D} = (P_k^{(b)})_{\genfrac{}{}{0pt}{}{k = 1,...,d}{b = 1,...,d+1}}
\end{equation}
we define measurement operators $M_{\delta, \omega, b} = \sqrt{E_{\delta, \omega, b}}$ similarly to~\eqref{defn::measurement_operators_qls_mixed_MUB} by
\begin{align}
    \label{defn::measurement_operators_qls_mixed_MUM}
    E_{\delta, \omega, b} &= \frac{1}{d+1} \left( \frac{e^{\frac{\delta}{2}}}{e^{\frac{\delta}{2}} + 1} \right)^d \sum_{k  =1}^d e^{-\frac{\delta}{2} \norm{\omega - e_k}_1} P_k^{(b)} \hspace{10pt} \text{for } \delta = 4 \arctanh(\alpha).
\end{align}
This measurement has the same outcome space as the one defined for MUBs. We can therefore define the vector $H$ as in~\eqref{eqn::duality_2_desing_MUB} and bases on on $n$ i.i.d. copies $H^{(i)}$ we again define $\bar{H} = \frac{1}{n }\sum_{i=1}^n H^{(i)}$ and the corresponding estimator
\begin{equation}
\label{eqn::qLS_est_before_proj_MUM}
    \bar{\rho}_n = \frac{d^2 - 1}{\alpha(\kappa d -1)} \sum_{m = 1}^D \bar{H}_m V_m - \frac{d^2 -1}{\alpha(\kappa d -1)} \left( \frac{\alpha(d-\kappa)}{d^2 - 1} + \beta \right) \mathbbm{1}.
\end{equation}
Projecting $\bar{\rho}$ onto the space of density matrices $\check{\rho}_n = \prod_{\mathcal{S}(\mathbb{C}^d)}(\bar{\rho}_n)$ and applying the spectral thresholding algorithm gives the state estimator $\hat{\rho}_n = \mathcal{A}(\check{\rho}_n)$ for which we have the following result, which we proof in Appendix~\ref{sec::proofs_upper_bound_MUMs}.
\begin{proposition}
\label{prop::upper_bound_qLS_qudits_MUM}
    Let $d \geq 2$, $1 \leq r \leq d$ and $\rho$ a rank $r$ quantum state on $\mathbb{C}^d$. For $\alpha \in (0,1)$ set $\delta = 4 \arctanh(\alpha)$ and for $\epsilon \in (0, 1]$ set
    \begin{equation*}
        t^2(\epsilon) = 92 \log(d/\epsilon) \frac{d^2}{\kappa n \alpha}.
    \end{equation*}
    Let $\bar{\rho}_n$ be as in~\eqref{eqn::qLS_est_before_proj_MUM} for a set of MUMs with efficiency parameter $\kappa$ and $\hat{\rho}_n$ the state estimator based on $\bar{\rho}_n$. Then $\hat{\rho}_n$ is a locally-$\alpha$-gentle estimator of $\rho$ such that there exists a constant $C > 0$ independent of $r, d, n, \epsilon, \kappa$ and $\alpha$ such that
    \begin{equation*}
        \sup_{\rho \in \mathcal{S}_r(\mathbb{C}^d)} \mathbb{E}_{\rho}\left[ \norm{\hat{\rho}_n - \rho}_F^2 \right]  \leq C \frac{rd^2}{n\alpha^2\kappa^2} \log\left( \frac{d}{\epsilon} \right).
    \end{equation*}
    If we additionally assume $\lambda_{min}(\rho) \geq 6t(\epsilon)$, then
    \begin{equation*}
        \mathbb{P}_{\rho}\left( \rank(\hat{\rho}_n = r \right) \geq 1 - \epsilon.
    \end{equation*}
\end{proposition}
We can see that for MUMs with efficiency parameter $\kappa = 1$, both the estimator and its rate are identical to their counterparts with MUBs. Furthermore, as long as we can construct a complete set of MUMs with efficiency parameter sufficiently large, \emph{i.e.} $\kappa \sim 1$, the optimal rate from Theorem~\ref{thm::upper_bound_qLS_qudits} can be generalized to arbitrary dimension.

\section{Lower bounds for gentle estimation of qudits}
\label{sec::lower_bounds_qudits}

In this section we will prove lower bounds that match the rate obtained in our upper bounds in Theorem~\ref{thm::upper_bound_qLS_qudits} up to a factor of $\log(d)$. 
\begin{theorem}
    \label{thm::lower_bound_qudit_estimation}
    Let $\alpha > 0$ be bounded away from $1/2$ by some constant and $r \leq d \in \mathbb{N}$. Then there exists a constant $c > 0$ independent of $n, \alpha, d, r$ such that it holds
    \begin{equation*}
        \inf_{(M, \hat{\rho})} \sup_{\rho \in \mathcal{S}_r(\mathbb{C}^d)} \mathbb{E}_{\rho} \left[ \norm{\hat{\rho} - \rho}_F^2 \right] \geq c \min\left\{ \varphi(d, r),  \frac{d^2r}{n\alpha^2} \right\}
    \end{equation*}
    for some non-increasing function $\varphi$, which we further specify below. Here the infimum is taken over all locally-$\alpha$-gentle meausrements $M$ and subsequent estimators $\hat{\rho}$.
\end{theorem}
To prove Theorem~\ref{thm::lower_bound_qudit_estimation} we give a general scheme based on Assouad's method where we reduce the estimation problem to a testing problem in a high dimensional hypercube. Afterwards we will construct appropriate quantum states that serve as hypothesis states for the aforementioned testing problem. We give constructions for general full rank states with $\varphi(r,d) = 1/(4d^2)$, for low rank quantum states with $\varphi(r,d) = 1/(4d^3r^5)$ and for pure states where we improve $\varphi(r,d) =  1/8$. 

\subsection{General Scheme}
\label{sec::reduction_scheme}
The idea for proving lower bounds is to construct states $\mathcal{F} = \{\rho_{\nu}\}_{\nu \in \mathcal{V}}$ with $\mathcal{V} = \{-1, 1\}^D$ for an appropriately chosen $D$. The parameter $D$ characterizes the hardness of the problem and is closely related to the amount of parameters describing the state space. Since the sets of pure and mixed qudits are manifolds in dimension $2(d-1)$ and $d^2 - 1$ respectively, we will chose $D$ depending on the problem at hand. 

\begin{definition}
\label{defn::mixed_distribution states}
    Let $D \in \mathbb{N}$, $\mathcal{V} = \{-1, 1\}^D$ and $\mathcal{F} = \{\rho_{\nu}\}_{\nu \in \mathcal{V}}$ be a set of $d$-dimensional quantum states. When given given $n$ copies $(\rho_{\nu, i})_{i = 1}^n$ and $n$ independent measurements (one on each register) $(M_i)_{i = 1}^n$ where each $M_i$ is given by measurements operators $M_{{\omega}_i,i}$, with $E_{\omega_i, i} = M_{{\omega}_i,i}^* M_{{\omega}_i,i}$ we define.
    \begin{alignat*}{3}
        \rho_{\pm j,i} &= \frac{1}{2^{D-1}} \sum_{\nu: \nu_j = \pm1} \rho_{\nu,i}, \hspace{50pt} && \mathbb{P}_{\pm j, i}^{R^M}(\{{\omega}_i\}) &&= \Tr[\rho_{\pm j,i} E_{\omega,i}]
        \\
        \rho_{\nu}^n &= \bigotimes_{i = 1}^n \rho_{\nu, i}, && \mathbb{P}_{\nu}^{R^M, n}(\{\omega\}) && = \Tr[\rho_{\nu}^n \bigotimes_{i = 1}^n E_{\omega_i,i}]
        \\
        \rho_{\pm j}^n &= \frac{1}{2^{D-1}} \sum_{\nu: \nu_j = \pm 1} \rho_{\nu}^n, &&\mathbb{P}_{\pm j}^{R^M, n}(\{\omega\}) &&= \Tr[\rho_{\pm j}^n \bigotimes_{i = 1}^n E_{\omega_i,i}]
    \end{alignat*}
    and $\omega = (\omega_1,...,\omega_n)$. Remember that the outcome distribution of $\rho_{\nu, i}$ is given by $\mathbb{P}_{\nu,i}^{R^M}$.
\end{definition}

Note that by definition for every $j \in {1,...,D}$ we have
\begin{align*}
    \frac{1}{2^{D-1}} \sum_{\nu: \nu = \pm j} \mathbb{P}_{\nu}^{R^M,n}(\{\omega\}) &= \frac{1}{2^{D-1}} \sum_{\nu: \nu = \pm j} \Tr[\rho_{\nu}^n \bigotimes_{i = 1}^n E_{{\omega}_i,i}] 
    = \mathbb{P}_{\pm j}^{R^M, n}(\{\omega\}).
\end{align*}
As such, we see that we have the following equalities between the two 
\begin{equation}
\label{eqn::distribution_equalities}
    \frac{1}{2^{D-1}} \sum_{\nu: \nu = \pm j} \mathbb{P}_{\nu}^{R^M,n} =  \mathbb{P}_{\pm j}^{R^M, n} \hspace{20pt} \text{and} \hspace{20pt} \mathbb{P}_{\pm j, i}^{R^M} = \frac{1}{2^{D-1}} \sum_{\nu: \nu = \pm 1} \mathbb{P}_{\nu, i}^{R^M}.
\end{equation}
Our first step in lower bounding the estimation risk is by showing that the set of hypothesis states $\mathcal{F} = \{\rho_{\nu}\}_{\nu \in \mathcal{V}}$ induces a $2\eta$-Hamming separation for our model.

\begin{definition}
\label{defn::Hamming_separation}
    We say that a collection $\mathcal{F} = \{\rho_{v}\}_{v \in \mathcal{V}}$ with $\mathcal{V} \subseteq \{-1,1\}^D$ induces a $2\eta$-Hamming separation for $\norm{\cdot}^2$ on $\mathcal{S} \subseteq \mathcal{S}(\mathbb{C}^d)$ if there exists a function $\mathbf{v}: \mathcal{S} \to \{-1, 1\}^D$ such that
    \begin{equation*}
        \norm{\rho - \rho_{\nu}}^2 \geq 2 \eta \sum_{j = 1}^D \mathbbm{1}_{\{\mathbf{v}(\rho)_j \neq v_j\}}.
    \end{equation*}
\end{definition}
A collection $\mathcal{F}$ that induces a $2 \eta$-Hamming separation on $\mathcal{S}$ separates our model in such a way that the problem of estimating a state reduces to the problem of testing between the states of the collection $\mathcal{F}$.
\begin{lemma}
\label{lem::C1}
    If $\mathcal{F} = \{\rho_v\}_{v \in \mathcal{V}}$ induces a $2\eta$-Hamming separation for $\norm{\cdot}^2$ it holds
    \begin{equation}
    \label{eqn::lemmaC_1}
        \inf_{(M, \hat{\rho})} \max_{\nu \in \mathcal{V}} \mathbb{E}_{\nu}\left[ \norm{\hat{\rho} - \rho_{\nu}}^2\right]
        \geq \eta \sum_{j = 1}^D \left(1 -  \norm{\mathbb{P}_{+j}^{R^M,n} - \mathbb{P}_{-j}^{R^M,n}}_{TV} \right).
    \end{equation}
    where the infimum on the left is taken over all measurements $M = M_1 \otimes ... \otimes M_n$ and subsequent estimators $\hat{\rho}$.
\end{lemma}
The proof of Lemma~\ref{lem::C1} is given in Appendix~\ref{sec::proofs_for_lower_bounds}. In order to further bound the expression~\eqref{eqn::lemmaC_1} from below, we relate the sum of total-variation distances of the distributions $\left(\mathbb{P}_{\pm j}^{R^M,n}\right)_{j = 1,...,D}$ to their symmetrized Kullback-Leibler divergences in the following way using a combination of Pinsker's and the Cauchy-Schwarz inequality
\begin{equation*}
        \sum_{j = 1}^D \norm{\mathbb{P}_{+j}^{R^M,n} - \mathbb{P}_{-j}^{R^M,n}}_{TV} \leq \frac{1}{2} \sqrt{D} \left( \sum_{j = 1}^D  D_{KL}^{sym}\left(\mathbb{P}_{+j}^{R^M,n} \middle| \middle| \mathbb{P}_{-j}^{R^M,n} \right) \right)^{\frac{1}{2}}.
    \end{equation*}
The last step of our reduction scheme is given by the main result of this section. A quantum data-processing inequality bounding the symmetrized Kullback-Leibler divergences of the outcome distributions of a gentle measurements. Its proof is given in Appendix~\ref{sec::proofs_for_lower_bounds}.
\begin{theorem}
\label{thm::C3}
Let $\alpha \in [0, \frac{1}{2})$. For the distributions defined in Definition~\ref{defn::mixed_distribution states}, assuming $\rho_{\nu, i} = \rho_{\nu}$ for $i = 1,...,n$, and for any locally-$\alpha$-gentle measurement it holds
\begin{equation*}
    \sum_{j = 1}^D D_{KL}^{sym}\left(\mathbb{P}_{+j}^{R^M,n} \middle| \middle| \mathbb{P}_{-j}^{R^M,n} \right) \leq \frac{64 n \alpha^2}{(1 - 2\alpha)^4} \sup_{\norm{A}_{op} \leq 1}\sum_{j = 1}^D \frac{1}{2^{D-1}} \hspace*{-5pt} \sum_{\nu \in \{-1,1\}^{D-1}} \hspace*{-5pt} \Tr \left[ A (\rho_{\nu_{j+}} - \rho_{\nu_{j-}})\right]^2.
\end{equation*}
Here, for fixed $\nu \in \{-1,1\}^{D-1}$, $\rho_{\nu_{j+}}, \rho_{\nu_{j-}}$ denote the states $\rho_{\nu, i}$ where the parameter vector is padded such that the $j-th$ component is either $+1$ or $-1$. Note that $\rho_{\nu_{j\pm}} \neq \rho_{\pm j}$.
\end{theorem}
Although Theorem~\ref{thm::C3} initially does not seem to simplify the problem, 
it is the crucial last step in the lower bound scheme as it decouples the distributions appearing in the KL-divergence. Note that by the variational formulation of the trace-distance \cite{audenaert_asymptotic_2008}, we have
\begin{equation*}
     \sup_{\norm{A}_{op} \leq 1} \Tr \left[ A (\rho_{\nu_{j+}} - \rho_{\nu_{j-}})\right] = \norm{\rho_{\nu_{j+}} - \rho_{\nu_{j-}}}_{Tr}.
\end{equation*}
The term in Theorem~\ref{thm::C3} can therefore always be bounded by the sum of individual testing errors of the hypotheses states. In that sense, it serves as a generalization of the quantum data-processing inequality for gentle measurements \cite{butucea_sample-optimal_2025}. It's strength however, lies in the simultaneous minimization of testing errors along the vertices of the $D$-dimensional hypercube. This leads to a much sharper bound on the KL-divergence. Altogether, for a collection $\mathcal{F}$ that induces a $2\eta$-Hamming separation for $\norm{\cdot}$, we have
\begin{align*}
        &\inf_{(M, \hat{\rho})} \max_{\rho \in \mathcal{S}} \mathbb{E}_{\nu}\left[ \norm{\hat{\rho} - \rho}^2\right]
        \\
        \geq& \eta \left( D - \frac{1}{2} \sqrt{D} \left( \frac{64 n \alpha^2}{(1 - 2\alpha)^4} \sup_{\norm{A}_{op} \leq 1} \sum_{j = 1}^D \frac{1}{2^{D-1}} \sum_{\nu \in \{-1,1\}^{D-1}} \Tr \left[ A (\rho_{\nu_{j+}} - \rho_{\nu_{j-},})\right]^2 \right)^{\frac{1}{2}} \right).
\end{align*}
In order to conclude, we choose the largest value of $\eta >0$ (as function of $n,\,\alpha,\,r$ and $d$) such that the quantity under brackets above remains bounded away from 0.
\subsection{Construction of binary hypothesis states}
\label{sec::lower_bound_constructions}
Let us demonstrate the power of our reduction scheme by applying it to several different setups. In particular, we show the optimality of the estimator $\hat{\rho}_n$ defined in~\eqref{eqn::qLS_est_qudits} in the case where the state $\rho$ to be estimated is pure, low-rank or full-rank. Combining these results show that the minimax rate for locally-$\alpha$-gentle quantum state tomography is $\frac{d^2r}{n \alpha^2}$ up to log factors. We end this section by demonstrating that the probability vector estimator $\hat{p}_n$ defined in~\eqref{eqn::estimator_probability_vector} is minimax optimal as well.

\subsubsection{Full rank states}
\label{sec::lower_bound_construction_full_rank}
We start our demonstrating with the case where the state we want to estimate is full rank, that is $r = d$. Let $(V_j)_{j = 1}^{d^2}$ be a basis of the real vector space of hermitian $d \times d$ matrices such that $V_{d^2} = \frac{1}{\sqrt{d}} \mathbbm{1}$. Note that, as $V_{d^2}$ is a multiple of the identity matrix, the matrices $V_j$ are traceless for $j = 1,...,d^2-1$. An example of such a basis are the generalized Gell-Mann-matrices which generalize the Pauli-matrices to the $d$-dimensional case. We can parametrize any quantum state $\rho$ as
\begin{equation}
\label{eqn::GGM_basis_rep}
    \rho = \rho_{mm} + \sum_{j = 1}^{d^2 -1} b_j V_j, \hspace{10pt} \text{for some } b_j \in \mathbb{R},
\end{equation}
where $\rho_{mm} = \frac{1}{d} \mathbbm{1}$ is the maximally mixed state. Now, let $D \in \{1,...,d^2 -1 \}$ and set $\mathcal{V} = \{-1,1\}^D$. We define the set of local binary hypotheses $\mathcal{F} = \{\rho_{\nu}\}_{\nu \in \mathcal{V}}$ by
\begin{equation}
\label{eqn::rho_nu_definition}
    \rho_{\nu} = \rho_{mm} + \Delta_{\nu} = \rho_{mm} + \frac{\epsilon}{\sqrt{dD}} \sum_{i = 1}^D \nu_i V_i.
\end{equation}
By construction, for any $\nu$, $\rho_{\nu}$ is hermitian with trace 1. Furthermore, $\rho_{\nu}$ is positive definite as long as the maximal eigenvalue of $\Delta_{\nu}$ is less than $\frac{1}{d}$, which is true as long as $\epsilon \leq \frac{1}{\sqrt{d}}$. In order to now apply our reduction scheme, we need to assure that the local hypotheses states~\eqref{eqn::rho_nu_definition} are far apart, i.e. induce a suitable Hamming distance. This is the content of the following lemma.
\begin{lemma}
\label{lem::Hamming_distance_full_rank}
    Let $D \in \{1,...,d^2 -1 \}$ and $\mathcal{V} = \{-1,1\}^D$. The set of local binary hypotheses $\mathcal{F} = \{\rho_{\nu}\}_{\nu \in \mathcal{V}}$ defined by~\eqref{eqn::rho_nu_definition} induces a $2\eta$-Hamming separation for $\norm{\cdot}_F^2$ on $\mathcal{S} = \mathcal{S}(\mathbb{C}^d)$ where $\eta = \frac{\epsilon^2}{2dD}$.
\end{lemma}
\begin{proof}
    Using the definitions (\ref{eqn::GGM_basis_rep}) and (\ref{eqn::rho_nu_definition}), together with the fact, that the $V_i$ are orthonormal with respect to the Hilbert-Schmidt inner product, we get
    \begin{align*}
        \norm{\rho - \rho_{\nu}}_{F}^2 &= \norm{\sum_{j = 1}^{d^2-1} b_j V_j - \frac{\epsilon}{\sqrt{dD}} \sum_{j = 1}^D \nu_j V_j }_{F}^2 \hspace*{-3pt} \geq \sum_{j = 1}^D \left( b_j - \frac{\epsilon}{\sqrt{dD}} \nu_j \right)^2 \hspace*{-3pt} \geq \frac{\epsilon^2}{dD} \sum_{j = 1}^D \mathbbm{1}_{\{\mathbf{v}(\rho)_j \neq \nu_j \}},
    \end{align*}
    where $\mathbf{v}(\rho)_j = \sgn(b_j)$ for the parametrization~\eqref{eqn::GGM_basis_rep} of $\rho$.
\end{proof}

The next step in the reduction scheme involves showing that while the states are suitably separated, their outcome distributions after a locally-gentle measurement remain close, which is the content of the following lemma.

\begin{lemma}
\label{lem::upper_bound_distances_mixed_states}
    Let $\epsilon \leq \frac{1}{\sqrt{d}}$, $D \leq d^2 - 1$ and $(\rho_{\nu})_{\nu \in \mathcal{V}}$ as in~\eqref{eqn::rho_nu_definition}. Then, for the states defined in Definition~\ref{defn::mixed_distribution states} it holds
    \begin{equation*}
        \sup_{\norm{A}_{op} \leq 1} \sum_{j = 1}^D \frac{1}{2^{D-1}} \sum_{\nu \in \{-1,1\}^{D-1}} \Tr \left[ A (\rho_{\nu_{j+},i} - \rho_{\nu_{j-}, i})\right]^2 \leq \frac{4\epsilon^2}{D}
    \end{equation*}
\end{lemma}
\begin{proof}
    By definition, for a fixed $\nu \in \mathcal{V}$, we have 
    \begin{equation*}
        \rho_{\nu_{j+}} - \rho_{\nu_{j-}} = \frac{2\epsilon}{\sqrt{d D}} V_j.
    \end{equation*}
    This shows, that the term $\rho_{\nu_{j+}} - \rho_{\nu_{j-}}$ is independent of the actual $\nu$, showing 
    \begin{equation*}
        \sum_{j = 1}^D \frac{1}{2^{D-1}} \sum_{\nu \in \{-1,1\}^{D-1}} \Tr \left[ A (\rho_{\nu_{j+}} - \rho_{\nu_{j-}})\right]^2 = \frac{4 \epsilon^2}{dD} \sum_{j = 1}^D \Tr \left[ A V_j \right]^2
    \end{equation*}
    Since the $V_i$ form an orthonormal basis of the hermitian matrices with respect to the trace inner product, using Bessel's inequality, we have 
    \begin{equation*}
        \sum_{j = 1}^D \Tr^2 \left( A V_j \right) \leq \norm{A}_{F}^2 \leq d
    \end{equation*} 
    for any positive semi-definite $A$ such that $\norm{A}_{op} \leq 1$ which gives the result.
\end{proof}

We finalize the derivation of this lower bound by combining the above results such that we have
\begin{align*}
    \inf_{(M, \hat{\rho})} \max_{\rho \in \mathcal{S}(\mathbb{C}^d)} \mathbb{E}_{\rho}\left[ \norm{\hat{\rho} - \rho}^2\right] &\geq \frac{\epsilon^2}{2dD} \left( D - \frac{1}{2} \sqrt{D} \left( \frac{256\alpha^2 n d \epsilon^2}{(1-2\alpha)^4 d D} \right)^{\frac{1}{2}} \right).
\end{align*}
Assume first that $\frac{(1-2\alpha)^2}{16} \frac{d^2}{\alpha \sqrt{n}} \leq \frac{1}{\sqrt{d}}$. Then we can chose $D = d^2-1$ and $\epsilon = \frac{(1-2\alpha)^2}{16} \frac{D}{\alpha \sqrt{n}} \leq \frac{1}{\sqrt{d}}$. In that case we get
\begin{align*}
    \inf_{(M, \hat{\rho})} \max_{\rho \in \mathcal{S}(\mathbb{C}^d)} \mathbb{E}_{\rho}\left[ \norm{\hat{\rho} - \rho}^2\right] &\geq \frac{(1-2\alpha)^4}{1024} \frac{D^2}{dn \alpha^2} \geq \frac{(1-2\alpha)^4}{1024} \frac{d^3}{n \alpha^2}.
    \intertext{
    If $\frac{(1-2\alpha)^2}{16} \frac{d^2}{\alpha \sqrt{n}} > \frac{1}{\sqrt{d}}$ we chose a $D = \floor{\frac{\alpha \sqrt{n}}{\sqrt{d}} \frac{16}{(1-2\alpha)^2}}$. Then $\epsilon =  \frac{(1-2\alpha)^2}{16} \frac{D}{\alpha \sqrt{n}} \leq \frac{1}{\sqrt{d}}$ which gives}
    \inf_{(M, \hat{\rho})} \max_{\rho \in \mathcal{S}(\mathbb{C}^d)} \mathbb{E}_{\rho}\left[ \norm{\hat{\rho} - \rho}^2\right] &\geq \frac{1}{4 d^2}
\end{align*}
finally giving the lower bound
\begin{equation*}
    \inf_{(M, \hat{\rho})} \max_{\rho \in \mathcal{S}(\mathbb{C}^d)} \mathbb{E}_{\rho}\left[ \norm{\hat{\rho} - \rho}^2\right] \geq \min\left\{ \frac{1}{4d^2}, \frac{(1-2\alpha)^4}{256} \frac{d^3}{n \alpha^2} \right\}.
\end{equation*}

\begin{remark}
\label{rem::High_rank_lower_bound}
    We can adapt this construction to states of rank $r$ in the following way. Assume that the first $r^2-1$ matrices $V_1,...,V_{r^2-1}$ are non-zero only in the upper left sub-matrix. Such a construction is always possible. For example by taking the generalized Gell-Mann-matrices for dimension $r$ and adding $d^2-r^2$ matrices to build a basis of the $d$-dimensional hermitian matrices. Following the same construction as above would then give a rate of $\frac{dr^2}{n \alpha^2}$, which is suboptimal in general. However, for the case of $r \asymp d$ this rate is equivalent to $\frac{d^2r}{n \alpha^2}$. This extension of the lower bound construction~\eqref{eqn::GGM_basis_rep} therefore leads to the optimal lower bound of $\frac{d^2r}{n \alpha^2}$ as long as, say, $r \geq d/2$. To complement this result, for $r < d/2$, we use a different construction which we describe in the following section. 
\end{remark}

\subsubsection{Low-rank states}
\label{sec::lower_bound_construction_low_rank} 
Let us now construct a suitable set of local binary hypotheses for low rank quantum state tomography. The main problem for low-rank states arises from the fact that the space of rank $r$ quantum states has no natural vector space structure which excludes a construction similar to those for full rank states. As we are dealing with low-rank state, \emph{i.e.} $r < d/2$, an adaptation as in Remark~\ref{rem::High_rank_lower_bound} would lead to a sub-optimal lower bound of $\frac{dr^2}{n \alpha^2}$. Therefore, rather than perturbing a central state linearly, we construct our local binary hypothesis states by permuting it unitarily which does not change the rank. A similar construction has already been used by \cite{lahiry_minimax_2024} to show local asymptotic normality of low-rank quantum states to a quantum-classical ensemble of Gaussian distributions and shifted pure and thermal quantum states. As working with such unitary perturbations is more difficult than working with linear perturbations, we start the proof of our lower bounds by showing that the unitary perturbations are almost linear in the close vicinity of the maximally mixed state $\rho_{mm}^{(r)}$ of rank $r$. Formally, we make use of the Lie group structure of $SU(d)$ and its corresponding Lie algebra $\mathfrak{su}(d)$ in order to construct linear perturbations in the Lie algebra and track these perturbations on the states through the exponential map. The Lie algebra $\mathfrak{su}(d)$ is given by the traceless $d \times d$ skew-hermitian matrices. Note that this is the definition typically used in mathematics, in contrast to the definition often used in physics where $\mathfrak{su}(d)$ is the space of traceless hermitian matrices. We can define a basis of $\mathfrak{su}(d)$ using the following matrices.
\begin{alignat*}{2}
    H_l &= i \ket{l}\bra{l} - i \ket{l+1}\bra{l+1} \quad \quad \quad && 1 \leq l \leq d-1
    \\
    T_{k,l} &= \ket{l}\bra{k} - \ket{k}\bra{l} && 1 \leq l < k \leq d
    \\
    T_{l,k} &= i\ket{l}\bra{k} - i\ket{k}\bra{l} && 1 \leq l < k \leq d.
\end{alignat*}
For $r < d$, set $D = (d-r)r$ and $\mathcal{V} = \{-1, 1\}^D$. For any $\nu \in \mathcal{V}$ and $\epsilon > 0$, we define 
\begin{equation}
\label{eqn::S_nu_definition}
    S_{\nu} = \sum_{j = 1}^D \nu_j T_j = \sum_{l = 1}^r \sum_{k = r + 1}^d \nu_{k, l} T_{k,l} = \begin{bmatrix}
        0 & V_{\nu}
        \\
        -V_{\nu}^t & 0
    \end{bmatrix}\hspace{20pt} \text{and} \hspace{20pt} U_{\nu} = \exp(\epsilon S_{\nu}),
\end{equation}
for $V_{\nu} = \sum_{l = 1}^r \sum_{k = r+1}^d \nu_{k,l} \ket{l}\bra{k} = \sum_{j = 1}^D \nu_j T_j$, where we slightly abuse notation by identifying an index $j \in D$ with a pair $(k, l) \in \{r+1,...,d\} \times \{1,...,r\}$. Since $S_{\nu} \in \mathfrak{su}(d)$ we have $U_{\nu} \in SU(d)$. Therefore the matrix
\begin{equation}
\label{eqn::rho_nu_low_rank}
    \rho_{\nu} = U_{\nu} \rho_{mm}^{(r)} U_{\nu}^* = U_{\nu} \frac{1}{r} \mathbbm{1}_r U_{\nu}^*, \hspace{20pt} \text{for } \rho_{mm}^{(r)} := \frac{1}{r} \mathbbm{1}_r = \frac{1}{r} \diag(1,...,1,0,...,0)
\end{equation}
does in fact define a quantum state of rank $r$. In order to express $\rho_{\nu}$ differently we make use of the Baker-Campbell-Hausdorff-formula (\cite{hall_lie_2015}, Proposition 3.35) to calculate
\begin{equation}
\label{eqn::rho_nu_low_rank_expansion}
    \rho_{\nu} = \rho_{mm}^{(r)} +  \frac{1}{r}\sum_{m = 1}^{\infty} \frac{\epsilon^m}{m!} ad_{S_{\nu}}^m(\mathbbm{1}_r) = \rho_{mm}^{(r)} + \Delta_{\nu}(\epsilon)
\end{equation}
where $ad_{A}(B) = [A, B] = AB - BA$ denotes the commutator of the matrices $A$ and $B$ and $ad_{A}^m(B) = \left[ A, \left[ ... ,[A, B] \right] \right]$. Equation~\eqref{eqn::rho_nu_low_rank_expansion} is a non-commutative Taylor expansion of the exponential map mapping the Lie-algebra $\mathfrak{su}(d)$ onto the Lie-group $SU(d)$. By analyzing the structure of our local alternative states, we can show that the perturbations are almost linear and are only along the off-diagonal elements. Formally, in Lemma~\ref{lem::calculation_Delta_nu_epsilon}, we show that for small $\epsilon$, $\Delta_{\nu}(\epsilon)$ can be written as
\begin{equation}
\label{eqn::Unitary_approximation}
    \Delta_{\nu}(\epsilon) = O(\epsilon) \begin{bmatrix}
        O(\epsilon) & V_{\nu}
        \\
        V_{\nu}^t & O(\epsilon)
    \end{bmatrix} = O(\epsilon) \begin{bmatrix}
        0 & V_{\nu}
        \\
        V_{\nu}^t & 0
    \end{bmatrix} + O(\epsilon^2).
\end{equation}
We then use this result to show the following separation property of the states $\rho_{\nu}$.
\begin{lemma}
\label{lem::Hamming_distance_low_rank}
    Let $D \in \{1,...,d^2 -1 \}$ and $\mathcal{V} = \{-1,1\}^D$ and assume that $\epsilon \leq \sqrt{\frac{3}{8}} \frac{1}{r \sqrt{d-r}}$. The set of local binary hypotheses $\mathcal{F} = \{\rho_{\nu}\}_{\nu \in \mathcal{V}}$ defined by~\eqref{eqn::rho_nu_low_rank} induces a $2\eta$-Hamming separation for $\norm{\cdot}_F^2$ on $\mathcal{S} = \mathcal{S}_r(\mathbb{C}^d)$ where $\eta = \frac{\epsilon^2}{r^2}$
\end{lemma}
The following Lemma gives an upper bound on the distance between the distributions of the measurement outcomes.
\begin{lemma}
\label{lem::upper_bound_distances_rank_r_states}
    Let $\epsilon \leq \frac{1}{d^2r^2}$, $D = r(d-r)$ and $(\rho_{\nu})_{\nu \in \mathcal{V}}$ as in~\eqref{eqn::rho_nu_low_rank}. Then, for the states defined in Definition~\ref{defn::mixed_distribution states} it holds
    \begin{equation*}
        \sup_{\norm{A}_{op} \leq 1} \sum_{j = 1}^D \frac{1}{2^{D-1}} \sum_{\nu \in \{-1,1\}^{D-1}} \Tr \left[ A (\rho_{\nu_{j+}} - \rho_{\nu_{j-}})\right]^2 \leq 72 \frac{\epsilon^2}{r}.
    \end{equation*}
\end{lemma}

This result now allows us to utilize our reduction scheme to obtain
\begin{align*}
    \inf_{(M, \hat{\rho})} \max_{\rho \in \mathcal{S}_r(\mathbb{C}^d)} \mathbb{E}_{\rho}\left[ \norm{\hat{\rho} - \rho}_F^2\right] \geq& \frac{\epsilon^2}{r^2} \left( r(d-r) - \frac{1}{2} \sqrt{r(d-r)} \left( \frac{64 n\alpha^2}{(1 - 2\alpha)^4} \frac{72 \epsilon^2}{r} \right)^{\frac{1}{2}} \right).
\end{align*}
We chose
\begin{align*}
    \epsilon = \min\left\{ \frac{1}{d^2 r^2} , \frac{(1-2\alpha)^2}{\sqrt{4608}} \frac{r\sqrt{d-r}}{\alpha \sqrt{n}}\right\}
\end{align*}
to get
\begin{align*}
    \inf_{(M, \hat{\rho})} \max_{\rho \in \mathcal{S}_r(\mathbb{C}^d)} \mathbb{E}_{\rho}\left[ \norm{\hat{\rho} - \rho}_F^2\right] 
    \geq \min\left\{ \frac{1}{4 d^3 r^5} ,\frac{(1-2\alpha)^4}{36864} \frac{r d^2}{n \alpha^2} \right\}
\end{align*}
as long as $r \leq d/2$.

\subsubsection{Pure states}
\label{sec::lower_bound_construction_pure_states}
Although the lower bound for pure states is included in the low rank case for $r = 1$, we will give a short description of how to construct a suitable set of local pure hypothesis states that allows for the application within our reduction scheme. We not only give an explicit construction for pure states, but also improve on the second order term of the rate from $1/d^3r^5$ to $1/4$. Let us assume without loss of generality that $d$ is even. In the case of odd $d$, we omit one dimension from our considerations. For an orthonormal basis $(\ket{k})_{k \in \{1,...,d\}}$, consider the pure states $\left(\rho_{\nu}\right)_{\nu \in \mathcal{V}} = \left(\ket{\psi_{\nu}}\bra{\psi_{\nu}}\right)_{\nu \in \mathcal{V}}$ for $\mathcal{V} = \{-1, 1\}^{D}$ and $D = \frac{d}{2}$, where
\begin{equation}
\label{eqn::rho_nu_pure}
    \ket{\psi_{\nu}} = \sum_{k = 1}^D \frac{1 + \epsilon \nu_k}{\sqrt{d(1+\epsilon^2)}} \ket{k} + \frac{1 - \epsilon \nu_k}{\sqrt{d(1+\epsilon^2)}} \ket{k + \frac{d}{2}},
\end{equation}
for $\epsilon \leq 1$.
\begin{lemma}
\label{lem::Hamming_distance_pure}
    Let $D = d/2$ and $\mathcal{V} = \{-1,1\}^D$. The set of local binary hypotheses $\mathcal{F} = \{\rho_{\nu}\}_{\nu \in \mathcal{V}}$ defined by~\eqref{eqn::rho_nu_pure} induce a $2\eta$-Hamming separation for $\norm{\cdot}_F^2$ on $\mathcal{S}_{pure}(\mathbb{C}^d)$ where $\eta = \frac{\epsilon^2}{2d}$.
\end{lemma}
\noindent
For the outcomes of any gentle measurements on these local hypothesis states we have the following result.
\begin{lemma}
\label{lem::upper_bound_distances_pure_states}
    Let $\epsilon \leq 1$, $D = \frac{d}{2}$ and $(\rho_{\nu})_{\nu \in \mathcal{V}}$ as in~\eqref{eqn::rho_nu_pure}. Then, for the states defined in Definition~\ref{defn::mixed_distribution states} it holds
    \begin{equation*}
        \sup_{\norm{A}_{op} \leq 1} \sum_{j = 1}^D \frac{1}{2^{D-1}} \sum_{\nu \in \{-1,1\}^{D-1}} \Tr \left[ A (\rho_{\nu_{j+},i} - \rho_{\nu_{j-}, i})\right]^2 \leq \frac{392 \epsilon^2}{d}.
    \end{equation*}
\end{lemma}
\noindent
Using our reduction scheme then gives the lower bound
\begin{align*}
    \inf_{(M, \hat{\rho})} \sup_{\rho \in \mathcal{S}_{pure}(\mathbb{C}^d)} \mathbb{E}_{\rho}\left[ \norm{\hat{\rho} - \rho}_{F}^2\right] &\geq \frac{\epsilon^2}{2d} \left( \frac{d}{2} - \frac{1}{2} \sqrt{\frac{d}{2}} \left( \frac{25088 n \alpha^2 \epsilon^2}{(1 - 2\alpha)^4 d} \right)^{\frac{1}{2}} \right).
    \intertext{We chose 
\begin{equation*}
    \epsilon = \min\left\{1,  \frac{(1-2\alpha)^2 d}{224 \alpha \sqrt{n}} \right\},
\end{equation*}
which gives the lower bound}
    \inf_{(M, \hat{\rho})} \sup_{\rho \in \mathcal{S}_{pure}(\mathbb{C}^d)} \mathbb{E}_{\rho}\left[ \norm{\hat{\rho} - \rho}_{F}^2\right] &\geq \min \left\{ \frac{1}{8}, \frac{(1-2\alpha)^4}{401408} \frac{d^2}{n \alpha^2} \right\}.
\end{align*}
This result concludes the proof of Theorem~\ref{thm::lower_bound_qudit_estimation}.

\subsection{Lower bounds for the estimation of a probability vector}
\label{sec::lower_bound_construction_probability_vector}
We finalize this section by showing that we can use the same proof technique not only for state tomography but also of some parameter of interest of the state by proving a lower bound on the minimax estimation error for the estimation of the probability vector of a pure state. The following theorem gives matching lower bounds to the upper bound of Theorem~\ref{thm::upper_bound_prob_amplitude}.
\begin{theorem}
    \label{thm::lower_bound_probability_vector}
    Let $\alpha > 0$ be bounded away from $1/2$ by some constant. For a fixed orthonormal basis $(\ket{\psi_k})_{k = 1}^d$ of $\mathbb{C}^d$ and a pure state $\rho = \ket{\psi}\bra{\psi}$ with $\ket{\psi} = \sum_{k = 1}^d \gamma_k \ket{\psi_k}$ we set $p(\rho) = [|\gamma_1|^2,...,|\gamma_d|^2]^{\top}$. Then it holds
    \begin{equation*}
        \inf_{(M, \hat{p})} \sup_{\rho \in \mathcal{S}_{pure}(\mathbb{C}^d)} \mathbb{E}\left[ \norm{\hat{p} - p(\rho)}_2^2 \right] \geq \min\left\{ \frac{1}{8d}, \frac{(1-2\alpha)^4d}{1024 n \alpha^2} \right\}
    \end{equation*}
    Here the infimum is taken over all locally-$\alpha$-gentle meausrements $M$ and subsequent estimators $\hat{p}$.
\end{theorem}

\begin{proof}
    In order to prove this, we let $D= \frac{d}{2}$ and consider the pure states $(\rho_{\nu})_{\nu \in \mathcal{V}} = (\ket{\psi_{\nu}}\bra{\psi_{\nu}})_{\nu \in \mathcal{V}}$ for $\mathcal{V} = \{-1, 1\}^D$ given by
    \begin{equation}
    \label{eqn::definition_pure_states_lower_bound_prop_amplitude}
        \ket{\psi_{\nu}} = \sum_{k = 1}^D \sqrt{\frac{1 + \epsilon \nu_k}{d}} \ket{k} + \sqrt{\frac{1 - \epsilon \nu_k}{d}} \ket{k +\frac{d}{2}} = \sum_{k = 1}^d \gamma_{\nu, k} \ket{k}
    \end{equation}
    In Lemma~\ref{lem::Hamming_distance_probability_vector}, we show that these states induce a a $2\eta$-Hamming separation for $\norm{p(\cdot)}_2^2$ on $\mathcal{S}_{pure}(\mathbb{C}^d)$ where $\eta = \frac{\epsilon^2}{2d^2}$. Furthermore, Lemma~\ref{lem::upper_bound_distances_pure_states_for_amplitudes} shows that
    \begin{equation*}
        \sup_{\norm{A}_{op} \leq 1} \sum_{j = 1}^D \frac{1}{2^{D-1}} \sum_{\nu \in \{-1,1\}^{D-1}} \Tr \left[ A (\rho_{\nu_{j+}} - \rho_{\nu_{j-}})\right]^2 \leq \frac{\epsilon^2}{d}.
    \end{equation*}
    We then apply our reduction scheme to get the following lower bound for the minimax estimation error of the probability vector of a pure state:
    \begin{align*}
        \inf_{(M, \hat{p})} \max_{\rho \in \mathcal{S}_{pure}(\mathbb{C}^d)} \mathbb{E}_{\rho}\left[ \norm{\hat{p} - p(\rho)}_2^2 \right] &\geq \frac{\epsilon^2}{2d^2} \left( \frac{d}{2} - \frac{1}{2} \sqrt{\frac{d}{2}} \left( \frac{64 n\alpha^2}{(1 - 2\alpha)^4} \frac{\epsilon^2}{d} \right)^{\frac{1}{2}} \right).
        \intertext{Chosing \begin{equation*}
            \epsilon = \min\left\{ 1, \frac{(1-2\alpha)^2 d}{\sqrt{128} \sqrt{n} \alpha} \right\}
        \end{equation*}
        then gives the rate}
        \inf_{(M, \hat{p})} \max_{\rho \in \mathcal{S}_{pure}(\mathbb{C}^d)} \mathbb{E}_{\rho}\left[ \norm{\hat{p} - p(\rho)}_2^2 \right] &\geq \min\left\{ \frac{1}{8d}, \frac{(1-2\alpha)^4d}{1024 n \alpha^2} \right\}.
    \end{align*}
\end{proof}

\bibliographystyle{plain}
\bibliography{estimatingquditsbib}       

\newpage
\begin{appendix}
\section{Proof for the results from section~\ref{sec::gentle_measurements}}
\label{sec::additional_proofs_gentleness_results}
\begin{proof}[Proof of Proposition~\ref{prop::Privacy_implies_gentleness_pure}]
    Note that, by Proposition~\ref{prop::equalities_qDP}, we have 
    $$
    \lambda_{max}(E_{\omega}) \leq e^{\delta} \lambda_{min}(E_{\omega})
    $$ for all $\omega \in \Omega$, where $E_{\omega} = M_{\omega}^*M_{\omega}$. Let $|M_{\omega}| = \sqrt{E_{\omega}}$ be the unique positive square root of $E_{\omega}$. Then $\Tilde{M} = (|M_{\omega}|)_{\omega \in \Omega}$ does define a quantum measurement that has the same outcome probabilities as $M_{\omega}$. Given $|M_{\omega}|$, there exists an orthonormal basis $\ket{v_{1}},...,\ket{v_{d}}$ of $\mathbb{C}^d$ and $\lambda_{1},...,\lambda_{d} > 0$ such that
    \begin{equation*}
        |M_{\omega}| = \sum_{i = 1}^d \lambda_{i} \ket{v_{i}}\bra{v_{i}},
    \end{equation*}
    where $\lambda_{1} = \sqrt{\lambda_{min}(E_{\omega})}$ and $\lambda_{d} = \sqrt{\lambda_{max}(E_{\omega})}$. Note that the eigenvectors are eigenvalues are dependent on $\omega$ even though we omit this in the notation.
    Then, for any pure state $\rho = \ket{\psi}\bra{\psi}$, using Lemma~\ref{lemmaTraceNorm} and the definition of the post-measurement state~\eqref{eqn::post_measurement_state_pure}, we have
    \begin{align*}
        \norm{\rho - \rho_{|M| \to \omega}}_{Tr}^2 = 1 - \frac{\left|\bra{\psi}|M_{\omega}|\ket{\psi}\right|^2}{\bra{\psi}|M_{\omega}|^2\ket{\psi}} &\leq 1 - \frac{4 \lambda_1 \lambda_d}{(\lambda_d+ \lambda_1)^2} = \left(\frac{\lambda_d - \lambda_1}{\lambda_d + \lambda_1}\right)^2 \leq \left(\frac{e^{\frac{\delta}{2}} -1}{e^{\frac{\delta}{2}} +1}\right)^2,
    \end{align*}
    where the first inequality is due to the Kantorovich inequality \cite{Moslehian2012} and the second inequality is due to the fact that the eigenvalues of $E_{\omega}$ are given by the square of the eigenvalues of $|M_{\omega}|$. This shows that $\Tilde{M}$ is $\alpha$-gentle on pure states.
\end{proof}

\begin{proof}[Proof of Proposition~\ref{prop::pure_states_imply_all_states}]
    Let $\rho = \sum_{k = 1}^{d}\lambda_k \ket{\psi_k}\bra{\psi_k} \in \mathcal{S}(\mathbb{C}^d)$ be any quantum state and $M$ an $\alpha$-gentle measurement on $\mathcal{S}_{pure}(\mathbb{C}^d)$. Define \begin{equation*}
        \ket{\Psi} = \sum_{k = 1}^{d} \sqrt{\lambda_k} \ket{\psi_k} \otimes \ket{\psi_k} \in \mathcal{S}_{pure}(\mathbb{C}^d \otimes \mathbb{C}^d).
    \end{equation*} 
    Then it holds $\rho = \Tr_2[\ket{\Psi}\bra{\Psi}]$, where $\Tr_2$ is the partial trace over the seconds Hilbert space $\mathbb{C}^d$. Furthermore, for the measurement $M \otimes I = (M_{\omega} \otimes I)_{\omega \in \Omega}$ it holds
    \begin{align*}
        \mathbb{P}_{\rho}\left( R^M = \omega \right) = \Tr\left[ \rho M_{\omega}^* M_{\omega} \right] &= \Tr\left[ \Tr_2[\ket{\Psi}\bra{\Psi}] M_{\omega}^* M_{\omega} \right] 
        \\
        &= \Tr\left[ \ket{\Psi}\bra{\Psi} (M_{\omega}^*M_{\omega} \otimes I) \right]
        \\
        &= \Tr\left[ \ket{\Psi}\bra{\Psi} (M_{\omega}^* \otimes I) (M_{\omega} \otimes I) \right]
        \\
        &= \Tr\left[ \ket{\Psi}\bra{\Psi} (M_{\omega} \otimes I)^* (M_{\omega} \otimes I) \right] = \mathbb{P}_{\ket{\Psi}}\left( R^{M\otimes I} = \omega \right).
    \end{align*}
    Furthermore, we have 
    \begin{align*}
        M_{\omega} \rho M_{\omega}^* = M_{\omega} \Tr_2[\ket{\Psi}\bra{\Psi}] M_{\omega}^* = \Tr_2 \left[ (M_{\omega} \otimes I) \ket{\Psi}\bra{\Psi} (M_{\omega}^* \otimes I) \right].
    \end{align*}
    This shows that
    \begin{align*}
        \rho_{M \to \omega} &= \frac{1}{\mathbb{P}_{\rho}\left( R^M = \omega \right)} M_{\omega}\rho M_{\omega}^* 
        \\
        &= \frac{1}{\mathbb{P}_{\ket{\Psi}}\left( R^{M\otimes I} = \omega \right)} \Tr_2 \left[ (M_{\omega} \otimes I) \ket{\Psi}\bra{\Psi} (M_{\omega}^* \otimes I) \right]y
        \\
        &= \Tr_2 \left[ \frac{1}{\sqrt{\mathbb{P}_{\ket{\Psi}}\left( R^{M\otimes I} = \omega \right)}} (M_{\omega} \otimes I) \ket{\Psi}\bra{\Psi} (M_{\omega}^* \otimes I) \frac{1}{\sqrt{\mathbb{P}_{\ket{\Psi}}\left( R^{M\otimes I} = \omega \right)}} \right]
        \\
        &= \Tr_2\left[ \ket{\Psi_{M \otimes I \to \omega}}\bra{\Psi_{M \otimes I \to \omega}} \right].
    \end{align*}
    Since the trace norm is contractive under quantum channels such as the partial trace, we have
    \begin{align*}
        \norm{\rho - \rho_{M \to \omega}}_{Tr} &= \norm{\Tr_2[\ket{\Psi}\bra{\Psi}] - \Tr_2\left[ \ket{\Psi_{M \otimes I \to \omega}}\bra{\Psi_{M \otimes I \to \omega}} \right]}_{Tr} 
        \\
        &\leq \norm{\ket{\Psi}\bra{\Psi} -  \ket{\Psi_{M \otimes I \to \omega}}\bra{\Psi_{M \otimes I \to \omega}}}_{Tr}.
    \end{align*}
    Therefore, the gentleness of $M$ on $\mathcal{S}(\mathbb{C}^d)$ is bounded by the gentleness of $M \otimes I$ on $\mathcal{S}_{pure}(\mathbb{C}^d \otimes \mathbb{C}^d)$. By Lemma~\ref{lem::improved_constant_positivity}, $M$ is $\delta$-quantum-differentially-private on $\mathcal{S}(\mathbb{C}^d)$ for $\delta = 4 \arctanh(\alpha)$. Note that the proof of Lemma~\ref{lem::improved_constant_positivity} is independent of this proof As such, even though we have will proof it later, we may use its result here already. As such, we have
    \begin{equation*}
        \lambda_{max}(M_{\omega}^*M_{\omega} \otimes I) = \lambda_{max}(M_{\omega}^*M_{\omega}) \leq e^{\delta} \lambda_{min}(M_{\omega}^*M_{\omega}) = e^{\delta} \lambda_{min}(M_{\omega}^*M_{\omega} \otimes I),
    \end{equation*}
    which shows that $M \otimes I$ is also $\delta$-quantum-differentially-private on $\mathcal{S}(\mathbb{C}^d \otimes \mathbb{C}^d)$. Proposition~\ref{prop::Privacy_implies_gentleness_pure} then shows, that $M \otimes I$ is $\alpha$-gentle on $\mathcal{S}_{pure}(\mathbb{C}^d \otimes \mathbb{C}^d)$ from which we get
    \begin{align*}
        \norm{\rho - \rho_{M \to \omega}}_{Tr} \leq \norm{\ket{\Psi}\bra{\Psi} -  \ket{\Psi_{M \otimes I \to \omega}}\bra{\Psi_{M \otimes I \to \omega}}}_{Tr} \leq \alpha.
    \end{align*}
\end{proof}

\begin{proof}[Proof of Proposition~\ref{lem::gentleness_implies_privacy}]
    Assume that $M = (M_{\omega})_{\omega \in \Omega}$ is an $\alpha$-gentle measurement on $\mathcal{S}(\mathbb{C}^d)$. Define $E_{\omega} := M_{\omega}^*M_{\omega}$. Let $\rho_1, \rho_2 \in \mathcal{S}(\mathbb{C}^d)$ s.t. $\norm{\rho_1 - \rho_2}_{Tr} = 1$. Let us further denote by $p_1$ (respectively $p_2$) the probability of obtaining outcome $y$ under $\rho_1$ (respectively $\rho_2$). That is
    \begin{equation*}
        p_1 = \mathbb{P}_{\rho_1}\left(R^M = \omega \right) \hspace{20pt} \text{and} \hspace{20pt} p_2 = \mathbb{P}_{\rho_2}\left(R^M = \omega \right) 
    \end{equation*} Without loss of generality we assume that $p_1 > p_2 \geq 0$. Now, let
    \begin{equation*}
        \rho_{\lambda} = \lambda \rho_1 + (1-\lambda) \rho_2 \hspace{10pt} \text{for all } \lambda \in (0,1).
    \end{equation*}
    The probability of obtaining the outcome $y$ when measuring $\rho_{\lambda}$ is
    \begin{equation*}
        p_{\lambda} = \mathcal{P}_{\rho_{\lambda}}(R^M = \omega) = \Tr\left[ \rho_{\lambda} E_{\omega} \right] = \lambda \Tr\left[ \rho_1 E_{\omega} \right] + (1 - \lambda) \Tr\left[ \rho_2 E_{\omega} \right] = \lambda p_1 + (1 - \lambda) p_2.
    \end{equation*}
    The post-measurement state of $\rho_{\lambda}$ is then given by
    \begin{equation*}
        (\rho_{\lambda})_{M \to \omega} = \frac{1}{p_{\lambda}} M_{\omega} \rho_{\lambda} M_{\omega}^* = \frac{\lambda p_1 (\rho_1)_{M \to \omega} + (1-\lambda) p_2 (\rho_2)_{M \to \omega}}{\lambda p_1 + (1-\lambda) p_2}.
    \end{equation*}
    Now if we define $\delta = \frac{\lambda p_1}{\lambda p_1 + (1-\lambda) p_2} - \lambda > 0$, we get
    \begin{align*}
        \rho_{\lambda} - (\rho_{\lambda})_{M \to \omega} = \frac{\lambda p_1}{p_{\lambda}} \left((\rho_1 - (\rho_1)_{M \to \omega}\right) + \frac{(1-\lambda) p_2}{p_{\lambda}} \left(\rho_2 - (\rho_2)_{M \to \omega}\right) + \delta \left(\rho_2 - \rho_1\right)).
    \end{align*}
    By the triangle inequality and gentleness we now have
    \begin{align*}
        \delta \norm{\rho_2 - \rho_1}_{Tr} \leq \frac{\lambda p_1}{p_{\lambda}} \alpha + \frac{(1-\lambda) p_2}{p_{\lambda}} \alpha + \alpha = 2 \alpha
    \end{align*}
    Since we further assumed $\norm{\rho_2 - \rho_1}_{Tr} = 1$, we get $\delta \leq 2 \alpha$. This allows us to write
    \begin{align*}
        p_1 = \frac{\lambda - \lambda^2 + \delta (1-\lambda)}{\lambda - \lambda^2 - \delta \lambda} p_2 \leq \frac{\lambda - \lambda^2 + 2\alpha (1-\lambda)}{\lambda - \lambda^2 - 2 \alpha \lambda} p_2 \hspace{10pt} \text{for all } \lambda \in (0,1 - 2 \alpha).
    \end{align*}
    The last inequality only holds as long as the denominator is positive which is the case for $\lambda < 1 - 2 \alpha$. As such, for $\lambda_0 = \frac{1 - 2 \alpha}{2} < 1 - 2 \alpha$, we obtain
    \begin{equation*}
        p_1 \leq \frac{\lambda_0 - \lambda_0^2 + 2\alpha (1-\lambda_0)}{\lambda_0 - \lambda_0^2 - 2 \alpha \lambda_0} p_2 = \left( \frac{1 + 2\alpha}{1 - 2\alpha} \right)^2 p_2.
    \end{equation*}
    Now, since we started with $\rho_1, \rho_2$ such that $\norm{\rho_1 - \rho_2}_{Tr} = 1$, the last relation holds for every pure state. As such $M$ is $2 \log\left( \frac{1 + 2\alpha}{1 - 2\alpha} \right)$-quantum-differentially-private on pure states. Proposition~\ref{prop::equalities_qDP} then proves that $M$ is $2 \log\left( \frac{1 + 2\alpha}{1 - 2\alpha} \right)$-quantum-differentially-private on the whole space.
\end{proof}

\begin{proof}[Proof of Lemma~\ref{lem::improved_constant_positivity}]
    Let $M_{\omega} = \sum_{i = 1}^d \lambda_i \ket{v_i}\bra{v_i}$ be positive-definite with maximal and minimal eigenvalue $\lambda_d$ and $\lambda_1$ respectively. Consider the gentleness of $M_{\omega}$ on the pure state $\ket{\psi} = \frac{1}{\sqrt{\lambda_1 + \lambda_d}} (\sqrt{\lambda_d} \ket{v_1} + \sqrt{\lambda_1} \ket{v_d})$. Lemma~\ref{lemmaTraceNorm} then shows that, due to the gentleness of $M$ on $\rho = \ket{\psi}\bra{\psi}$, we have
    \begin{align*}
        \alpha^2 \geq \norm{\rho - \rho_{M \to \omega}}_{Tr}^2 = 1- \frac{\left| \bra{\psi} M_{\omega} \ket{\psi} \right|^2}{\bra{\psi} M_{\omega}^2 \ket{\psi}}.
    \end{align*}
    For the numerator and denominator we have
    \begin{align*}
        \left|\bra{\psi} M_{\omega} \ket{\psi}\right|^2 &= \frac{1}{(\lambda_1 + \lambda_d)^2} \left( \lambda_d \lambda_1 + \lambda_d \lambda_1 \right)^2 = \frac{4 \lambda_1^2 \lambda_d^2}{(\lambda_1+ \lambda_d)^2}
        \\
        \bra{\psi} M_{\omega}^2 \ket{\psi} &= \frac{1}{\lambda_1 + \lambda_d} \left( \lambda_d \lambda_1^2 + \lambda_d^2 \lambda_1 \right) = \lambda_1 \lambda_d.
    \end{align*}
    from which we obtain
    \begin{equation*}
        \frac{\left| \bra{\psi} M_{\omega} \ket{\psi} \right|^2}{\bra{\psi} M_{\omega}^2 \ket{\psi}} = \frac{4 \lambda_1 \lambda_d}{(\lambda_1 + \lambda_d)^2}
    \end{equation*}
    This gives
    \begin{equation*}
        \alpha^2 \geq \frac{(\lambda_d - \lambda_1)^2}{(\lambda_d + \lambda_1)^2} \hspace{20pt} \text{or equivalently} \hspace{20pt} \alpha \geq \frac{\lambda_d - \lambda_1}{\lambda_d + \lambda_1} = \frac{\frac{\lambda_d}{\lambda_1} - 1}{\frac{\lambda_d}{\lambda_1} +1 }
    \end{equation*}
    Using the fact that for $E_{\omega} = M_{\omega}^2$ we have $\lambda_{max}(E_{\omega}) = \lambda_d^2$ and $\lambda_{min}(E_{\omega}) = \lambda_1^2$, we get
    \begin{align*}
        \alpha \geq \frac{e^{2 \frac{1}{4}\log\left( \frac{\lambda_{max}(E_{\omega})}{\lambda_{min}(E_{\omega})} \right)} - 1}{e^{2 \frac{1}{4}\log\left( \frac{\lambda_{max}(E_{\omega})}{\lambda_{min}(E_{\omega})} \right)} + 1} = \tanh\left( \frac{1}{4} \log\left( \frac{\lambda_{max}(E_{\omega})}{\lambda_{min}(E_{\omega})} \right) \right).
    \end{align*}
    Finally, using the monotonicity of $\tanh$, we get
    \begin{equation*}
        \frac{\lambda_{max}(E_{\omega})}{\lambda_{min}(E_{\omega})} \leq e^{4 \arctanh(\alpha)}
    \end{equation*}
    which together with Proposition~\ref{prop::equalities_qDP} shows that $M$ is $\delta$-quantum-differentially-private on $\mathcal{S}(\mathbb{C}^d)$ for
    \begin{equation*}
        \delta = 4 \arctanh(\alpha).
    \end{equation*}
\end{proof}

\begin{lemma}
\label{lem::gentleness_of_gentleized_general_measurement}
    Let $\alpha \in [0,1]$ and $\Omega = \{1,...,|\Omega|\}$ be a finite set such that $M = (M_\omega)_{\omega \in \Omega}$ is a quantum measurement. For $E_{\omega} = M_{\omega}^*M_{\omega}$ and $e_m \in \{0,1\}^D$ being the vector that is zero everywhere except in the m-th entry, define
    \begin{equation*}
    \label{eqn::gentle-ize_finite_measurement}
        E_{\delta, z} = \left( \frac{e^{\frac{\delta}{2}}}{e^{\frac{\delta}{2}} + 1} \right)^{|\Omega|} \sum_{m = 1}^{|\Omega|} e^{-\frac{\delta}{2} \norm{z - e_m}_1} E_m, \hspace{20pt} \text{with} \hspace{20pt} M_{\delta, z} = \sqrt{E_{\delta, z}}
    \end{equation*}
    for every $z \in \{0,1\}^{|\Omega|}$ and $\delta = 4 \arctanh(\alpha)$ and set $M_{\delta} = (M_{\delta, z})_{z \in \{0,1\}^D}$. Then $M_{\delta}$ is an $\alpha$-gentle measurement on $\mathcal{S}(\mathbb{C}^d)$.
\end{lemma}

\begin{proof}
    We will make use of Proposition~\ref{prop::Privacy_implies_gentleness_pure} and~\ref{prop::pure_states_imply_all_states} to show the results. We first assure that the operators $M_{\delta} = (M_{\delta, z})_{z \in \{0,1\}^{|\Omega|}}$ do define a quantum measurement, that is 
    \begin{equation*}
        \sum_{m = 1}^{|\Omega|} E_m = \mathbbm{1}.
    \end{equation*}
    Since each $E_{\delta, z}$ is a sum of positive operators, $E_{\delta, z}$ is itself positive and we can take the operator square root $M_{\delta, z} = \sqrt{E_{\delta, z}}$. In order to assure that $M_{\delta} = (M_{\delta, z})_{z \in \{0,1\}^{|\Omega|}}$ indeed defines a quantum measurement we must assure that it fulfills the completeness relation. We have
    \begin{align*}
        \sum_{z \in \{0,1\}^{|\Omega|}} E_{\delta, z} &=  \left( \frac{e^{\frac{\delta}{2}}}{e^{\frac{\delta}{2}} + 1} \right)^{|\Omega|} \sum_{m = 1}^{|\Omega|} E_m \sum_{z \in \{0,1\}^{|\Omega|}} e^{-\frac{\delta}{2} \norm{z - e_m}_1} 
        \\
        &= \left( \frac{e^{\frac{\delta}{2}}}{e^{\frac{\delta}{2}} + 1} \right)^{|\Omega|} \sum_{m = 1}^{|\Omega|} E_m \left[ \sum_{\{z | z_m = 1\}} e^{-\frac{\delta}{2} \norm{z - e_m}_1} + \sum_{\{z | z_m = 0\}} e^{-\frac{\delta}{2} \norm{z - e_m}_1}\right]
        \\
        &= \left( \frac{e^{\frac{\delta}{2}}}{e^{\frac{\delta}{2}} + 1} \right)^{|\Omega|} \sum_{m = 1}^{|\Omega|} E_m \left[ \sum_{j = 0}^{{|\Omega|}-1} \binom{{|\Omega|}-1}{j} e^{-\frac{\delta}{2} j } + \sum_{j = 0}^{{|\Omega|}-1} \binom{{|\Omega|}-1}{j} e^{-\frac{\delta}{2} (j+1)} \right]
        \\
        &= \left( \frac{e^{\frac{\delta}{2}}}{e^{\frac{\delta}{2}} + 1} \right)^{|\Omega|}  \sum_{m = 1}^{|\Omega|} E_m \left( (e^{-\frac{\delta}{2}} + 1)^{{|\Omega|}-1} + (e^{-\frac{\delta}{2}} +1)^{{|\Omega|}-1} e^{-\frac{\delta}{2}} \right)
        \\
        &= \frac{e^{\frac{\delta}{2}}}{e^{\frac{\delta}{2}} + 1} \left( 1 + e^{-\frac{\delta}{2}} \right) \mathbbm{1}
        \\
        &= \mathbbm{1}.
    \end{align*}
    Therefore, taking $M_{\delta, z} = \sqrt{E_{\delta, z}}$, the operators $M_{\delta} = (M_{\delta, z})_{z \in \{0,1\}^{|\Omega|}}$ do indeed define a quantum measurement. Now, let $z \in \{0,1\}^{|\Omega|}$ with $\norm{z}_1 = j$ be fixed. We have
    \begin{equation*}
        e^{-\frac{\delta}{2} (j+1)} \leq e^{-\frac{\delta}{2} \norm{z - e_m}} \leq e^{-\frac{\delta}{2} (j-1)}
    \end{equation*}
    for all $m$ from which we obtain
    \begin{equation*}
        \frac{\mathbbm{P}_{\rho}(R^{M_{\delta}} = z)}{\mathbbm{P}_{\rho'}(R^{M_{\delta}} = z)} \leq e^{\frac{\delta}{2}((k+1)-(k-1))} \leq e^{\delta}
    \end{equation*}
    for any $\rho = \ket{\psi}\bra{\rho}$ $\rho' = \ket{\psi'}\bra{\psi'}$ and any $z \in \{0,1\}^{|\Omega|}$. This shows that $M_{\delta}$ is $\delta$-quantum-differentially-private on $\mathcal{S}_{pure}(\mathbb{C}^d)$. Proposition~\ref{prop::Privacy_implies_gentleness_pure} then assures that $M_{\delta}$ is $\alpha$-gentle on $\mathcal{S}_{pure}(\mathbb{C}^d)$. Since the operators $M_{\delta, \omega}$ are positive and self-adjoint, Proposition~\ref{prop::pure_states_imply_all_states} assures the $\alpha$-gentleness of $M_{\delta}$ on $\mathcal{S}(\mathbb{C}^d)$.
\end{proof}

\newpage
\section{Proofs for the results from Section~\ref{sec::qLS_pure_states}}
\label{sec::additional_proofs_gentleness_introduction_section}
\begin{proof}[Proof of Proposition~\ref{prop::physical_implementation}]
    Let $\ket{\psi} = \sum_{l = 1}^d \gamma_l \ket{l}$ and recall $\rho = \ket{\psi} \bra{\psi}$. The state of the combined system $\rho \otimes \sigma$ is pure and given by $\rho \otimes \sigma = \ket{\Psi}\bra{\Psi}$ for
    \begin{equation*}
        \ket{\Psi} = \ket{\psi} \otimes \ket{0} \otimes ... \otimes \ket{0}.
    \end{equation*}
    
    The first part of the algorithm is given by the Rotation gates $R_{\delta}$ which rotate the initial $\ket{0}$ states into a superposition of $\ket{0}$ and $\ket{1}$ parametrized by the privacy parameter $\delta$. Define the states
    \begin{align*}
        \ket{\phi_{1,\delta}} &= \left( e^{\frac{\delta}{2}} + 1 \right)^{-\frac{1}{2}} \left( \ket{0} + e^{\frac{\delta}{4}} \ket{1} \right) \in \mathcal{S}_{pure}(\mathbb{C}^2)
        \intertext{and}
        \ket{\phi_{0,\delta}} &= \left( e^{\frac{\delta}{2}} + 1 \right)^{-\frac{1}{2}} \left( e^{\frac{\delta}{4}} \ket{0} + \ket{1} \right) \in \mathcal{S}_{pure}(\mathbb{C}^2).
    \end{align*}
    Now, let $R_{\delta}: \mathbb{C}^2 \to \mathbb{C}^2$ be the unitary matrix that maps $\ket{0}$ to $\ket{\phi_{1,\delta}}$, that is
    \begin{alignat*}{3}
        R_{\delta}: & \; \mathbb{C}^2 && \; \to && \; \mathbb{C}^2
        \\
        & \ket{0} && \mapsto && \ket{\phi_{0,\delta}}
        \\
        & \ket{1} && \mapsto && \ket{\phi_{0,\delta}}^{\perp},
    \end{alignat*}
    where 
    \begin{equation*}
        \ket{\phi_{0,\delta}}^{\perp} = \left( e^{\frac{\delta}{2}} + 1 \right)^{-\frac{1}{2}} \left( \ket{0} - e^{\frac{\delta}{4}} \ket{1} \right).
    \end{equation*}
    As such, applying the $R_{\delta}$ to each of the $d$ qubit registers transforms the system at the end of the first step of the algorithm into
    \begin{equation*}
        \ket{\Psi_{\delta}} = \left( \mathbbm{1} \otimes R_{\delta} \otimes ... \otimes R_{\delta} \right) \left( \ket{\Psi} \right) = \ket{\psi} \otimes \ket{\phi_{0,\delta}} \otimes ... \otimes \ket{\phi_{0,\delta}}.
    \end{equation*}

    The second part of the quantum algorithm is given by applying a number of qubit controlled NOT-gates to the system. Each can be seen as entangling the $k$-th mode of the qudit with one of the $d$ qubit systems in such a way that it flips the "true" value of the mode with low probability. More precisely, if the qudit is in the $k$-th basis state, the gate flips the two modes of the qubit system, \emph{i.e.} the first (qudit) system controls the second (qubit) system. It is given by
    \begin{alignat*}{7}
        U_k: \; & \mathbb{C}^d \; && \otimes \; && \mathbb{C}^2 && \to \; && \mathbb{C}^d \; && \otimes \; && \mathbb{C}^2
        \\
        & \ket{l} && \otimes && \ket{0} && \mapsto && \ket{l} && \otimes  &&\ket{0}
        \\
        & \ket{l} && \otimes && \ket{1} && \mapsto && \ket{l} && \otimes  &&\ket{1}
        \\
        & \ket{k} && \otimes && \ket{0} && \mapsto && \ket{k} && \otimes  &&\ket{1}
        \\
        & \ket{k} && \otimes && \ket{1} && \mapsto && \ket{k} && \otimes  &&\ket{0}
    \end{alignat*}
    for $l \neq k$. Let us now consider the action of $U_k$ on the state $\ket{\psi} \otimes \ket{\phi_{0,\delta}}$ which is given by
    \begin{align*}
        U_k \left( \ket{\psi} \otimes \ket{\phi_{0,\delta}} \right) = \gamma_k \ket{k} \otimes \ket{\phi_{1,\delta}} + \sum_{l \neq k} \gamma_l \ket{l} \otimes \ket{\phi_{0,\delta}}.
    \end{align*}
    In order to perform the algorithm as depicted in Figure~\ref{alg::Quantum_Label_Switch}, we perform consecutively the operations $U_1,\ldots,U_d$. Due to the product structure of the state and of the measurement, the order in which we perform the $U_k$ operations does not matter. Formally, we can describe the global action by an extended version $\tilde{U}_k$ which is defined on $\mathbb{C}^d \otimes (\mathbb{C}^2)^{\otimes d}$ such that it acts like $U_k$ on the $k$-th register, that is, formally
    \begin{alignat*}{15}
        \tilde{U}_k: \quad & \mathbb{C}^d \quad && \otimes && (\mathbb{C}^{2})^{\otimes (k-1)} && \otimes \quad && \mathbb{C}^2 \quad && \otimes \; && (\mathbb{C}^{2})^{\otimes (d-k)} \; && \to \quad  && \mathbb{C}^d \quad && \otimes && (\mathbb{C}^{2})^{\otimes (k-1)} && \otimes \quad && \mathbb{C}^2 \quad&& \otimes \; && (\mathbb{C}^{2})^{\otimes (d-k)}
        \\
        & \ket{l} && \otimes && \quad \ket{v} && \otimes && \ket{0} && \otimes && \quad \ket{v'} && \mapsto && \ket{l} && \otimes && \quad \ket{v} && \otimes && \ket{0} && \otimes && \quad \ket{v'}
        \\
        & \ket{l} && \otimes && \quad \ket{v} && \otimes && \ket{1} && \otimes && \quad \ket{v'} && \mapsto && \ket{l} && \otimes && \quad \ket{v} && \otimes && \ket{1} && \otimes && \quad \ket{v'}
        \\
        & \ket{k} && \otimes && \quad \ket{v} && \otimes && \ket{0} && \otimes && \quad \ket{v'} && \mapsto && \ket{k} && \otimes && \quad \ket{v} && \otimes && \ket{1} && \otimes && \quad \ket{v'}
        \\
        & \ket{k} && \otimes && \quad \ket{v} && \otimes && \ket{1} && \otimes && \quad \ket{v'} && \mapsto && \ket{k} && \otimes && \quad \ket{v} && \otimes && \ket{0} && \otimes && \quad \ket{v'},
    \end{alignat*}
    where $l \neq k$ and $\ket{v} \in (\mathbb{C}^2)^{\otimes (k-1)}, \ket{v'} \in (\mathbb{C}^2)^{\otimes (d-k)}$ are arbitrary vectors. The action of $\tilde{U}_k$ on $\ket{\psi} \otimes \ket{v} \otimes \ket{\phi_{1,\delta}} \otimes \ket{v'}$ is therefore given by
    \begin{align*}
        \tilde{U}_k \left( \ket{\psi} \otimes \ket{v} \otimes \ket{\phi_{0,\delta}} \otimes \ket{v'} \right) = \gamma_k \ket{k} \otimes \ket{v} \otimes \ket{\phi_{1,\delta}} \otimes \ket{v'}+ \sum_{l \neq k} \gamma_l \ket{l} \otimes \ket{v} \otimes \ket{\phi_{0,\delta}} \otimes \ket{v'}.
    \end{align*}
    Therefore, after applying consecutively the controlled gates $\tilde{U}_1,...,\tilde{U}_d$ to the state 
    \begin{equation*}
        \ket{\Psi_{\delta}} = \ket{\psi} \otimes \ket{\phi_{0,\delta}} \otimes ... \otimes \ket{\phi_{0,\delta}}
    \end{equation*}
    yields
    \begin{align*}
        \ket{\Psi_{\delta,\tilde{U}}} = \tilde{U}_d\bigg(... \tilde{U}_1\Big( \ket{\psi} \otimes \ket{\phi_{0,\delta}} \otimes ... \otimes \ket{\phi_{0,\delta}} \Big) ... \bigg) = \sum_{k = 1}^d \gamma_k \ket{k} \otimes \ket{\phi_{1,\delta}^{(k)}},
    \end{align*}
    where
    \begin{align*}
        \ket{\phi_{1,\delta}^{(k)}} = \ket{\phi_{0,\delta}}^{\otimes (k-1)} \otimes \ket{\phi_{1,\delta}} \otimes \ket{\phi_{0,\delta}}^{\otimes (d-k)}.
    \end{align*}

    The algorithm now works by measuring the $d$ qubit systems in their computational basis $\ket{0}, \ket{1}$. In order to calculate the outcome probabilities of this measurement we first give an alternative representation of $\ket{\phi_{1,\delta}^{(k)}}$. By using the distributive property of the tensor product we can write
    \begin{align*}
        \ket{\phi_{1,\delta}^{(k)}} &= \ket{\phi_{0,\delta}}^{\otimes (k-1)} \otimes \ket{\phi_{1,\delta}} \otimes \ket{\phi_{0,\delta}}^{\otimes (d-k)}
        \\
        =& \frac{1}{(e^{\frac{\delta}{2}} +1)^{\frac{d}{2}}} \sum_{\omega_{-k} \in \{0,1\}^{d-1}} e^{\frac{\delta}{4} (d-1 - \norm{\omega_{-k}}_1)} \cdot \\
        & \quad \cdot \ket{\omega_1} \otimes ... \otimes \ket{\omega_{k-1}}  \otimes \left( \ket{0} + e^{\frac{\delta}{4}} \ket{1} \right) \otimes  \ket{\omega_{k+1}} \otimes ... \otimes \ket{\omega_{d}} ,
        \intertext{where $\omega_{-k} = (\omega_1,\ldots, w_{k-1},\omega_{k+1},\ldots,\omega_d)\in \{0,1\}^{d-1}$ is the vector where we omitted the $k$-th entry. We can further expand this as}
        =& \frac{1}{(e^{\frac{\delta}{2}} +1)^{\frac{d}{2}}} \sum_{\omega_{-k} \in \{0,1\}^{d-1}} e^{\frac{\delta}{4} (d-1 - \norm{\omega_{-k}}_1)} \Bigg[ \left( \ket{\omega_1} \otimes ... \otimes \ket{\omega_{k-1}} \right) \otimes \ket{0} \otimes \left( \ket{\omega_{k+1}} \otimes ... \otimes \ket{\omega_{d}} \right)
        \\
        &\qquad \qquad \qquad + \left( e^{\frac{\delta}{4}} \ket{\omega_1} \otimes ... \otimes \ket{\omega_{k-1}} \right) \otimes \ket{1} \otimes \left( \ket{\omega_{k+1}} \otimes ... \otimes \ket{\omega_{d}} \right) \Bigg]
        \\
        &= \frac{1}{(e^{\frac{\delta}{2}} +1)^{\frac{d}{2}}} \sum_{\omega \in \{0,1\}^d} e^{\frac{\delta}{4} (d - \norm{\omega - e_k}_1)} \ket{\omega_1} \otimes ... \otimes \ket{\omega_d}
        \\
        &= \left(\frac{e^{\frac{\delta}{2}}}{e^{\frac{\delta}{2}} +1}\right)^{\frac{d}{2}} \sum_{\omega \in \{0,1\}^d} e^{-\frac{\delta}{4} \norm{\omega - e_k}_1} \ket{\omega_1} \otimes ... \otimes \ket{\omega_d}
    \end{align*}
    As such, we get
   \begin{align*}
        \ket{\Psi_{\delta,\tilde{U}}} 
        = \left( \frac{e^{\frac{\delta}{2}}}{e^{\frac{\delta}{2}} + 1} \right)^{\frac{d}{2}} \sum_{k = 1}^d \gamma_k \sum_{\omega \in \{0,1\}^d} e^{-\frac{\delta}{4} \norm{\omega - e_k}_1} 
         \ket{k} \otimes \ket{\omega_1} \otimes ... \otimes \ket{\omega_d}.
    \end{align*}

    Let us now see what happens when we measure the qubit registers of the system, which is in the state $\ket{\Psi_{\delta,\tilde{U}}}$, in the computational basis, that is, using the measurement $\tilde{M}_{\tilde{B}}$ defined by $\tilde{M}_{\tilde{B}, \omega}$
    \begin{equation*}
        \tilde{M}_{\tilde{B},\omega} = (\mathbbm{1} \otimes \ket{\omega_1}\bra{\omega_1} \otimes ... \otimes \ket{\omega_d}\bra{\omega_d})_{\omega \in \{0,1\}^d}.
    \end{equation*}
    The operators $\tilde{M}_{\tilde{B},\omega}$ are given by projections in the enlarged Hilbert space $\mathbb{C}^d \otimes (\mathbb{C}^{2})^{\otimes d}$. The probability of obtaining the outcome $\omega^{\circ} \in \{0,1\}^d$ is given for $\ket{\omega^{\circ}} = \ket{\omega^{\circ}_1} \otimes ... \otimes  \ket{\omega^{\circ}_d}$
    \begin{align*}
        \mathbb{P}_{\rho_{\delta}}(R^{\tilde{M}_{\tilde{B}}} = \omega^{\circ}) = \norm{(\mathbbm{1} \otimes \ket{\omega^{\circ}}\bra{\omega^{\circ}})\ket{\Psi_{\delta,\tilde{U}}}}^2.
    \end{align*}
    We have
    \begin{align*}
        (\mathbbm{1} \otimes \ket{\omega^{\circ}}\bra{\omega^{\circ}})\ket{\Psi_{\delta,\tilde{U}}} &= \sum_{k = 1}^d \gamma_k \ket{k} \otimes \ket{\omega^{\circ}}\bra{\omega^{\circ}}\ket{\phi_{0,\delta}^{(k)}}
        \\
        &=  \sum_{k = 1}^d \gamma_k \ket{k} \otimes \ket{\omega^{\circ}} \left( \frac{e^{\frac{\delta}{2}}}{e^{\frac{\delta}{2}} + 1} \right)^{\frac{d}{2}} \sum_{\omega \in \{0,1\}^d} e^{-\frac{\delta}{4}\norm{\omega - e_k}_1} \bra{\omega^{\circ}}\ket{\omega}
        \\
        &=  \left( \frac{e^{\frac{\delta}{2}}}{e^{\frac{\delta}{2}} + 1} \right)^{\frac{d}{2}}  \sum_{k = 1}^d e^{-\frac{\delta}{4}\norm{\omega^{\circ} - e_k}_1} \gamma_k \ket{k} \otimes \ket{\omega^{\circ}} 
    \end{align*}
    and therefore
    \begin{align*}
        \mathbb{P}_{\rho_{\delta}}(R^{\tilde{M}_{\tilde{B}}} = \omega^{\circ}) &= \norm{(\mathbbm{1} \otimes \ket{\omega^{\circ}}\bra{\omega^{\circ}})\ket{\Psi_{\delta,\tilde{U}}}}^2 
        \\
        &= \left( \frac{e^{\frac{\delta}{2}}}{e^{\frac{\delta}{2}} +1} \right)^d \sum_{k = 1}^d  e^{-\frac{\delta}{2} \norm{\omega^{\circ} - e_k}_1} |\gamma_k|^2 = \mathbb{P}_{\rho}(R^{M_{\delta}} = \omega^{\circ}).
    \end{align*}
    Furthermore, we can calculate the post-measurement state of $\rho_{\delta,\tilde{U}}$ as
    \begin{equation*}
        (\rho_{\delta,\tilde{U}})_{\tilde{M}_{\tilde{B}} \to \omega^{\circ}} = \frac{1}{\mathbb{P}_{\rho_{\delta}}(R^{\tilde{M}_{\tilde{B}}} = \omega^{\circ})} (\mathbbm{1} \otimes \ket{\omega^{\circ}}\bra{\omega^{\circ}})\ket{\Psi_{\delta,\tilde{U}}} \bra{\Psi_{\delta,\tilde{U}}} (\mathbbm{1} \otimes \ket{\omega^{\circ}}\bra{\omega^{\circ}}).
    \end{equation*}
    For a composite system $\mathcal{H}_1 \otimes \mathcal{H}_2$ the partial trace with respect to the second subsystem can be calculated, see \cite{maziero_computing_2017}:
    \begin{equation*}
        \Tr_2[\rho] = \sum_{k = 1}^{d_2} \left( \mathbbm{1} \otimes \bra{b_k} \right) \rho \left( \mathbbm{1} \otimes \ket{b_k} \right),
    \end{equation*}
    where $(\ket{b_k})_{k = 1}^{d_2}$ is an orthonormal basis of $\mathcal{H}_2$.
    Applying this formula to our case gives
    \begin{equation*}
        \Tr_2\left[ (\mathbbm{1} \otimes \ket{\omega^{\circ}}\bra{\omega^{\circ}})\ket{\Psi_{\delta,\tilde{U}}} \bra{\Psi_{\delta,\tilde{U}}} (\mathbbm{1} \otimes \ket{\omega^{\circ}}\bra{\omega^{\circ}}) \right] = M_{\delta,\omega^{\circ}} \ket{\psi}\bra{\psi}M_{\delta,\omega^{\circ}}^*,
    \end{equation*}
    where $M_\delta$ has been introduced in~\eqref{eqn::measurement_operators_gentle}. Since we have already shown $\mathbb{P}_{\rho_{\delta,\tilde{U}}}(R^{\tilde{M}_{\tilde{B}}} = \omega^{\circ}) = \mathbb{P}_{\rho}(R^{M_{\delta}} = \omega^{\circ})$ and using the form of the post-measurement states for $M_{\delta}$ given in~\eqref{eqn::post_measuremenent_state_pure_state}, we obtain
    \begin{equation*}
        \Tr_2\left[ (\rho_{\delta,\tilde{U}})_{\tilde{M}_{\tilde{B}} \to \omega^{\circ}} \right] = \rho_{M_{\delta} \to \omega^{\circ}}.
    \end{equation*}
    This shows that the implementation we provided indeed gives the same outcome probabilities and states as the measurement~\eqref{eqn::measurement_operators_gentle}, showing the equality of the two.
\end{proof}

\newpage
\section{Proof of Theorem~\ref{thm::upper_bound_qLS_qudits}}
\label{sec::additional_calculations_upper_bound}
In this section we give a proof of Theorem~\ref{thm::upper_bound_qLS_qudits} which we restate for the readers convenience.
\begin{Restatementtheorem*}{\scshape~ \ref{thm::upper_bound_qLS_qudits}}.\hspace{7pt}
    Let $d \geq 2$, $1 \leq r \leq d$ and $\rho \in \mathcal{S}_r(\mathbb{C}^d)$. For $\alpha \in (0, 1)$ set $\delta = 4 \arctanh(\alpha)$ and for $\epsilon \in (0,1]$ set
    \begin{equation*}
        t^2(\epsilon) = 92 \log(d/\epsilon) \frac{d^2}{n\alpha^2} 
    \end{equation*}
    Let $\hat{\rho}_n = \mathcal{A}(\check{\rho}_n)$ be the estimator described previously, where $\mathcal{A}$ is the spectral thresholding Algorithm~\ref{alg::spectral_thresholding} with threshold $2t(\epsilon)$. Then $\hat{\rho}_n$ is a locally-$\alpha$-gentle estimator of $\rho$ such that and there exists a constant $C > 0$ independent of $r, d, n, \epsilon$ and $\alpha$ such that
    \begin{equation*}
        \sup_{\rho \in \mathcal{S}_r(\mathbb{C}^d)} \mathbb{E}_{\rho}\left[ \norm{\hat{\rho}_n - \rho}_F^2 \right] \leq C \frac{rd^2}{n \alpha^2} \log\left(\frac{d}{\epsilon}\right).
    \end{equation*}
    If we additionally assume $\lambda_{min}(\rho) \geq 6t(\epsilon)$, then
    \begin{equation*}
        \mathbb{P}_{\rho}\left( \rank(\hat{\rho}_n) = r \right) \geq 1 - \epsilon.
    \end{equation*}
\end{Restatementtheorem*}

\begin{proof}
    As we have already seen, the measurement $M_{\delta}$ defined in~\eqref{defn::measurement_operators_qls_mixed_MUB} on which the estimator is based, remains gentle, which shows that the estimator $\bar{\rho}_n$ is locally-$\alpha$-gentle. 
    Now, let $H \in \{0,1\}^D$ be as in~\eqref{eqn::duality_2_desing_MUB}. Similarly as for pure states, we have
    \begin{align}
        \mathbb{E}_{\rho}[H_m] 
        \label{eqn::expectation_hm}
        &= \frac{1}{d+1} \left( \alpha  \Tr\left[ \ket{v_m}\bra{v_m} \rho \right] + \beta \right)
    \end{align}
    for $\beta = (e^{\frac{\delta}{2}} + 1)^{-1}$, which is shown in Lemma~\ref{lem::Expectation_of_measurement_MUB} in Appendix~\ref{sec::technical_proofs_for_upper_bound}. 
    Using equation (\ref{eqn::2design_property}) we can then calculate
    \begin{align*}
        \mathbb{E}_{\rho}\left[ \sum_{m = 1}^D H_m \ket{v_m}\bra{v_m} \right] 
        &= \frac{\alpha}{d+1} \rho + \left(\frac{\alpha}{d+1} + \beta\right) \mathbbm{1}
    \end{align*}
    as shown in Lemma~\ref{lem::expectation_H_m_v_m_MUB}.
    By linearity of the expectation, we see that $\bar{\rho}_n$ is unbiased. We now want to apply the Matrix-Bernstein-inequality from \cite{tropp_introduction_2015} which we will restate for convenience.
    \begin{theorem}[Matrix-Bernstein-inequality, \cite{tropp_introduction_2015}]
    \label{thm::Matrix-Bernstein-inequality}
        Let $A_i$ be a sequence of independent $d$-dimensional hermitian matrices with
        \begin{equation*}
            \mathbb{E}\left[ A_i \right] = 0, \hspace{10pt} \norm{A_i}_{op} \leq R \text{ \; a.s. for all $i$ and }\sigma^2 = 
            \norm{\sum_{i = 1}^n \mathbb{E}\left[ A_i^2 \right]}_{op}.
        \end{equation*}
        Then, for all $t > 0$, it holds
        \begin{equation*}
            \mathbb{P}\left( \norm{\sum_{i = 1}^n A_i}_{op} \geq t \right) 
            \leq d \exp\left( -\frac{3t^2}{2 \left( 3 \sigma^2 + R t \right)} \right).
        \end{equation*}
    \end{theorem}
    Given $n$ independent copies of the qudit $\rho$ and measuring all qudits with the measurement $M_{\delta}$, we obtain $n$ i.i.d. random matrices $\tilde{\rho}_i$ given by
    \begin{equation}
    \label{eqn::qLS_est_before_proj_single}
        \tilde{\rho}_i = \frac{(d+1)}{\alpha}  \sum_{m = 1}^D {H}_m^{(i)} \ket{v_m}\bra{v_m} - \left( 1 + \frac{d+1}{\alpha}\beta \right) \mathbbm{1} \hspace{20pt} \text{such that } \bar{\rho}_n = \frac{1}{n} \sum_{i = 1}^n \tilde{\rho}_i
    \end{equation}
    In order to apply the Matrix-Bernstein-inequality we need to calculate $\mathbb{E}[\bar{\rho}_1^2]$. Lemma~\ref{lem::second_moment_rho_hat_MUB} assures that 
    \begin{align*}
        \mathbb{E}_{\rho}[\tilde{\rho}_1^2] =& \left( (1 - 2\beta) \frac{d+1}{\alpha} - 2 \right) \rho + \left( \beta^2 \left( \frac{d+1}{\alpha} \right)^2 + (2\beta + 1) \frac{d+1}{\alpha} + 1  + \beta \right) \mathbbm{1}
    \end{align*}
    Using the fact that $\norm{\rho}_{op}, \norm{\mathbbm{1}}_{op} \leq 1$ and the triangle inequality, we can bound the operator norm of the second moment of $\hat{\rho}^2$ by
    \begin{align}
    \label{eqn::second_moment_operator_norm}
\norm{\mathbb{E}_{\rho}\left[\tilde{\rho}_1^2\right]}_{op} &\leq 
        \frac{10d^2}{\alpha^2}
    \end{align}
    for $d \geq 2$.  This now allows us to apply the Matrix-Bernstein-inequality for $A_i = \frac{1}{n}(\tilde{\rho}_i - \rho)$, where we have
    \begin{align*}
        \sigma^2 = \norm{\sum_{i = 1}^n  \mathbb{E}_{\rho}\left[ \left( \frac{1}{n}(\tilde{\rho}_i - \rho) \right)^2 \right]}_{op} = \frac{1}{n} \norm{\mathbb{E}_{\rho}\left[  \tilde{\rho}_1^2 \right]}_{op} \leq 10 \frac{1}{n} \frac{d^2}{\alpha^2}
    \end{align*}
    for $d \geq 2$. Using the fact that $\hat{\rho}$ is unbiased for $\rho$ we get
    \begin{align*}
        R = \norm{A_i}_{op} = \norm{\frac{1}{n}\left( \tilde{\rho}_i - \rho \right)}_{op} \leq 3 \frac{1}{n} \frac{d}{\alpha}
    \end{align*}
    for $d \geq 2$, where we used $\sum_{m = 1}^D H_m \ket{v_m}\bra{v_m} \preceq \mathbbm{1}$. This allows us to bound the error of the estimator $\bar{\rho}_n = \frac{1}{n} \sum_{i = 1}^n \tilde{\rho}_i$ as
    \begin{align}
        \label{eqn::concentration_inequality_qubit_estimator_Operator}
        \mathbb{P}_{\rho}\left( \norm{\bar{\rho}_n - \rho}_{op} \geq t \right) 
        \leq d \exp\left( - \frac{3t^2}{2 \left( 3\sigma^2 + Rt \right)} \right)  \leq d \exp\left( -\frac{1}{4} \frac{t^2 n \alpha^2}{d^2} \right).
    \end{align}
    for $t \in [0, 2], \alpha \in [0,1], d \geq 2$. Lemma~\ref{lem::operator_norm_properties_projection} then assures us that
    \begin{equation*}
        \mathbb{P}_{\rho}\left( \norm{\check{\rho}_n - \rho}_{op} \geq t \right) \leq d \exp\left( -\frac{1}{16} \frac{t^2 n \alpha^2}{d^2} \right) = \epsilon.
    \end{equation*}
    Now, by Proposition~\ref{prop::spectral_thresholding_properties}, it holds with probability at least $1-\epsilon$ that
    \begin{equation*}
        \norm{\hat{\rho}_n - \rho}_F^2 \leq 64rt^2(\epsilon), \hspace{20pt} \text{where } t^2(\epsilon) = 16\log\left( \frac{d}{\epsilon}\right) \frac{d^2}{n\alpha^2}.
    \end{equation*}
    This inequality remains true for all $0 < \epsilon' \leq \epsilon$. Now, let $x(\epsilon') = 64rt^2(\epsilon')$, from which we obtain
    \begin{equation*}
        \epsilon' = d \exp\left( -\frac{1}{1024} \frac{n\alpha^2}{dr^2} x(\epsilon') \right).
    \end{equation*}
    From this we obtain the expected error in Frobenius-norm as
    \begin{align*}
        \mathbb{E}_{\rho}\left[ \norm{\hat{\rho}_n - \rho}_F^2 \right] 
        =& \int \limits_0^{x(\epsilon)} \mathbb{P}_{\rho}\left( \norm{\hat{\rho}_n - \rho}_F^2 \geq x \right) dx+ \int \limits_{x(\epsilon)}^{\infty} \mathbb{P}_{\rho}\left( \norm{\hat{\rho}_n - \rho}_F^2 \geq x \right) dx
        \\
        \leq&
        3072 \frac{rd^2}{n \alpha^2} \log\left(\frac{d}{\epsilon}\right).
    \end{align*}
    Furthermore, Proposition~\ref{prop::spectral_thresholding_properties} also assures us that with probability at least $1-\epsilon$ we have
    \begin{equation*}
        \rank{\hat{\rho}_n}  =r.
    \end{equation*}
    as long as $\lambda_{min}(\rho) \geq 6t(\epsilon)$.
\end{proof}

\subsection{Technical proofs for the proof of Theorem~\ref{thm::upper_bound_qLS_qudits}}
\label{sec::technical_proofs_for_upper_bound}

\begin{lemma}
\label{lem::Expectation_of_measurement_MUB}
    Let $M_{\delta}$ be the measurement defined in~\eqref{defn::measurement_operators_qls_mixed_MUB} and $H$ as in~\eqref{eqn::duality_2_desing_MUB}. Then it holds
    \begin{align}
    \mathbb{E}_{\rho}[H_m] = \frac{1}{d+1} \left( \alpha\Tr\left[\ket{v_m}\bra{v_m} \rho \right] + \beta \right)
\end{align}
for $\beta = (1+e^{\frac{\delta}{2}})^{-1} \in (0,\frac{1}{2})$ and $\alpha = \tanh(\delta/4)$.
\end{lemma}
\begin{proof}
    Let $m \in \{1,...,D\}$ and $k_0$ and $b_0$ be the unique natural numbers such that $m = db_0 + k_0$. Using the equivalence~\eqref{eqn::duality_2_desing_MUB}, for the expected value of $H_m$ we have
\begin{align*}
    \mathbb{E}_{\rho}[H_m] &= \sum_{h} h_m \mathbb{P}\big( H_m = h_m \big)
    \\
    &= \sum_{\{h | h_m = 1\}} h_m \mathbb{P}\big( F(R^M) = h_m \big)
    \\
    &= \sum_{\{\omega | \omega_{k_0} = 1\}} \omega_{k_0} \mathbb{P}\big( R^M = (\omega, b_0) \big)
    \\
    &= \sum_{\{\omega | \omega_{k_0}^{(b_0)} = 1\}} \Tr \bigg[  \frac{1}{d+1} \left(\frac{e^{\frac{\delta}{2}}}{e^{\frac{\delta}{2}}+1} \right)^d \sum_{k = 1}^d e^{-\frac{\delta}{2}\norm{\omega - e_k}_1} \ket{e_k^{(b_0)}}\bra{e_k^{(b_0)}} \rho \bigg]
    \\
    \begin{split}
        &= \frac{1}{d+1} \left(\frac{e^{\frac{\delta}{2}}}{e^{\frac{\delta}{2}}+1} \right)^d \Bigg[ \sum_{k \neq k_0} \sum_{\{\omega | \omega_{k_0}^{(b_0)} = 1\}} e^{-\frac{\delta}{2}\norm{\omega - e_k}_1} \Tr[\ket{e_k^{(b_0)}}\bra{e_k^{(b_0)}} \rho]
        \\
        &\qquad \qquad \qquad \qquad \quad + \sum_{\{\omega | \omega_{k_0}^{(b_0)} = 1\}} e^{-\frac{\delta}{2}\norm{\omega - e_{k_0}}_1}\Tr[\ket{e_{k_0}^{(b_0)}}\bra{e_{k_0}^{(b_0)}} \rho] \Bigg].
    \end{split}
\end{align*}
Let us now consider the sums in each term separately. For the first part it, let $\omega \in \{0,1\}^d$ such that $ \omega_{k_0}^{(b_0)} = 1$. Since $k \neq k_0$, the vector $\omega - e_k$ differs in at least one entry, namely the $k_0$-th. The difference $\norm{\omega - e_k}_1$ is then only dependent on the $d-1$ other entries of the vector $\omega$. Since there are an equal amount of vectors $\omega$ where these entries are $0$ and $1$, we can rewrite the sum as 
\begin{align*}
    \sum_{\{\omega | \omega_{k_0}^{(b_0)} = 1\}} e^{-\frac{\delta}{2}\norm{\omega - e_k}_1} = e^{-\frac{\delta}{2}} \sum_{j = 0}^{d-1} \binom{d-1}{j} (e^{-\frac{\delta}{2}})^j = e^{-\frac{\delta}{2}} \left( e^{-\frac{\delta}{2}} +1 \right)^{d-1}.
\end{align*}
Similarly, if $\omega_{k_0}^{(b_0)} = 1$, then the $k_0$-th entry of $z^{(b_0)} - e_k$ is $0$ and so we can rewrite the sum as 
\begin{align*}
    \sum_{\{\omega | \omega_{k_0}^{(b_0)} = 1\}} e^{-\frac{\delta}{2}\norm{\omega - e_{k_0}}_1} = \sum_{j = 0}^{d-1} \binom{d-1}{j} (e^{-\frac{\delta}{2}})^j = \left( e^{-\frac{\delta}{2}} +1 \right)^{d-1}.
\end{align*}
This allows us to rewrite the above sum as
\begin{align*}
    \mathbb{E}_{\rho}[H_m] =& \frac{1}{d+1} \bigg(\frac{e^{\frac{\delta}{2}}}{e^{\frac{\delta}{2}}+1}  \bigg)^d \Bigg[\sum_{k \neq k_0} e^{-\frac{\delta}{2}} \left( e^{-\frac{\delta}{2}} +1 \right)^{d-1} \Tr[\ket{e_k^{(b_0)}}\bra{e_k^{(b_0)}} \rho]
    \\
    & \qquad \qquad \qquad \qquad+ \left( e^{-\frac{\delta}{2}} +1 \right)^{d-1} \Tr[\ket{e_{k_0}^{(b_0)}}\bra{e_{k_0}^{(b_0)}} \rho] \Bigg].
    \intertext{Using the fact that the $\ket{e_k^{(b_0)}}$ form an orthonormal basis of $\mathbb{C}^d$ for fixed $b_0$, we may write $\sum_{k \neq k_0} \ket{e_k^{(b_0)}}\bra{e_k^{(b_0)}} = \mathbbm{1} - \ket{e_{k_0}^{(b_0)}}\bra{e_{k_0}^{(b_0)}}$ and further simplify the sum as}
    = &\frac{1}{d+1} \frac{1}{1 + e^{-\frac{\delta}{2}}}  \bigg[ e^{-\frac{\delta}{2}} \left( 1 - \Tr\left[ \ket{e_{k_0}^{(b_0)}}\bra{e_{k_0}^{(b_0)}} \rho \right] \right) + \Tr\left[ \ket{e_{k_0}^{(b_0)}}\bra{e_{k_0}^{(b_0)}} \rho \right] \bigg]
    \\
    = &\frac{1}{d+1} \bigg[ \frac{1 - e^{-\frac{\delta}{2}}}{1+ e^{-\frac{\delta}{2}}}  \Tr\left[ \ket{e_{k_0}^{(b_0)}}\bra{e_{k_0}^{(b_0)}} \rho \right] + \frac{e^{-\frac{\delta}{2}}}{1+e^{-\frac{\delta}{2}}}\bigg]
    \\
    = &\frac{1}{d+1} \bigg[ \frac{1 - e^{-\frac{\delta}{2}}}{1+ e^{-\frac{\delta}{2}}}  \Tr\left[ \ket{v_m}\bra{v_m} \rho \right] + \frac{e^{-\frac{\delta}{2}}}{1+e^{-\frac{\delta}{2}}}\bigg]
    \\
    = &\frac{1}{d+1} \bigg[ \frac{ e^{\frac{\delta}{2}}-1}{e^{\frac{\delta}{2}} + 1}  \Tr\left[\ket{v_m}\bra{v_m} \rho \right] + \frac{1}{e^{\frac{\delta}{2}} + 1} \bigg]
    \\
    = &\frac{1}{d+1} \left[ \alpha \Tr\left[\ket{v_m}\bra{v_m} \rho \right] + \beta \right]
\end{align*}
using again that $m$ may is uniquely determined by $b_0$ and $k_0$. 
\end{proof}

\begin{lemma}
\label{lem::expectation_H_m_v_m_MUB}
    Let $\ket{v_m}$ be as in~\eqref{eqn::definition_v_m}, $M_{\delta}$ be the measurement defined in~\eqref{defn::measurement_operators_qls_mixed_MUB} and $H$ as in~\eqref{eqn::duality_2_desing_MUB}. Then it holds
    \begin{equation*}
        \mathbb{E}_{\rho}\left[ \sum_{m = 1}^D H_m \ket{v_m}\bra{v_m} \right] = \frac{\alpha}{d+1} \rho + \left(\frac{\alpha}{d+1} + \beta \right) \mathbbm{1}
    \end{equation*}
\end{lemma}
\begin{proof}
    We calculate
\begin{align}
    \mathbb{E}_{\rho}\left[ \sum_{m = 1}^D H_m \ket{v_m}\bra{v_m} \right] =& \sum_{m = 1}^D \mathbb{E}_{\rho}[H_m] \ket{v_m}\bra{v_m} \notag
    \\
    =& \sum_{m = 1}^D \frac{1}{d+1} \left( \alpha \Tr\left[ \rho \ket{v_m}\bra{v_m} \right] + \beta \right) \ket{v_m}\bra{v_m} \notag
    \\
    =& \frac{\alpha}{d+1} \sum_{m = 1}^D \Tr\left[ \rho \ket{v_m}\bra{v_m} \right]\ket{v_m}\bra{v_m} + \frac{\beta}{d+1} \sum_{m = 1}^D \ket{v_m}\bra{v_m} \notag
    \intertext{Using the $2$-design properties~\eqref{eqn::2design_property} and~\eqref{eqn::ONB_property} of $\ket{v_m}\bra{v_m}$ and $\Tr[\rho] = 1$, we can simplify this sum to be}
    =& \frac{\alpha}{d+1} (\rho + \mathbbm{1}) + \frac{\beta}{d+1} (d+1) \mathbbm{1} \notag
    \\
    \label{eqn::expectation_hm_vm}
    =& \frac{\alpha}{d+1} \rho + \left( \frac{\alpha}{d+1} + \beta \right) \mathbbm{1}.
\end{align}
\end{proof}

\begin{lemma}
\label{lem::second_moment_rho_hat_MUB}
    Let 
    \begin{equation*}
        \tilde{\rho} = \frac{d+1}{\alpha} \sum_{m = 1}^D H_m \ket{v_m}\bra{v_m} - \left( 1+ \frac{d+1}{\alpha}\beta \right) \mathbbm{1},
    \end{equation*}
    be defined as in~\eqref{eqn::qLS_est_before_proj_single}. Then it holds
    \begin{align*}
        \mathbb{E}_{\rho}[\tilde{\rho}_1^2] =& \left( (1 - 2\beta) \frac{d+1}{\alpha} - 2 \right) \rho + \left( \beta^2 \left( \frac{d+1}{\alpha} \right)^2 + (2\beta + 1) \frac{d+1}{\alpha} + 1  + \beta \right) \mathbbm{1}
    \end{align*}
\end{lemma}

\begin{proof}
    We calculate
\begin{alignat}{1}
    \mathbb{E}_{\rho}[\tilde{\rho}^2] =& \left(\frac{d+1}{\alpha}\right)^2 \mathbb{E}_{\rho}\Bigg[ \bigg(\sum_{m = 1}^D H_m \ket{v_m}\bra{v_m} \bigg)^2 \Bigg] \notag
    \\
    &- 2 \frac{d+1}{\alpha} \left( 1 + \frac{d+1}{\alpha}\beta \right) \mathbb{E}_{\rho}\Bigg[ \sum_{m = 1}^D H_m \ket{v_m}\bra{v_m} \Bigg] \notag
    \\
    \label{eqn::second_moment_calculations}
    &+ \left( 1 + \frac{d+1}{\alpha}\beta \right)^2 \mathbbm{1}.
\end{alignat}
We can decompose the first expectation into two separate sums using the fact that $H_m^2 = H_m$ and $\bra{v_m}\ket{v_m} = 1$ in the following way
\begin{align}
\label{eqn::expectation_sqaured_hm}
\begin{split}
    \mathbb{E}_{\rho}\Bigg[ \bigg(\sum_{m = 1}^D H_m \ket{v_m}\bra{v_m} \bigg)^2 \Bigg] =& \mathbb{E}_{\rho}\Bigg[ \sum_{m = 1}^D H_m \ket{v_m}\bra{v_m} \Bigg] 
    \\
    &+ \mathbb{E}_{\rho}\Bigg[ \sum_{m = 1}^D \sum_{l \neq m} H_m H_l \ket{v_m}\bra{v_m}\ket{v_l}\bra{v_l} \Bigg]
\end{split}
\end{align}

We now study the second summand. Let us again write $m = db_m + k_m$ and $l = db_l + k_l$. By definition of the vector $H$, at least one of the terms $H_m$ and $H_l$ will be $0$ if $b_m \neq b_l$. We may therefore write
\begin{align*}
    &\mathbb{E}_{\rho}\Bigg[ \sum_{m = 1}^D \sum_{l \neq m} H_m H_l \ket{v_m}\bra{v_m}\ket{v_l}\bra{v_l} \Bigg] 
    \\
    = &\sum_{b_m = 1}^{d+1} \sum_{k_m = 1}^d \sum_{k_l \neq k_m}^d \mathbb{E}_{\rho}\left[ W_{k_m} W_{k_l} \right] \ket{e_{k_m}^{(b_m)}}\bra{e_{k_m}^{(b_m)}}\ket{e_{k_l}^{(b_m)}}\bra{e_{k_l}^{(b_m)}}.
\end{align*}
Using the fact that $\bra{e_{k_m}^{(b_m)}}\ket{e_{k_l}^{(b_m)}} = 0$ for $k_m \neq k_l$ due to the orthogonality of the vectors, we see that
\begin{align}
\label{eqn::cross_terms_zero}
    \mathbb{E}_{\rho}\Bigg[ \sum_{m = 1}^D \sum_{l \neq m} H_m H_l \ket{v_m}\bra{v_m}\ket{v_l}\bra{v_l} \Bigg] = 0.
\end{align}
Therefore, we have
\begin{equation}
\label{eqn::expectation_square_is_expectation}
    \mathbb{E}_{\rho}\Bigg[ \bigg(\sum_{m = 1}^D H_m \ket{v_m}\bra{v_m} \bigg)^2 \Bigg] = \mathbb{E}_{\rho}\Bigg[ \sum_{m = 1}^D H_m \ket{v_m}\bra{v_m} \Bigg].
\end{equation}
As such, we obtain
\begin{align*}
    \mathbb{E}_{\rho}\left[ \tilde{\rho}^2 \right] =& \mathbb{E}_{\rho}\Bigg[ \sum_{m = 1}^D H_m \ket{v_m}\bra{v_m} \Bigg] \left( \left( \frac{d+1}{\alpha} \right)^2 - 2 \frac{d+1}{\alpha} \left( 1 + \frac{d+1}{\alpha} \beta \right) \right) 
    \\
    &+ \left( 1 + \frac{d+1}{\alpha} \beta \right)^2 \mathbbm{1}.
    \intertext{Having calculated the expectation already in Lemma~\ref{lem::expectation_H_m_v_m_MUB}, we can rewrite this expression as}
    =& \left( (1 - 2\beta) \frac{d+1}{\alpha} - 2 \right) \rho + \left( \beta^2 \left( \frac{d+1}{\alpha} \right)^2 + (2\beta + 1) \frac{d+1}{\alpha} + 1  + \beta \right) \mathbbm{1}
\end{align*}
\end{proof}

\begin{lemma}
\label{lem::operator_norm_properties_projection}
    Let $\bar{\rho} \in \mathbb{C}^{d \times d}$ be a (not-necessarily positive and not necessarily normed) hermitian matrix and $\check{\rho} \in \mathbb{C}^{d \times d}$ be its projection onto the convex set $\mathcal{S}(\mathbb{C}^d)$. Then, it holds
    \begin{equation*}
        \norm{\check{\rho} - \rho}_{op} \leq 2 \norm{\bar{\rho} - \rho}_{op} \hspace{20pt} \text{for all } \rho \in \mathcal{S}(\mathbb{C}^d).
    \end{equation*}
\end{lemma}
\begin{proof}
    Let us start by giving a formula for the projection of $\bar{\rho}$ onto $\mathcal{S}(\mathbb{C}^d)$. We start by noting that the optimal projection $\check{\rho}$ has to be in the same eigenbasis as $\bar{\rho}$ (as shown in \cite{smolin_efficient_2012}) and in that case the problem reduces to the optimal projection of the eigenvalues onto the probability simplex which has  already been solved (see for example \cite{duchi_efficient_2008}). As such for $\bar{\rho} = \bar{U} \bar{D} \bar{U}^*$ with $\bar{D} = \diag\left( \bar{\lambda}_1,...,\bar{\lambda}_d \right)$, its projection onto $\mathcal{S}(\mathbb{C}^d)$ is given by 
    \begin{equation*}
        \check{\rho} = \bar{U} \check{D} \bar{U}^* \hspace{20pt} \text{for } \check{D} = \diag\left( \check{\lambda}_1,...,\check{\lambda}_d \right),
    \end{equation*}
    where $\check{\lambda}_i = \max(\bar{\lambda}_i - \theta, 0)$ for some $\theta$ such that $\sum_{i = 1}^d \check{\lambda}_i = 1$. Since the operator norm and the set $\mathcal{S}(\mathbb{C}^d)$ are unitarily invariant, we can assume that $\rho$ is diagonal and it is equivalent to show 
    \begin{equation*}
        \norm{\check{D} - \rho}_{op} \leq 2 \norm{\bar{D} - \rho}_{op} \hspace{20pt} \text{for all } \rho \in \mathcal{S}(\mathbb{C}^d).
    \end{equation*}
    Now, let $\rho = UDU^*\in \mathcal{S}(\mathbb{C}^d)$ with $D = \diag(\lambda_1,...,\lambda_d)$ such that $\lambda_k \geq 0$ for $k = 1,..,d$ and $\sum_{k = 1}^d\lambda_k = 1$.  Set $\epsilon = \norm{\bar{D} - \rho}_{op}$. Then, we have
    \begin{align*}
        \max_{k = 1,...,d} \left| \check{\lambda}_k - \bar{\lambda}_k \right| &\leq |\theta|
        \intertext{and by Weyl's inequality \cite{tropp_introduction_2015}}
        \max_{k = 1,...,d} \left| \bar{\lambda}_k - \lambda_k \right| &\leq \epsilon.
    \end{align*}
    We will now show that $|\theta| \leq \epsilon$. In order to do this, we will first write $\bar{\lambda}_k = \lambda_k + e_k$, where $|e_k| \leq \epsilon$. Assume now that $\theta > \epsilon$. Then we have
    \begin{equation*}
        \bar{\lambda}_k - \theta = \lambda_k + e_k - \theta \leq \lambda_k.
    \end{equation*}
    As such, we also have $\check{\lambda}_k \leq \max(\lambda_k - \theta, 0) \leq \lambda_k$. However, there exists at least $k_0$ such that $\lambda_{k_0} > 0$. Since $\theta > \epsilon \geq 0$ we have $\check{\lambda}_{k_0} < \lambda_{k_0}$ from which we obtain
    \begin{equation*}
        1 = \sum_{k = 1}^d \check{\lambda}_k < \sum_{k = 1}^d \lambda_k = 1,
    \end{equation*}
    which is a contradiction. This shows $\theta < \epsilon$. If we now assume that $\theta < -\epsilon$ we have 
    \begin{equation*}
        \check{\lambda}_k = \max(\bar{\lambda}_k - \theta, 0) = \bar{\lambda}_k - \theta = \lambda_k + e_k - \theta > \lambda_k.
    \end{equation*}
    Here we obtain
    \begin{equation*}
        1 = \sum_{k = 1}^d \check{\lambda}_k > \sum_{k = 1}^d \lambda_k = 1,
    \end{equation*}
    which is again a contradiction. As such we have $|\theta| \leq \epsilon$. By the triangle inequality we have
    \begin{equation*}
        \norm{\check{D} - \rho}_{op} \leq \norm{\check{D} - \bar{D}}_{op} + \norm{\bar{D} - \rho}_{op} \leq |\theta| + \epsilon \leq 2 \epsilon \leq 2 \norm{\bar{D} - \rho}_{op}.
    \end{equation*}
    which shows the claim.
\end{proof}

\subsection{Technical proofs for the spectral thresholding algorithm}
The spectral thresholding algorithm for an estimator $\check{\rho}_n$ is given as follows. It is essentially the same as in \cite{butucea_spectral_2015} and \cite{lahiry_minimax_2024}. We include both it and its properties again for the readers convenience.

\begin{algorithm}
\SetKwInOut{Input}{Input}
\SetKwInOut{Output}{Output}
    \Input{Threshold $2t$ and $\check{\rho}_n = \check{U}_n\check{D}_n\check{U}_n^*$, with $\check{D}_n = \diag(\check{\lambda}_1,...,\check{\lambda}_d)$ such that $\check{\lambda}_1 > ... > \check{\lambda}_{\check{r}}$}
    \Output{$\hat{\rho}_n$ an estimator of $\rho$}
    $l = 0$\;
    $\hat{\lambda}_j^{(0)} = \check{\lambda}_j$ for $j = 1,...,d$\;
    \For{$l = 1,...,d$}{
        \eIf{$\hat{\lambda}_{d-l+1}^{(l-1)} > 2t$}{
            STOP
        }{
        $\hat{\lambda}_{d-l+1}^{(l)} = 0$\;
        $\hat{\lambda}_{j}^{(l)} = \hat{\lambda}_{j}^{(l-1)} + \frac{1}{d-l}\hat{\lambda}_{d-l+1}^{(l-1)}$ for $j = 1,...,d-l$\;
        }
    }
    $\hat{D}_n = \diag\left(\hat{\lambda}_1^{(d)},...,\hat{\lambda}_d^{(d)}\right)$\;
    $\hat{\rho}_n =  \check{U}_n\hat{D}_n\check{U}_n^*$\;
\caption{Spectral thresholding algorithm}
\label{alg::spectral_thresholding}
\end{algorithm}

Algorithm~\ref{alg::spectral_thresholding} assures that if the state $\rho$ is low rank, so is the estimator $\hat{\rho}_n$. It does so by iterating over the ordered eigenvalues of $\check{\rho}_n$ from the lowest to the largest. If the current iteration of the eigenvalue is smaller than the threshold, the algorithm sets this eigenvalue to zero and distributes the mass evenly among the remaining larger values. It does so until it reaches an eigenvalue that is above the threshold. We have the following two results.

\begin{lemma}
\label{lem::eigenvalue_properties_algorithm}
    Let $\hat{\rho}_n = \mathcal{A}(\check{\rho}_n)$ be the output of the spectral thresholding algorithm~\ref{alg::spectral_thresholding} with threshold $2t$ for the input being the estimator~\eqref{eqn::qLS_est_qudits} of a state $\rho$. Denote by $\check{\lambda}_1 > ... > \check{\lambda}_d$ the eigenvalues of $\check{\rho}_n$ and by $\hat{\lambda}_1 > ... > \hat{\lambda}_d$ the eigenvalues of $\hat{\rho}_n$. Then, for all $j = 1,...,d$ it holds
    \begin{equation*}
        \left| \hat{\lambda}_j - \check{\lambda}_j \right| \leq 4t.
    \end{equation*}
\end{lemma}

\begin{proof}
    This inequality holds obviously true for all $j > \hat{r}$, as in that case $\hat{\lambda}_j = 0$ and $\check{\lambda}_j < 2t$. Let us now consider the case $j \leq \hat{r}$. In that case we note that in the second last iteration (which is the $d - \hat{r}-1$-st), for all $j = 1,...., \hat{r}+1$, we can prove that
    \begin{equation*}
        \hat{\lambda}_{j}^{(d - \hat{r}-1)} = \check{\lambda}_j+ \frac{1}{\hat{r} + 1}\sum_{l = \hat{r} + 2}^{d} \check{\lambda}_l,
    \end{equation*}
    i.e., the mass of the smallest $d - \hat{r} - 2$ smallest eigenvalues is evenly distributed among the first $\hat{r}+1$ eigenvalues. Since the algorithm does not stop in the $d-\hat{r}$-th iteration, for $j = \hat{r}+1$ we have
    \begin{equation*}
        \hat{\lambda}_{\hat{r}+1}^{(d - \hat{r}-1)} < 2t,
    \end{equation*}
    which shows that
    \begin{equation*}
        \frac{1}{\hat{r} + 1}\sum_{l = \hat{r} + 2}^{d} \check{\lambda}_l < 2t
    \end{equation*}
    and in particular, that we have
    \begin{equation*}
        \hat{\lambda}_{j}^{(d - \hat{r}-1)} < \check{\lambda}_j+ 2t \hspace{10pt} \text{for all } j = 1,...,\hat{r} + 1 \hspace{10pt} \text{ and } \hspace{10pt} \hat{\lambda}_{\hat{r} + 1}^{(d - \hat{r}-1)} <  2t.
    \end{equation*}
    Heuristically, the algorithm cannot put more mass than $2t$ on any eigenvalue. If it did, the algorithm would stop at iteration $d- \hat{r}$ and the rank of the estimator would be $\hat{r} +1$. Now, in the last iteration, the algorithm distributes the mass currently on $\hat{\lambda}_{\hat{r} + 1}^{(d - \hat{r})}$ on the remaining $\hat{r}$ eigenvalues, i.e. for all $j = 1,...,\hat{r}$
    \begin{equation*}
        \hat{\lambda}_{j}^{(d - \hat{r})} = \hat{\lambda}_{j}^{(d - \hat{r}-1)} + \hat{\lambda}_{\hat{r}+1}^{(d-\hat{r}-1)} < \check{\lambda}_j+ 2t + 2t < \check{\lambda}_j + 4t.
    \end{equation*}
    Since the algorithm then stops at, we have 
    \begin{equation*}
        \hat{\lambda}_j < \check{\lambda}_j + 4t \hspace{20pt} \text{for all } j = 1,...,\hat{r}.
    \end{equation*}
\end{proof}

The following result has already been proven for similar estimators in \cite{butucea_spectral_2015} and \cite{lahiry_minimax_2024}. We however include the proof again for the readers convenience.

\begin{proposition}
\label{prop::spectral_thresholding_properties}
    Let $\epsilon > 0$ and 
    \begin{equation*}
        t^2 = t^2(\epsilon) = 16 \log\left( \frac{d}{\epsilon} \right) \frac{d^2}{n\alpha^2}.
    \end{equation*}
    Furthermore, let $\hat{\rho}_n = \mathcal{A}(\check{\rho}_n)$ be the output of the spectral thresholding algorithm~\ref{alg::spectral_thresholding} with threshold $2t$ for the input being the estimator~\eqref{eqn::qLS_est_qudits} of $\rho$. Then we have
    \begin{align*}
        \mathbb{P}_{\rho}\left( \norm{\hat{\rho}_n - \rho}_F^2 \leq 64rt^2 \right) &\geq 1- \epsilon.
        \intertext{If we additionally assume that $\lambda_{min}(\rho) \geq 6t$, then}
        \mathbb{P}_{\rho}\left( \rank(\hat{\rho}_n) = \rank(\rho) \right) &\geq 1 - \epsilon
    \end{align*}
\end{proposition}

\begin{proof}
    By equation~\eqref{eqn::concentration_inequality_qubit_estimator_Operator} we have
\begin{equation*}
    \mathbb{P}_{\rho}\left( \norm{\check{\rho}_n - \rho}_{op} \leq t \right) \geq 1 - d \exp\left( -\frac{1}{16} \frac{t^2 n \alpha^2}{d^2} \right) = 1- \epsilon.
\end{equation*}
Let us now write
\begin{align*}
    \rho &= UDU^* \qquad \text{with } D = \diag(\lambda_1,...,\lambda_r,0,...,0)
    \\
    \check{\rho}_n &= \check{U}_n \check{D}_n \check{U}_n^* \quad\text{with } \check{D}_n = \diag(\check{\lambda}_1,...,\check{\lambda}_{\check{r}},0,...,0)
    \\
    \hat{\rho}_n &= \check{U}_n \hat{D}_n \check{U}_n^* \quad\text{with } \check{D}_n = \diag(\hat{\lambda}_1,...,\hat{\lambda}_{\check{r}},0,...,0).
\end{align*}
Note that by the properties of the algorithm we have $\hat{\lambda}_j = \check{\lambda}_j + \frac{1}{\hat{r}}\sum_{k > \hat{r}} \check{\lambda}_k$. Under the event that $\norm{\check{\rho}_n - \rho}_{op} \leq t$, which happens with probability at least $1 - \epsilon$, we have by Weyl's perturbation Theorem \cite{tropp_acm_2024}
\begin{equation}
\label{eqn::spectral_thresholding_eigenvalue_property}
    \max_{j = 1,...,\max\{r, \check{r}\}} |\lambda_j - \check{\lambda}_j| \leq \norm{\check{\rho}_n - \rho}_{op} \leq t.
\end{equation}
We will now prove that in this event it holds $\hat{r} \leq r$. Assume first that $\hat{r} > r$. Then it holds $\hat{\lambda}_{\hat{r}} > 2t$. However, for all $j = 1,...,\hat{r}$ we have
\begin{align*}
    \left|\hat{\lambda}_j - \lambda_j \right| &= \left|\check{\lambda}_j - \lambda_j - \frac{1}{\hat{r}} \sum_{k > \hat{r}} \check{\lambda}_k\right| \\
    &= \left| \check{\lambda}_j - \lambda_j - \frac{1}{\hat{r}} \sum_{k \leq \hat{r}} (\lambda_k - \check{\lambda}_k) \right| 
    \\
    &\leq \left| \check{\lambda}_j - \lambda_j \right| + \frac{1}{\hat{r}} \sum_{k \leq \hat{r}} \left| \lambda_k - \check{\lambda}_k \right| \leq 2t,
\end{align*}
which shows that $\hat{r} > r$ cannot hold. Let us now analyze the error of the estimator. By the triangle inequality we have
\begin{equation*}
    \norm{\hat{\rho}_n - \rho}_F \leq \norm{\check{U}_n\hat{D}_n \check{U}_n^* - \check{U}_n D \check{U}_n^*}_F + \norm{\check{U}_n D \check{U}_n^* - UDU^*}_F.
\end{equation*}
We now analyze both summands separately. Since $\hat{r} \leq r$, for the first one we have
\begin{equation*}
    \norm{\check{U}_n\hat{D}_n \check{U}_n^* - \check{U}_n D \check{U}_n^*}_F = \norm{\hat{D}_n - D}_F = \left( \sum_{j = 1}^r \left( \hat{\lambda}_j - \lambda_j \right)^2 \right)^{\frac{1}{2}} \leq 5 \sqrt{r} t
\end{equation*}
using Lemma~\ref{lem::eigenvalue_properties_algorithm}. For the second summand, we first bound its operator norm in the following way.
\begin{align*}
    \norm{\check{U}_n D \check{U}_n^* - UDU^*}_{op} &\leq \norm{\check{U}_n D \check{U}_n^* - \check{U}_n \check{D}_n \check{U}_n^*}_{op} + \norm{\check{U}_n \check{D}_n \check{U}_n^* - UDU^*}_{op}
    \\
    &\leq \norm{D - \check{D}_n}_{op} + \norm{\check{\rho}_n - \rho}_{op} \leq 2t.
\end{align*}
Since $\check{U}_n D \check{U}_n^*$ and $UDU^*$ are both matrices of at most rank $r$, the difference has rank at most $2r$ which allows us to bound
\begin{equation*}
    \norm{\check{U}_n D \check{U}_n^* - UDU^*}_{F} \leq 2 \sqrt{2r} t.
\end{equation*}    
Finally, we obtain
\begin{equation*}
    \norm{\hat{\rho}_n - \rho}_F \leq 8 \sqrt{r} t
\end{equation*}
in the event that $\norm{\check{\rho}_n - \rho}_{op} \leq t$, which shows
\begin{equation*}
    \mathbb{P}_{\rho}\left( \norm{\hat{\rho}_n - \rho}_F \leq 8\sqrt{r}t \right) \geq 1- \epsilon.
\end{equation*}
Let us now additionally assume that $\lambda_{min}(\rho) > 6t$. Assume that in this case $\hat{r} < r$. Then we have $\hat{\lambda}_r = 0$. However, it holds
\begin{align*}
    \hat{\lambda}_r = \lambda_r - \left|\lambda_r - \check{\lambda}_r \right| - \left| \check{\lambda}_r - \hat{\lambda} \right| \leq 6t - 3t - t = 2t,
\end{align*}
which is again a contradiction. Together with the results $\hat{r} \leq r$, we have the following result.
\begin{equation*}
    \mathbb{P}_{\rho}\left( \rank(\hat{\rho}_n) = \rank(\rho) \right) \geq 1 - \epsilon.
\end{equation*}
This completes the rank property of the estimator $\hat{\rho}_n$. 
\end{proof}

\newpage
\section{Proof of Proposition~\ref{prop::upper_bound_qLS_qudits_MUM}}
\label{sec::proofs_upper_bound_MUMs}
The proof of Proposition~\ref{prop::upper_bound_qLS_qudits_MUM} essentially follows the proof of Theorem~\ref{thm::upper_bound_qLS_qudits}. As such, we first state three analogous results of the Lemmas~\ref{lem::Expectation_of_measurement_MUB},~\ref{lem::expectation_H_m_v_m_MUB} and~\ref{lem::second_moment_rho_hat_MUB} which are given by the Lemmas~\ref{lem::Expectation_of_measurement_MUM},~\ref{lem::expectation_H_m_v_m_MUM} and~\ref{lem::second_moment_rho_hat_MUM}. The proof of Lemma~\ref{lem::Expectation_of_measurement_MUM} is almost identical to the proof of Lemma~\ref{lem::Expectation_of_measurement_MUB}. The proof of Lemma~\ref{lem::expectation_H_m_v_m_MUM} is also analogous to the proof of Lemma~\ref{lem::expectation_H_m_v_m_MUM}. However, instead of the $2$-design properties of MUBs~\eqref{eqn::2design_property} and~\eqref{eqn::ONB_property} we make use of the following linear inversion and completeness properties (see \cite{kalev_mutually_2014}): Let $X$ be a hermitian matrix. For a set of mutually unbiased measurements $(V_m)_{m = 1,...,D}$ with efficiency parameter $\kappa$ it holds
\begin{align}
\label{eqn::2_design_property_MUM}
    \sum_{m = 1}^D \Tr\left[ XV_m \right] V_m &= \frac{\kappa d -1}{d-1} X + \frac{d - \kappa}{d-1} \Tr[X] \mathbbm{1}
    \intertext{and}
\label{eqn::ONB_property_MUM}
    \sum_{m = 1}^D V_m &= (d+1) \mathbbm{1}.
\end{align}
    Finally, the proof of Lemma~\ref{lem::second_moment_rho_hat_MUM} is similar to the proof of Lemma~\ref{lem::second_moment_rho_hat_MUB} with the additional difficulty that due to the missing orthogonality of measurement operators within a bases we have no analogous result to equation~\ref{eqn::expectation_square_is_expectation}. As such, we must treat some cross terms more carefully. With these results we can then show that for the estimator $\tilde{\rho}$ based on mutually unbiased measurements given in~\eqref{eqn::qLS_est_before_proj_MUM} we have
    \begin{align*}
        \mathbb{E}_{\rho}\left[ \bar{\rho} \right] = \rho \hspace{20pt} \text{and} \hspace{20pt} \norm{\mathbb{E}_{\rho}\left[ \bar{\rho}^2\right]}_{op} \leq \frac{10d^2}{\alpha^2 \kappa^2}.
    \end{align*}
    To proof Proposition~\ref{prop::upper_bound_qLS_qudits_MUM} we then follow the remainder of the proof of Theorem~\ref{thm::upper_bound_qLS_qudits} from~\eqref{eqn::second_moment_operator_norm} replacing $\alpha$ by $\alpha \kappa$. Let us now state the Lemmas~~\ref{lem::Expectation_of_measurement_MUM},~\ref{lem::expectation_H_m_v_m_MUM} and~\ref{lem::second_moment_rho_hat_MUM}.

\begin{lemma}
\label{lem::Expectation_of_measurement_MUM}
    Let $M_{\delta}$ be the measurement defined in~\eqref{defn::measurement_operators_qls_mixed_MUM} and $H$ as in~\eqref{eqn::duality_2_desing_MUB}. Then it holds
    \begin{align}
    \mathbb{E}_{\rho}[H_m] = \frac{1}{d+1} \bigg[ \alpha \Tr\left[V_m \rho \right] + \beta \bigg],
\end{align}
wherer $\alpha = \tanh(\delta/4)$ and $\beta = (e^{\frac{\delta}{2}} + 1)^{-1} \in (0,\frac{1}{2})$.
\end{lemma}

\begin{proof}
    Let $m \in \{1,...,D\}$ and $k_0$ and $b_0$ be the unique natural numbers such that $m = db_0 + k_0$. Using the equivalence~\eqref{eqn::duality_2_desing_MUB}, for the expected value of $H_m$ we have
    \begin{align*}
        \mathbb{E}_{\rho}[H_m] &= \sum_{h} h_m \mathbb{P}\big( H_m = h_m \big)
        \\
        &= \sum_{\{h | h_m = 1\}} h_m \mathbb{P}\big( F(R^M) = h_m \big)
        \\
        &= \sum_{\{\omega | \omega_{k_0} = 1\}} \omega_{k_0} \mathbb{P}\big( R^M = (\omega, b_0) \big)
        \\
        &= \sum_{\{\omega | \omega_{k_0}^{(b_0)} = 1\}} \Tr \bigg[  \frac{1}{d+1} \left(\frac{e^{\frac{\delta}{2}}}{e^{\frac{\delta}{2}}+1} \right)^d \sum_{k = 1}^d e^{-\frac{\delta}{2}\norm{\omega - e_k}_1} P_k^{(b_0)} \rho \bigg]
        \\
        \begin{split}
            &= \frac{1}{d+1} \left(\frac{e^{\frac{\delta}{2}}}{e^{\frac{\delta}{2}}+1} \right)^d \Bigg[ \sum_{k \neq k_0} \sum_{\{\omega | \omega_{k_0}^{(b_0)} = 1\}} e^{-\frac{\delta}{2}\norm{\omega - e_k}_1} \Tr[P_k^{(b_0)} \rho]
            \\
            &\qquad \qquad \qquad \qquad \quad + \sum_{\{\omega | \omega_{k_0}^{(b_0)} = 1\}} e^{-\frac{\delta}{2}\norm{\omega - e_{k_0}}_1}\Tr[P_{k_0}^{(b_0)} \rho] \Bigg].
        \end{split}
    \end{align*}
    Let us now consider the sums in each term separately. For the first part it, let $\omega \in \{0,1\}^d$ such that $ \omega_{k_0}^{(b_0)} = 1$. Since $k \neq k_0$, the vector $\omega - e_k$ differs in at least one entry, namely the $k_0$-th. The difference $\norm{\omega - e_k}_1$ is then only dependent on the $d-1$ other entries of the vector $\omega$. Since there are an equal amount of vectors $\omega$ where these entries are $0$ and $1$, we can rewrite the sum as 
    \begin{align*}
        \sum_{\{\omega | \omega_{k_0}^{(b_0)} = 1\}} e^{-\frac{\delta}{2}\norm{\omega - e_k}_1} = e^{-\frac{\delta}{2}} \sum_{j = 0}^{d-1} \binom{d-1}{j} (e^{-\frac{\delta}{2}})^j = e^{-\frac{\delta}{2}} \left( e^{-\frac{\delta}{2}} +1 \right)^{d-1}.
    \end{align*}
    Similarly, if $\omega_{k_0}^{(b_0)} = 1$, then the $k_0$-th entry of $z^{(b_0)} - e_k$ is $0$ and so we can rewrite the sum as 
    \begin{align*}
        \sum_{\{\omega | \omega_{k_0}^{(b_0)} = 1\}} e^{-\frac{\delta}{2}\norm{\omega - e_{k_0}}_1} = \sum_{j = 0}^{d-1} \binom{d-1}{j} (e^{-\frac{\delta}{2}})^j = \left( e^{-\frac{\delta}{2}} +1 \right)^{d-1}.
    \end{align*}
    This allows us to rewrite the above sum as
    \begin{align*}
        \mathbb{E}_{\rho}[H_m] =& \frac{1}{d+1} \bigg(\frac{e^{\frac{\delta}{2}}}{e^{\frac{\delta}{2}}+1}  \bigg)^d \Bigg[\sum_{k \neq k_0} e^{-\frac{\delta}{2}} \left( e^{-\frac{\delta}{2}} +1 \right)^{d-1} \Tr[P_k^{(b_0)} \rho]
        \\
        & \qquad \qquad \qquad \qquad+ \left( e^{-\frac{\delta}{2}} +1 \right)^{d-1} \Tr[P_{k_0}^{(b_0)} \rho] \Bigg].
        \intertext{Using the fact that the $P_k^{(b_0)}$ form a quantum measurement for fixed $b_0$, we may write $\sum_{k \neq k_0} P_k^{(b_0)} = \mathbbm{1} - P_{k_0}^{(b_0)}$ and further simplify the sum as}
        = &\frac{1}{d+1} \frac{1}{1 + e^{-\frac{\delta}{2}}}  \bigg[ e^{-\frac{\delta}{2}} \left( 1 - \Tr\left[ P_{k_0}^{(b_0)} \rho \right] \right) + \Tr\left[ P_{k_0}^{(b_0)} \rho \right] \bigg]
        \\
        = &\frac{1}{d+1} \bigg[ \frac{1 - e^{-\frac{\delta}{2}}}{1+ e^{-\frac{\delta}{2}}}  \Tr\left[ P_{k_0}^{(b_0)} \rho \right] + \frac{e^{-\frac{\delta}{2}}}{1+e^{-\frac{\delta}{2}}}\bigg]
        \\
        = &\frac{1}{d+1} \bigg[ \frac{1 - e^{-\frac{\delta}{2}}}{1+ e^{-\frac{\delta}{2}}}  \Tr\left[ V_m \rho \right] + \frac{e^{-\frac{\delta}{2}}}{1+e^{-\frac{\delta}{2}}}\bigg]
        \\
        = &\frac{1}{d+1} \bigg[ \frac{ e^{\frac{\delta}{2}}-1}{e^{\frac{\delta}{2}} + 1}  \Tr\left[V_m \rho \right] + \frac{1}{e^{\frac{\delta}{2}} + 1} \bigg]
        \\
        =& \frac{1}{d+1} \bigg[ \alpha \Tr\left[V_m \rho \right] + \beta \bigg].
    \end{align*}
    using again that $m$ may is uniquely determined by $b_0$ and $k_0$. 
\end{proof}

\begin{lemma}
\label{lem::expectation_H_m_v_m_MUM}
    Let $(V_m)_{m = 1}^D$ be as in~\eqref{eqn::definition_v_m_MUM}, $M_{\delta}$ be the measurement defined in~\eqref{defn::measurement_operators_qls_mixed_MUM} and $H$ as in~\eqref{eqn::duality_2_desing_MUB}. Then it holds
    \begin{equation*}
        \mathbb{E}_{\rho}\left[ \sum_{m = 1}^D H_m V_m \right] = \frac{\alpha(\kappa d -1)}{d^2 - 1} \rho + \left( \frac{\alpha(d - \kappa)}{d^2 -1} + \beta \right) \mathbbm{1},
    \end{equation*}
    where $\kappa \in (1/d, 1]$ is the efficiency parameter of the MUM, $\alpha = \tanh(\delta/4)$ and $\beta = (e^{\frac{\delta}{2}} + 1)^{-1}$.
\end{lemma}
\begin{proof}
    We calculate
\begin{align}
    \mathbb{E}_{\rho}\left[ \sum_{m = 1}^D H_m \ket{v_m}\bra{v_m} \right] =& \sum_{m = 1}^D \mathbb{E}_{\rho}[H_m] \ket{v_m}\bra{v_m} \notag
    \\
    =& \sum_{m = 1}^D \frac{1}{d+1} \bigg[ \alpha \Tr\left[V_m \rho \right] + \beta \bigg] V_m \notag
    \\
    =& \frac{\alpha}{d+1} \sum_{m = 1}^D \Tr\left[ V_m \rho \right] V_m + \frac{\beta}{d+1} \sum_{m = 1}^D V_m. \notag
    \intertext{Using the linear inversion properties~\eqref{eqn::2_design_property_MUM} and~\eqref{eqn::ONB_property_MUM} of $V_m$ and $\Tr[\rho] = 1$, we can simplify this sum to be}
    \label{eqn::expectation_hm_vm_MUM}
    =& \frac{\alpha(\kappa d -1)}{d^2 - 1} \rho + \left( \frac{\alpha(d - \kappa)}{d^2 -1} + \beta \right) \mathbbm{1}
\end{align}
\end{proof}

\begin{lemma}
\label{lem::second_moment_rho_hat_MUM}
    Let 
    \begin{equation*}
        \tilde{\rho} = \frac{d^2-1}{\alpha(\kappa d - 1)} \sum_{m = 1}^D H_m V_m - \frac{d^2 -1}{\alpha(\kappa d -1)} \left( \frac{\alpha(d-\kappa)}{d^2 - 1} + \beta\right) \mathbbm{1}
    \end{equation*}
    be defined as in~\eqref{eqn::qLS_est_before_proj_MUM} for a singular outcome with $\kappa \in (1/d, 1]$ being the efficiency parameter of the MUM, $\alpha = \tanh(\delta/4)$ and $\beta = (e^{\frac{\delta}{2}} + 1)^{-1}$. Then it holds
    \begin{align*}
        \norm{\mathbb{E}_{\rho}[\tilde{\rho}^2]}_{op} \leq& 10 \frac{d^2}{\alpha^2 \kappa^2}  \hspace{10pt} \text{and} \hspace{10pt} \norm{\tilde{\rho}}_{op} \leq 2 \frac{d}{\alpha \kappa}.
    \end{align*}
\end{lemma}

\begin{proof}
    We first set
    \begin{align*}
        \tau_1 := \frac{d^2-1}{\alpha(\kappa d - 1)} \hspace{20pt} \text{and} \hspace{20pt} \tau_2 = \frac{d^2 -1}{\alpha(\kappa d -1)} \left( \frac{\alpha(d-\kappa)}{d^2 - 1} + \beta\right).
    \end{align*}
    Then, it holds
    \begin{alignat}{1}
        \mathbb{E}_{\rho}[\tilde{\rho}^2] =& \tau_1^2 \mathbb{E}_{\rho}\Bigg[ \bigg(\sum_{m = 1}^D H_m V_m \bigg)^2 \Bigg] - 2 \tau_1 \tau_2 \mathbb{E}_{\rho}\Bigg[ \sum_{m = 1}^D H_m V_m \Bigg] + \tau_2^2 \mathbbm{1}. \label{eqn::second_moment_calculations_MUM}
    \end{alignat}
    Let us consider the first expectation. We can write this term as
    \begin{align}
    \label{eqn::Second_moment_summand_1}
        \mathbb{E}_{\rho}\Bigg[ \bigg(\sum_{m = 1}^D H_m V_m \bigg)^2 \Bigg] = \mathbb{E}_{\rho}\left[ \sum_{m = 1}^D H_m^2 \right] V_m^2 + \mathbb{E}_{\rho}\left[ \sum_{m_1 \neq m_2} H_{m_1}H_{m_2} \right] V_{m_1} V_{m_2}
    \end{align}
    Since $H_m^2 = H_m$ and $0 \lesssim V_m \lesssim \mathbbm{1}$ (\emph{i.e.} $V_m$ and $\mathbbm{1} - V_m$ are hermitian and positive semi-definite), we have
    \begin{equation}
    \label{eqn::Second_moment_summand_1_summand_1}
        \mathbb{E}_{\rho}\left[ \sum_{m = 1}^D H_m^2 \right] V_m^2 = \mathbb{E}_{\rho}\left[ \sum_{m = 1}^D H_m \right] V_m^2 \lesssim \mathbb{E}_{\rho}\left[ \sum_{m = 1}^D H_m \right] V_m.
    \end{equation}
    For the second summand in~\eqref{eqn::Second_moment_summand_1} we now consider the identity~\eqref{eqn::duality_2_desing_MUB}. Let $b_1, b_2 \in \{1,...,d+1\}$ and $k_1, k_2 \in \{1,...,d\}$ be the unique integers such that $m_1 = (d-1)b_1 + k_1$ and $m_2 = (d-1)b_2 + k_2$. We can then rewrite
    \begin{equation}
    \label{eqn::Second_moment_summand_1_summand_2}
        \mathbb{E}_{\rho}\left[ \sum_{m_1 \neq m_2} H_{m_1}H_{m_2} \right] V_{m_1} V_{m_2} = \sum_{b = 1}^{d+1} \sum_{k_1 \neq k_2}^d \mathbb{E}_{\rho}\left[ W_{k_1} W_{k_2} \right] P_{k_1}^{(b)} P_{k_2}^{(b)},
    \end{equation}
    where we used that $b_1 = b_2$ as otherwise one of the terms $W_{k_1} W_{k_2}$ would be zero. Let us now calculate the expectation
    \begin{align*}
        \mathbb{E}_{\rho}\left[ W_{k_1} W_{k_2} \right] =& \sum_{\{\omega | \omega_{k_1} = \omega_{k_2} = 1 \}} \omega_{k_1} \omega_{k_2} \mathbb{P}_{\rho} \left( F(R^{M_{\delta}}) = (\omega, b) \right)
        \\
        =& \frac{1}{d+1} \left( \frac{e^{\frac{\delta}{2}}}{e^{\frac{\delta}{2}}+1} \right)^d \sum_{\{\omega | \omega_{k_1} = \omega_{k_2} = 1 \}} \sum_{k = 1}^d e^{-\frac{\delta}{2} \norm{\omega - e_k}_1} \Tr\left[ \rho P_k^{(b)} \right]
        \\
        =& \frac{1}{d+1} \left( \frac{e^{\frac{\delta}{2}}}{e^{\frac{\delta}{2}}+1} \right)^d \Bigg( \sum_{k \not \in \{k_1, k_2\}}^d \sum_{\{\omega | \omega_{k_1} = \omega_{k_2} = 1 \}}  e^{-\frac{\delta}{2} \norm{\omega - e_k}_1} \Tr\left[ \rho P_k^{(b)} \right]
        \\
        & \qquad \qquad \qquad \qquad + \sum_{\{\omega | \omega_{k_1} = \omega_{k_2} = 1 \}}  e^{-\frac{\delta}{2} \norm{\omega - e_{k_1}}_1} \Tr\left[ \rho P_{k_1}^{(b)} \right]
        \\
        & \qquad \qquad \qquad \qquad + \sum_{\{\omega | \omega_{k_1} = \omega_{k_2} = 1 \}}  e^{-\frac{\delta}{2} \norm{\omega - e_{k_2}}_1} \Tr\left[ \rho P_{k_2}^{(b)} \right] \Bigg)
        \\
        =& \frac{1}{d+1} \left( \frac{e^{\frac{\delta}{2}}}{e^{\frac{\delta}{2}}+1} \right)^d \Bigg( \sum_{k \not \in \{k_1, k_2\}}^d \Tr\left[ \rho P_k^{(b)} \right] e^{-\delta} \left( e^{-\frac{\delta}{2}} + 1 \right)^{d-2}
        \\
        &\qquad \qquad \qquad \qquad + \Tr\left[ \rho P_{k_1}^{(b)} \right] e^{-\frac{\delta}{2}} \left( e^{-\frac{\delta}{2}} + 1 \right)^{d-2}
        \\
        &\qquad \qquad \qquad \qquad + \Tr\left[ \rho P_{k_2}^{(b)} \right] e^{-\frac{\delta}{2}} \left( e^{-\frac{\delta}{2}} + 1 \right)^{d-2} \Bigg)
        \\
        =& \frac{1}{d+1} \left( \frac{1}{e^{\frac{\delta}{2}}+1} \right)^2 \Bigg( 1 + \Tr\left[ \rho \left( P_{k_1}^{(b)} + P_{k_2}^{(b)} \right) \right] \left( e^{\frac{\delta}{2}} - 1 \right) \Bigg)
        \\
        =& \frac{1}{d+1} \beta \Bigg( \beta + \alpha \Tr\left[ \rho \left( P_{k_1}^{(b)} + P_{k_2}^{(b)} \right) \right] \Bigg),
    \end{align*}
    where we used the fact that for a fixed $b$ the operators $P_k^{(b)}$ form a quantum measurement and therefore $\sum_{k \not \in \{k_1, k_2\}}^d P_k^{(b)} = \mathbbm{1} - P_{k_1}^{(b)} - P_{k_2}^{(b)}$. Plugging this result into~\eqref{eqn::Second_moment_summand_1_summand_2} gives
    \begin{align}
    \label{eqn::Second_moment_summand_1_summand_2_version_2}
        \mathbb{E}_{\rho}\left[ \sum_{m_1 \neq m_2} H_{m_1}H_{m_2} \right] V_{m_1} V_{m_2} =& \frac{\beta^2}{d+1} \sum_{b = 1}^{d+1} \sum_{k_1 \neq k_2} P_{k_1}^{(b)} P_{k_2}^{(b)} 
        \\
        &+ \frac{\beta \alpha}{d+1} \sum_{b = 1}^{d+1} \sum_{k_1 \neq k_2} \Tr\left[ \rho \left( P_{k_1}^{(b)} + P_{k_2}^{(b)} \right) \right] P_{k_1}^{(b)} P_{k_2}^{(b)}. \notag
    \end{align}
    We again calculate both summands separately. For the first summand, we have
    \begin{align*}
        \frac{\beta^2}{d+1} \sum_{b = 1}^{d+1} \sum_{k_1 \neq k_2} P_{k_1}^{(b)} P_{k_2}^{(b)} &= \frac{\beta^2}{d+1} \sum_{b = 1}^{d+1} \sum_{k_1}^d P_{k_1}^{(b)} \sum_{\genfrac{}{}{0pt}{}{k_2 = 1}{k_2 \neq k_1}}^d P_{k_2}^{(b)} 
        \\
        &= \frac{\beta^2}{d+1} \sum_{b = 1}^{d+1} \sum_{k_1}^d P_{k_1}^{(b)} \left(\mathbbm{1} - P_{k_1}^{(b)} \right)
        \\
        &\lesssim \frac{\beta^2}{d+1} \sum_{b = 1}^{d+1} \sum_{k_1}^d P_{k_1}^{(b)} = \beta^2 \mathbbm{1},
    \end{align*}
    where we again used $0 \lesssim P_{k_1}^{(b)} \lesssim \mathbbm{1}$ and the fact that the $P_k^{(b)}$ form a quantum measurement for fixed $b$. For the second summand in~\eqref{eqn::Second_moment_summand_1_summand_2_version_2}, we first note that the expression is symmetric in $k_1$ and $k_2$. As such, we have
    \begin{align*}
        \frac{\beta \alpha}{d+1} \sum_{b = 1}^{d+1} \sum_{k_1 \neq k_2} \Tr\left[ \rho \left( P_{k_1}^{(b)} + P_{k_2}^{(b)} \right) \right] P_{k_1}^{(b)} P_{k_2}^{(b)} =& 2 \frac{\beta \alpha}{d+1} \sum_{b = 1}^{d+1} \sum_{k_1 = 1}^d \Tr\left[ \rho P_{k_1}^{(b)} \right] P_{k_1}^{(b)} \sum_{\genfrac{}{}{0pt}{}{k_2 = 1}{k_2 \neq k_1}}^d P_{k_2}^{(b)}
        \\
        =& 2 \frac{\beta \alpha}{d+1} \sum_{b = 1}^{d+1} \sum_{k_1 = 1}^d \Tr\left[ \rho P_{k_1}^{(b)} \right] P_{k_1}^{(b)} \hspace*{-2pt} \left( \hspace*{-2pt} \mathbbm{1} - P_{k_1}^{(b)} \right)
        \\
        \lesssim& 2 \frac{\beta \alpha}{d+1} \sum_{b = 1}^{d+1} \sum_{k_1 = 1}^d \Tr\left[ \rho P_{k_1}^{(b)} \right] P_{k_1}^{(b)}
        \\
        =& 2 \frac{\beta \alpha}{d+1} \sum_{m = 1}^D \Tr\left[ \rho V_m \right] V_m
        \\
        =& 2 \frac{\alpha \beta}{d+1} \left( \frac{\kappa d -1}{d-1} \rho + \frac{d-\kappa}{d-1} \mathbbm{1} \right)
    \end{align*}
    using the linear inversion property~\eqref{eqn::2design_property} of the MUMs. This finally shows 
    \begin{align}
    \label{eqn::Second_moment_summand_1_summand_2_final_version}
        \mathbb{E}_{\rho}\left[ \sum_{m_1 \neq m_2} H_{m_1}H_{m_2} \right] &= \beta^2 \mathbbm{1} + 2 \frac{\alpha \beta}{d+1} \left( \frac{\kappa d -1}{d-1} \rho + \frac{d-\kappa}{d-1} \mathbbm{1} \right)
        \\
        &= \left( \beta^2 + 2 \frac{\alpha \beta}{d+1} + 2 \frac{d-\kappa}{d-1} \right) \mathbbm{1} + 2\frac{\alpha \beta}{d+1} \frac{\kappa d -1}{d-1} \rho. \notag
    \end{align}
    Pluggin equations~\eqref{eqn::Second_moment_summand_1_summand_1} and~\eqref{eqn::Second_moment_summand_1_summand_2_final_version} into~\eqref{eqn::second_moment_calculations_MUM} and using Lemma~\ref{lem::expectation_H_m_v_m_MUM} yields
    \begin{align*}
        \mathbb{E}_{\rho}[\tilde{\rho}^2] \lesssim& \left( \left( \tau_1^2 + 2 \tau_1 \tau_2 \right) \frac{1}{\tau_1} + 2 \tau_1^2 \frac{\alpha}{d+1} \beta \frac{\kappa d -1}{d-1} \right) \rho
        \\
        &+ \left( \left( \tau_1^2 + 2 \tau_1 \tau_2 \right) \left( \frac{\alpha(d - \kappa)}{d^2 -1} + \beta \right) + \tau_2 + \tau_1^2 \left( \beta^2 + 2 \frac{\alpha}{d+1} \beta + 2 \frac{d-\kappa}{d-1} \right)\right) \mathbbm{1}.
    \end{align*}
    We now use the fact that $\tau_1, \tau_2 \leq \frac{d^2}{\alpha \kappa d}$ in order to bound the operator norm of this term by
    \begin{equation*}
        \norm{\mathbb{E}_{\rho}[\tilde{\rho}^2]}_{op} \leq 10 \frac{d^2}{\alpha^2 \kappa^2}.
    \end{equation*}
    Furthermore, we have
    \begin{align*}
        \norm{\tilde{\rho}}_{op} \leq \tau_1 \norm{\sum_{m = 1}^D H_m V_m}_{op} + \tau_2.
    \end{align*}
    Writing again $m = (d-1)b + k$ for unique $b \in \{1,...,d+1\}$ and $k \in \{1,...,d\}$ we first note that the entries of $H$ are empty except for those corresponding to the $b$-th subvector. We therefore have
    \begin{align*}
        \sum_{m = 1}^D H_m V_m = \sum_{k = 1}^d W_k P_{k}^{(b)} \lesssim \sum_{k = 1}^d P_k^{(b)} = \mathbbm{1},
    \end{align*}
    where the maximum is attained if all entries of $W_k$ are $1$. This shows
    \begin{equation*}
        \norm{\tilde{\rho}}_{op} \leq 2 \frac{d}{\alpha \kappa}.
    \end{equation*}
\end{proof}

\newpage 
\section{Technical proofs for the lower bound scheme}
\label{sec::proofs_for_lower_bounds}

\begin{Restatementlemma*}{\scshape~\ref{lem::C1}.\hspace{7pt}}
    If $\mathcal{F} = \{\rho_v\}_{v \in \mathcal{V}}$ induces a $2\eta$-Hamming separation for $\norm{\cdot}^2$ it holds
    \begin{equation*}
        \inf_{(M, \hat{\rho})} \max_{\nu \in \mathcal{V}} \mathbb{E}_{\nu}\left[ \norm{\hat{\rho} - \rho_{\nu}}^2\right]
        \geq \eta \sum_{j = 1}^D \left(1 -  \norm{\mathbb{P}_{+j}^{R^M,n} - \mathbb{P}_{-j}^{R^M,n}}_{TV} \right).
    \end{equation*}
    where the infimum on the left is taken over all measurements $M = M_1 \otimes ... \otimes M_n$ and subsequent estimators $\hat{\rho}$.
\end{Restatementlemma*}
\begin{proof}
    Fix a measurement $M = \bigotimes_{i = 1}^n M_i$ and an arbitrary estimator $\hat{\rho}_n: \Omega_1 \otimes ... \otimes \Omega_n \to \mathcal{S}(\mathbb{C}^d)$. We then have
    \begin{align*}
        \max_{\nu \in \mathcal{V}} \mathbb{E}_{\nu}\left[ \norm{\hat{\rho}_n - \rho_{\nu}}^2\right] &\geq \frac{1}{|\mathcal{V}|} \sum_{\nu \in \mathcal{V}} \mathbb{E}_{\nu}\left[ \norm{\hat{\rho} - \rho_{\nu}}^2\right]
        \\
        &\geq \frac{1}{|\mathcal{V}|} \sum_{\nu \in \mathcal{V}} 2 \eta \sum_{j = 1}^D \mathbb{E}_{\nu} \left[ \mathbbm{1}_{\left\{ \mathbf{v}(\hat{\rho}_n)_j \neq \nu_j \right\}} \right]
        \\
        &\geq 2 \eta \sum_{j = 1}^D  \frac{1}{|\mathcal{V}|} \sum_{\nu \in \mathcal{V}} \mathbb{E}_{\nu} \left[ \mathbbm{1}_{\left\{ \mathbf{v}(\hat{\rho}_n)_j \neq \nu_j \right\}} \right]
    \end{align*}
    using the fact that the maximum is larger that the average and the $2\eta$-Hamming separation assumption. We will now consider the latter sum. Here, we have
    \begin{align*}
        \frac{1}{|\mathcal{V}|} \sum_{\nu \in \mathcal{V}} \mathbb{E}_{\nu} \left[ \mathbbm{1}_{\left\{ \mathbf{v}(\hat{\rho}_n)_j \neq \nu_j \right\}} \right] =& \frac{1}{|\mathcal{V}|} \sum_{\nu \in \mathcal{V}} \mathbb{P}_{\nu}^{R^M,n} \left( \mathbf{v}(\hat{\rho}_n)_j \neq \nu_j \right)
        \\
        =& \frac{1}{|\mathcal{V}|} \sum_{\nu: \nu_j = +1} \mathbb{P}_{\nu}^{R^M,n} \left( \mathbf{v}(\hat{\rho}_n)_j \neq \nu_j \right) 
        \\
        &+ \frac{1}{|\mathcal{V}|} \sum_{\nu: \nu_j = -1} \mathbb{P}_{\nu}^{R^M,n} \left( \mathbf{v}(\hat{\rho}_n)_j \neq \nu_j \right)
        \\
        =& \frac{1}{2} \mathbb{P}_{+j}^{R^M,n}\left( \mathbf{v}(\hat{\rho}_n)_j \neq +1 \right) + \frac{1}{2} \mathbb{P}_{-j}^{R^M,n}\left( \mathbf{v}(\hat{\rho}_n)_j \neq -1 \right)
    \end{align*}
    using equation (\ref{eqn::distribution_equalities}) in the last step. This gives the lower bound
    \begin{align*}
        \max_{\nu \in \mathcal{V}} \mathbb{E}_{\nu}\left[ \norm{\hat{\rho}_n - \rho_{\nu}}^2\right] \geq \eta \sum_{j = 1}^D \left( \mathbb{P}_{+j}^{R^M,n}\left( \mathbf{v}(\hat{\rho}_n)_j \neq +1 \right) + \mathbb{P}_{-j}^{R^M,n}\left( \mathbf{v}(\hat{\rho}_n)_j \neq -1 \right) \right).
    \end{align*}
    Taking the infimum over all pairs of estimators and measurements $(\hat{\rho}, M)$ gives
    \begin{align*}
        \inf_{(M, \hat{\rho}_n)} \max_{\nu \in \mathcal{V}} \mathbb{E}_{\nu}\left[ \norm{\hat{\rho}_n - \rho_{\nu}}^2\right]
        &\geq \inf_{(M, \hat{\rho}_n)} \eta \sum_{j = 1}^D \left( \mathbb{P}_{+j}^{R^M,n}\left( \mathbf{v}(\hat{\rho}_n)_j \neq +1 \right) + \mathbb{P}_{-j}^{R^M,n}\left( \mathbf{v}(\hat{\rho}_n)_j \neq -1 \right) \right)
        \\
        &\geq \eta \sum_{j = 1}^D \inf_{(M, \hat{\psi})} \left( \mathbb{P}_{+j}^{R^M, n} (\hat{\psi} \neq +1) + \mathbb{P}_{-j}^{R^M, n} (\hat{\psi} \neq -1) \right)
        \intertext{where we take the infimum over all test functions $\psi$ in the last step. Since the latter is given by the total-variation distance between the two probability distributions, we obtain the final result}
        &\geq \eta \sum_{j = 1}^D  \left(1 -  \norm{\mathbb{P}_{+j}^{R^M,n} - \mathbb{P}_{-j}^{R^M,n}}_{TV} \right).
    \end{align*}
\end{proof}

\begin{Restatementtheorem*}{\scshape~\ref{thm::C3}.\hspace{7pt}}
Let $\alpha \in [0, \frac{1}{2})$. For the distributions defined in Definition~\ref{defn::mixed_distribution states}, assuming $\rho_{\nu, i} = \rho_{\nu}$ for $i = 1,...,n$, and for any locally-$\alpha$-gentle measurement it holds
\begin{equation*}
    \sum_{j = 1}^D D_{KL}^{sym}\left(\mathbb{P}_{+j}^{R^M,n} \middle| \middle| \mathbb{P}_{-j}^{R^M,n} \right) \leq \frac{64 n \alpha^2}{(1 - 2\alpha)^4} \sup_{\norm{A}_{op} \leq 1}\sum_{j = 1}^D \frac{1}{2^{D-1}} \hspace*{-5pt} \sum_{\nu \in \{-1,1\}^{D-1}} \hspace*{-5pt} \Tr \left[ A (\rho_{\nu_{j+}} - \rho_{\nu_{j-}})\right]^2.
\end{equation*}
Here, for fixed $\nu \in \{-1,1\}^{D-1}$, $\rho_{\nu_{j+}}, \rho_{\nu_{j-}}$ denote the states $\rho_{\nu, i}$ where the parameter vector is padded such that the $j-th$ component is either $+1$ or $-1$. Note that $\rho_{\nu_{j\pm},i} \neq \rho_{\pm j, i}$.
\end{Restatementtheorem*}

\begin{proof}
    We start by using the convexity of the Kullback-Leibler divergence for the distributions
    \begin{equation*}
        \mathbb{P}_{\pm j}^{R^M, n} = \frac{1}{2^{D-1}} \sum_{\nu: \nu_j = \pm1} \mathbb{P}_{\nu}^{R^M, n}
    \end{equation*}
    we get 
    \begin{equation*}
        D_{KL}\left( \mathbb{P}_{+j}^{R^M, n} \middle|\middle| \mathbb{P}_{-j}^{R^M, n} \right) \leq \frac{1}{2^{D-1}} \sum_{\nu \in \{-1,1\}^{D-1}} D_{KL}\left( \mathbb{P}_{\nu_{j+}}^{R^M,n} \middle| \middle| \mathbb{P}_{\nu_{j-}}^{R^M,n} \right)
    \end{equation*}
    where we use the notation $\mathbb{P}_{\nu_{j+}}^{R^M,n}$ and $\mathbb{P}_{\nu_{j-}}^{R^M,n}$ to denote whether the $j$-th component of $\nu$ is positive or negative. We now use the fact that the $\mathbb{P}_{\nu}^{R^M,n}$ are product measures to write
    \begin{equation*}
        \sum_{\nu \in \{-1,1\}^{D-1}} D_{KL}\left( \mathbb{P}_{\nu_+}^{R^M,n} \middle| \middle| \mathbb{P}_{\nu_-}^{R^M,n} \right) = \sum_{\nu \in \{-1,1\}^{D-1}} \sum_{i = 1}^n D_{KL}\left( \mathbb{P}_{\nu_+}^{R^{M_i}} \middle| \middle| \mathbb{P}_{\nu_-}^{R^{M_i}} \right).
    \end{equation*}
    For the symmetrized Kullback-Leibler divergence we now have
    \begin{align*}
        D_{KL}^{sym}\left(\mathbb{P}_{+j}^{R^M,n} \middle| \middle| \mathbb{P}_{-j}^{R^M,n} \right) =& D_{KL}\left(\mathbb{P}_{+j}^{R^M,n} \middle| \middle| \mathbb{P}_{-j}^{R^M,n} \right) + D_{KL}\left(\mathbb{P}_{-j}^{R^M,n} \middle| \middle| \mathbb{P}_{+j}^{R^M,n} \right)
        \\
        \leq& \frac{1}{2^{D-1}} \sum_{\nu \in \{-1,1\}^{D-1}} \sum_{i = 1}^n D_{KL}\left( \mathbb{P}_{\nu_+}^{R^{M_i}} \middle| \middle| \mathbb{P}_{\nu_-}^{R^{M_i}} \right) 
        \\
        &+ \frac{1}{2^{D-1}} \sum_{\nu \in \{-1,1\}^{D-1}} \sum_{i = 1}^n D_{KL}\left( \mathbb{P}_{\nu_-}^{R^{M_i}} \middle| \middle| \mathbb{P}_{\nu_+}^{R^{M_i}} \right)
        \\
        =& \frac{1}{2^{D-1}} \sum_{\nu \in \{-1,1\}^{D-1}} \sum_{i = 1}^n D_{KL}^{sym}\left( \mathbb{P}_{\nu_+}^{R^{M_i}} \middle| \middle| \mathbb{P}_{\nu_-}^{R^{M_i}} \right).
    \end{align*}
    Now, using the fact that locally-$\alpha$ gentle measurements are locally-$\delta$ quantum differentially private with $\delta = 2\log(\frac{1+2\alpha}{1-2\alpha})$, we will bound the last symmterized Kullback-Leibler divergence as
    \begin{align*}
        D_{KL}^{sym}\left( \mathbb{P}_{\nu_+}^{R^{M_i}} \middle| \middle| \mathbb{P}_{\nu_-}^{R^{M_i}} \right) &= \sum_{\omega \in \Omega} \left( p_{\nu_{j+}}^{R^{M_i}}(\omega) - p_{\nu_{j-}}^{R^{M_i}}(\omega) \right) \log\left( \frac{p_{\nu_{j-}}^{R^{M_i}}(\omega)}{p_{\nu_{j+}}^{R^{M_i}}(\omega)} \right)
        \\
        &\leq \sum_{\omega \in \Omega} \left( p_{\nu_{j+}}^{R^{M_i}}(\omega) - p_{\nu_{j-}}^{R^{M_i}}(\omega) \right)^2 \frac{1}{\min\left\{ p_{\nu_{j+}}^{R^{M_i}}(\omega), p_{\nu_{j-}}^{R^{M_i}}(\omega) \right\}}
        \\
        &\leq \sum_{\omega \in \Omega} \Tr \left[ E_{\omega, i} (\rho_{\nu_{j+},i} - \rho_{\nu_{j-}, i}) \right]^2 \frac{1}{\lambda_{min}(E_{\omega})}
        \\
        &= \sum_{\omega \in \Omega} \Tr \left[ \left( \frac{E_{\omega, i}}{\lambda_{min}(E_{\omega})} - \mathbbm{1} \right) (\rho_{\nu_{j+},i} - \rho_{\nu_{j-}, i}) \right]^2 \lambda_{min}(E_{\omega})
        \\
        &= \sum_{\omega \in \Omega} (e^{\delta} -1)^2 \Tr \left[ A_{\omega, i} (\rho_{\nu_{j+},i} - \rho_{\nu_{j-}, i}) \right]^2 \lambda_{min}(E_{\omega})
    \end{align*}
    where
    \begin{equation*}
        A_{\omega, i} := \left( \frac{E_{\omega, i}}{\lambda_{min}(E_{\omega})} - \mathbbm{1} \right)\frac{1}{(e^{\delta} - 1)}
    \end{equation*}
    is positive semi-definite with largest eigenvalue being $1$ due to the properties of gentleness. As such, since $\rho_{\nu} = \rho_{\nu, i}$  for all $i = 1,...,n$, we may bound the above term in the following way by taking the supremum over positive semi-definite operators with operator norm bounded by $1$ and the fact that $\sum_{\omega \in \Omega} \lambda_{min}(E_{\omega}) \leq 1$
    \begin{align*}
        \sum_{j = 1}^D D_{KL}^{sym}\left(\mathbb{P}_{+j}^{R^M,n} \middle| \middle| \mathbb{P}_{-j}^{R^M,n} \right) &\leq \sum_{j = 1}^D \frac{1}{2^{D-1}} \sum_{\nu \in \{-1,1\}^{D-1}} \sum_{i = 1}^n D_{KL}\left( \mathbb{P}_{\nu_+}^{R^{M_i}} \middle| \middle| \mathbb{P}_{\nu_-}^{R^{M_i}} \right).
        \\
        &\leq n (e^{\delta} -1)^2 \sup_{\norm{A}_{op} \leq 1} \sum_{j = 1}^D \frac{1}{2^{D-1}} \hspace*{-6pt} \sum_{\nu \in \{-1,1\}^{D-1}} \hspace*{-6pt} \Tr \left[ A (\rho_{\nu_{j+}} - \rho_{\nu_{j-}}) \right]^2.
    \end{align*}
    Plugging in 
    \begin{equation*}
        \left(e^{\delta} -1\right)^2 = \frac{64 \alpha^2}{(1 - 2\alpha)^4}
    \end{equation*} 
    gives the final result
    \begin{equation*}
        \sum_{j = 1}^D D_{KL}^{sym}\left(\mathbb{P}_{+j}^{R^M,n} \middle| \middle| \mathbb{P}_{-j}^{R^M,n} \right) \leq \frac{64 n \alpha^2}{(1 - 2\alpha)^4} \sup_{\norm{A}_{op} \leq 1}\sum_{j = 1}^D \frac{1}{2^{D-1}} \hspace*{-5pt} \sum_{\nu \in \{-1,1\}^{D-1}} \hspace*{-5pt} \Tr \left[ A (\rho_{\nu_{j+}} - \rho_{\nu_{j-}})\right]^2.
    \end{equation*}
\end{proof}

\newpage
\section{Technical Lemmas for the construction of local alternatives}

\subsection{Low-rank states}
In this section we prove the necessary results in order to apply our lower bound scheme to low-rank states. We start with a technical result on the representation of the local alternative states $\rho_{\nu}$. More precisely, for the unitary matrices defined in~\eqref{eqn::S_nu_definition} by
\begin{equation*}
    S_{\nu} = \sum_{l = 1}^r \sum_{k = r + 1}^d \nu_{k, l} T_{k,l} = \begin{bmatrix}
        0 & V_{\nu}
        \\
        -V_{\nu}^t & 0
    \end{bmatrix}\hspace{20pt} \text{and} \hspace{20pt} U_{\nu} = \exp(\epsilon S_{\nu})
\end{equation*}
we have the representation 
\begin{equation*}
    \rho_{\nu} = U_{\nu} \rho_{mm}^{(r)} U_{\nu}^* = \rho_{mm}^{(r)} +  \frac{1}{r}\sum_{m = 1}^{\infty} \frac{\epsilon^m}{m!} ad_{S_{\nu}}^m(\mathbbm{1}_r) = \rho_{mm}^{(r)} + \Delta_{\nu}(\epsilon).
\end{equation*}
The following results then shows that $\Delta_{\nu}(\epsilon)$ is almost linear for small perturbations $\epsilon$. We do this by deriving the non-commutative Taylor expansion of $\exp(\epsilon S_{\nu})$ with respect to $\epsilon$. Intuitively, as $\sin^2(\epsilon) = \epsilon^2 + O(\epsilon^3)$ and $\sinc(\epsilon) = 1 + O(\epsilon)$ we can see that the perturbation in the first-order approximation has only non-diagonal entries. Formally, we have the following result.

\label{sec::Appendix_low_rank_states}
\begin{lemma}
\label{lem::calculation_Delta_nu_epsilon}
    Let $S_{\nu}$ be as in~\eqref{eqn::S_nu_definition} and $\Delta_{\nu}(\epsilon)$ as in~\eqref{eqn::rho_nu_low_rank}. Then  it holds
    \begin{equation*}
        \Delta_{\nu}(\epsilon) = -\frac{1}{r} \begin{bmatrix} \sin^2\left( \epsilon \sqrt{G_{\nu}} \right) &
            2 \epsilon \sinc(2\epsilon \sqrt{G_{\nu}}) V_{\nu}
            \\
            2 \epsilon V_{\nu}^t \sinc(2\epsilon \sqrt{G_{\nu}}) & -\sin^2\left( \epsilon \sqrt{H_{\nu}} \right)
        \end{bmatrix},
    \end{equation*}
    where $G_{\nu} = V_{\nu}V_{\nu}^t$ and $H_{\nu} = V_{\nu}^t V_{\nu}$ are Gram-matrices and $\sinc(x) = \frac{\sin(x)}{x}$ is the sinc-function.
\end{lemma}
\begin{proof}
    Note that $G_{\nu} = V_{\nu}V_{\nu}^t$ is positive semi-definite, and therefore, there exist $U_{G_{\nu}} \in U(d)$ and a diagonal matrix $D_{G_{\nu}} = \diag(\lambda_1(G_{\nu}),...,\lambda_d(G_{\nu}))$ such that $G_{\nu} = U_{G_{\nu}} D_{G_{\nu}} U_{G_{\nu}}^*$. As such we can define
    \begin{equation}
    \label{eqn::sinc_matrix_expansion}
        \sinc(2 \epsilon \sqrt{G_{\nu}}) = U_{G_{\nu}} \sinc(2 \epsilon D_{G_{\nu}}^{1/2}) U_{G_{\nu}}^* = U_{G_{\nu}} \begin{bmatrix}
            \sinc(2 \epsilon  \lambda_1(G_{\nu}))&  & \multicolumn{2}{c}{\text{\kern0.5em\smash{\raisebox{-1ex}{\Large 0}}}}
            \\
            & \ddots & 
            \\
            \multicolumn{2}{c}{\text{\kern-0.5em\smash{\raisebox{0.75ex}{\Large 0}}}} & & \sinc(2 \epsilon  \lambda_d(G_{\nu}))
        \end{bmatrix} U_{G_{\nu}}^*.
    \end{equation}
    The same applies to $H_{\nu} = V_{\nu}^tV_{\nu}$. Let us now prove the statement by induction. We claim that $ad_{S_{\nu}}^m(\mathbbm{1}_r)$ is always of the following block form:
    \begin{equation}
    \label{eqn::ad_block_matrix_form}
        ad_{S_{\nu}}^m(\mathbbm{1}_r) = \begin{bmatrix}
            A_m & B_m
            \\
            B_m^t & C_m
        \end{bmatrix}, \qquad \text{for } A_m \in \mathbb{Z}^{r \times r}, B_m \in \mathbb{Z}^{r \times (d-r)}, C_m \in \mathbb{Z}^{(d-r) \times (d-r)}
    \end{equation}
    such that $A_m, C_m = 0$ for odd $m$ and $B_m = 0$ for even $m$. Furthermore, for even $2m$ it holds
    \begin{equation*}
        ad_{S_{\nu}}^{2m}(\mathbbm{1}_r) =  (-1)^m 2^{2m-1} 
        \begin{bmatrix}
            G_{\nu}^m & 0
            \\
            0 & -H_{\nu}^m
        \end{bmatrix}
    \end{equation*}and for odd $2m+1$ we have
    \begin{equation*}
        ad_{S_{\nu}}^{2m+1}(\mathbbm{1}_r) = (-1)^{m+1} 2^{2m} \begin{bmatrix}
            0 & G_{\nu}^m V_{\nu}
            \\
            V_{\nu}^t G_{\nu}^m  & 0
        \end{bmatrix}.
    \end{equation*}
    Let us prove of the claim by starting with $ad_{S_{\nu}}(\mathbbm{1}_r)$ and $ad_{S_{\nu}}^{2}(\mathbbm{1}_r)$ which lets us start the proof by induction. We see that
    \begin{align*}
        ad_{S_{\nu}}(\mathbbm{1}_r) = \left[  S_{\nu}, \mathbbm{1}_r \right] &= - \begin{bmatrix}
        0 & V_{\nu}
        \\
        V_{\nu}^t & 0
        \end{bmatrix} = (-1)^{0+1} 2^0 \begin{bmatrix}
            0 & G_{\nu}^0 V_{\nu}
            \\
            V_{\nu}^t G_{\nu}^0  & 0
        \end{bmatrix}
    \end{align*}
    and
    \begin{align*}
        ad_{S_{\nu}}^2(\mathbbm{1}_r) &= 2 \begin{bmatrix}
        - V_{\nu}V_{\nu}^t & 0
        \\
        0 & V_{\nu}^t V_{\nu}
        \end{bmatrix} = (-1)^1 2^{2\cdot 1-1} 
            \begin{bmatrix}
                G_{\nu}^1 & 0
                \\
                0 & -V_{\nu}^t G^{0}V_{\nu}
            \end{bmatrix}.
    \end{align*}
    Now, if we assume $ad_{S_{\nu}}^{2m+1}(\mathbbm{1}_r)$ and $ad_{S_{\nu}}^{2m}(\mathbbm{1}_r)$ of the above form, we can directly calculate
    \begin{equation*}
        ad_{S_{\nu}}^{2(m+1)+1}(\mathbbm{1}_r) = (-1)^{(m+1)+1} 2^{2(m+1)} \begin{bmatrix}
            0 & G_{\nu}^{m+1} V_{\nu}
            \\
            V_{\nu}^t G_{\nu}^{m+1}  & 0
        \end{bmatrix}
    \end{equation*}
    and
    \begin{equation*}
        ad_{S_{\nu}}^{2(m+1)}(\mathbbm{1}_r) =  (-1)^{(m+1)} 2^{2(m+1)-1} 
        \begin{bmatrix}
            G_{\nu}^{(m+1)} & 0
            \\
            0 & -H_{\nu}^{(m+1)}
        \end{bmatrix}
    \end{equation*}
    which shows that $ad_{S_{\nu}}^{m}(\mathbbm{1}_r)$ is always of the claimed form~\eqref{eqn::ad_block_matrix_form}. Now, as the Taylor-expansion of $\sinc(x)$ and $\sin^2(x)$ are given by 
    \begin{equation*}
        \sinc(x) = \sum_{m = 0}^{\infty} \frac{(-1)^n x^{2n}}{(2n+1)!} \hspace{20pt} \text{and} \hspace{20pt} \sin^2(x) = \sum_{m = 1}^{\infty} (-1)^m \frac{2^{(2m -1)} x^{2m}}{(2m)!},
    \end{equation*}
    for $\Delta_{\nu}(\epsilon)$ we have
    \begin{align*}
        \Delta_{\nu}(\epsilon) &= \frac{1}{r}\sum_{m = 1}^{\infty} \frac{\epsilon^m}{m!} ad_{S_{\nu}}^m(\mathbbm{1}_r) 
        \\
        &= \frac{1}{r} \begin{bmatrix}
            \sum_{m = 1}^{\infty} (-1)^m \frac{2^{(2m-1)} \epsilon^{2m}}{(2m)!} G_{\nu}^m & \sum_{m = 1}^{\infty} \frac{(-1)^{m+1} (2\epsilon)^{2m+1}}{(2m+1)!} G_{\nu}^m V_{\nu}
            \\
            \\
            V_{\nu}^t \sum_{m = 1}^{\infty} \frac{(-1)^{m+1} (2\epsilon)^{2m+1}}{(2m+1)!} G_{\nu}^m & -\sum_{m = 1}^{\infty} (-1)^m \frac{2^{(2m-1)} \epsilon^{2m}}{(2m)!} H_{\nu}^m
        \end{bmatrix}
        \\
        \\
        &= \frac{1}{r} \begin{bmatrix}
            \sum_{m = 1}^{\infty} (-1)^m \frac{2^{(2m-1)} \epsilon^{2m}}{(2m)!} \sqrt{G_{\nu}}^{2m} & \sum_{m = 1}^{\infty} \frac{(-1)^{m+1} (2\epsilon)^{2m+1}}{(2m+1)!} \sqrt{G_{\nu}}^{2m} V_{\nu}
            \\
            \\
            V_{\nu}^t \sum_{m = 1}^{\infty} \frac{(-1)^{m+1} (2\epsilon)^{2m+1}}{(2m+1)!} \sqrt{G_{\nu}}^{2m} & -\sum_{m = 1}^{\infty} (-1)^m \frac{2^{(2m-1)} \epsilon^{2m}}{(2m)!} \sqrt{H_{\nu}}^{2m}
        \end{bmatrix}
        \\
        \\
        &= -\frac{1}{r} \begin{bmatrix} \sin^2\left( \epsilon \sqrt{G_{\nu}} \right) &
            2 \epsilon \sinc(2\epsilon \sqrt{G_{\nu}}) V_{\nu}
            \\
            2 \epsilon V_{\nu}^t \sinc(2\epsilon \sqrt{G_{\nu}}) & -\sin^2\left( \epsilon \sqrt{H_{\nu}} \right)
        \end{bmatrix}.
    \end{align*}
    This completes the proof of the Lemma.
\end{proof}

With the results of Lemma~\ref{lem::calculation_Delta_nu_epsilon} we are now able to proof the separation property of the state $\rho_{\nu}$ defined in~\eqref{eqn::rho_nu_low_rank}. 

\begin{Restatementlemma*}{\scshape~\ref{lem::Hamming_distance_low_rank}.\hspace{7pt}}
    Let $D \in \{1,...,d^2 -1 \}$ and $\mathcal{V} = \{-1,1\}^D$ and assume that $\epsilon \leq \sqrt{\frac{3}{8}} \frac{1}{r \sqrt{d-r}}$. The set of local binary hypotheses $\mathcal{F} = \{\rho_{\nu}\}_{\nu \in \mathcal{V}}$ defined by~\eqref{eqn::rho_nu_low_rank} induces a $2\eta$-Hamming separation for $\norm{\cdot}_F^2$ on $\mathcal{S} = \mathcal{S}_r(\mathbb{C}^d)$ where $\eta = \frac{\epsilon^2}{r^2}$
\end{Restatementlemma*}
\begin{proof}
    Let $\rho_{\nu} = \rho_{mm}^{(r)} + \Delta_{\nu}(\epsilon)$ for 
    \begin{equation*}
        \Delta_{\nu}(\epsilon) = -\frac{1}{r} \begin{bmatrix} \sin^2\left( \epsilon \sqrt{G_{\nu}} \right) &
            2 \epsilon \sinc(2\epsilon \sqrt{G_{\nu}}) V_{\nu}
            \\
            2 \epsilon V_{\nu}^t \sinc(2\epsilon \sqrt{G_{\nu}}) & -\sin^2\left( \epsilon \sqrt{H_{\nu}} \right)
        \end{bmatrix},
    \end{equation*}
    as shown in Lemma~\ref{lem::calculation_Delta_nu_epsilon} and assume that $\epsilon \leq \sqrt{\frac{3}{8}} \frac{1}{r \sqrt{d-r}}$. Then, noting that the maximal possible eigenvalue of $G_{\nu}$ is given by $\lambda_{max}(D_{G_{\nu}}) \leq r(d-r)$ and $\sinc(x)$ being monotone and decreasing on $(0, \pi)$, we have
    \begin{align}
        \label{eqn::bound_on_sinc_operator_norm}
        \norm{\sinc(2\epsilon \sqrt{G_{\nu}}) - \mathbbm{1}_r}_{op} &= \norm{\sinc(2\epsilon \sqrt{D_{G_{\nu}}}) - \mathbbm{1}_r}_{op}
        \\
        &\leq \left| \sinc\left(2\epsilon \lambda_{max}({D_{G_{\nu}}^{1/2}})\right)  - 1\right| \notag
        \\
        &\leq 1 - \sinc\left(2\epsilon \sqrt{r(d-r)} \right) \notag
        \\
        &\leq \frac{\left(2\epsilon \sqrt{r(d-r)}\right)^2}{3} \notag
        \\
        &\leq \frac{1}{2\sqrt{r}}. \notag
    \end{align}
    For any state $\rho \in \mathcal{S}_r(\mathbb{C}^d)$ we now have
    \begin{align*}
        \norm{\rho - \rho_{\nu}}_F^2 &\geq 2 \norm{\rho_{1,2} + \frac{2\epsilon}{r} \sinc(2\epsilon \sqrt{G_{\nu}}) V_{\nu} }_F^2
        \intertext{where $\rho_{1,2}$ denotes the $r \times (d-r)$ sub-matrix of $\rho$ in the upper-right. We can then write this as}
        &= 2 \norm{\rho_{1,2} + \frac{2\epsilon}{r}\left(\sinc(2\epsilon \sqrt{G_{\nu}}) - \mathbbm{1}_r + \mathbbm{1}_r \right) V_{\nu} }_F^2
        \\
        &= 2 \norm{\rho_{1,2} + \frac{2\epsilon}{r} V_{\nu} +  \frac{2\epsilon}{r}A_{\nu} V_{\nu} }_F^2, 
        \intertext{where $A_{\nu} = \left(\sinc(2\epsilon \sqrt{G_{\nu}}) - \mathbbm{1}_r\right)$ has operator-norm bounded by $\frac{1}{2\sqrt{r}}$ as we have seen in~\eqref{eqn::bound_on_sinc_operator_norm}. As such, for the entries of $A_{\nu}V_{\nu}$ it holds $\left|(A_{\nu}V_{\nu})\right|_{k,l} \leq \frac{1}{2}$ as
        \begin{equation}
            \label{eqn::bound_on_entries_of_AV}
            \left|(A_{\nu}V_{\nu})\right|_{k,l} = \left| \sum_{i = 1}^r (A_{\nu})_{k,i} (V_{\nu})_{i,l} \right| \leq \sqrt{r}  \left| \sum_{i = 1}^r (A_{\nu})_{k,i} \frac{(V_{\nu})_{i,l}}{\sqrt{r}} \right| \leq \sqrt{r} \norm{A_{\nu}}_{op}
        \end{equation}
        using the fact that the entries of $V_{\nu}$ are in $\{-1,1\}$ which is a $\sqrt{r}$-multiple of a unit vector. This lets us bound the above term from below by}
        &= 2 \sum_{l = 1}^r \sum_{k = r+1}^{d} \left| \hat{\rho}_{k,l} + \frac{2 \epsilon}{r} \nu_{k,l} + \frac{2\epsilon}{r} (A_{\nu}V_{\nu})_{k,l} \right|^2.
        \intertext{We now define the function $\mathbf{v}:\mathcal{S}_r(\mathbb{C}^d) \to \{-1, 1\}^D$ via $\mathbf{v}(\hat{\rho})_{k,l} = -\sgn(\Re(\hat{\rho}_{k,l}))$ which allows us to further bound the term from below by}
        &\geq 2 \frac{\epsilon^2}{r^2} \sum_{l = 1}^r \sum_{k = r + 1}^d \mathbbm{1}_{\{ \mathbf{v}(\hat{\rho})_{k,l} \neq \nu_{k,l} \}},
    \end{align*}
    proving the Lemma.
\end{proof}

\begin{Restatementlemma*}{\scshape~\ref{lem::upper_bound_distances_rank_r_states}.\hspace{7pt}}
    Let $\epsilon \leq \frac{1}{d^2r^2}$, $D = r(d-r)$ and $(\rho_{\nu})_{\nu \in \mathcal{V}}$ as in~\eqref{eqn::rho_nu_low_rank}. Then, for the states defined in Definition~\ref{defn::mixed_distribution states} it holds
    \begin{equation*}
        \sup_{\norm{A}_{op} \leq 1} \sum_{j = 1}^D \frac{1}{2^{D-1}} \sum_{\nu \in \{-1,1\}^{D-1}} \Tr \left[ A (\rho_{\nu_{j+}} - \rho_{\nu_{j-}})\right]^2 \leq 72 \frac{\epsilon^2}{r}.
    \end{equation*}
\end{Restatementlemma*}

\begin{proof}
    Note that, as we identify an index $j \in \{1,...,D\}$ with a pair $(k,l) \in \{r+1,...,d\} \times \{1,...,r\}$, we write the matrix $\Delta_{\nu_j} = (\rho_{\nu_{j+}} - \rho_{\nu_{j-}})$ in the following form, using the form~\eqref{eqn::rho_nu_low_rank_expansion} and Lemma~\ref{lem::calculation_Delta_nu_epsilon}, as 
    \begin{align}
        \Delta_{\nu_j} =& \frac{1}{r} \begin{bmatrix}
        \sin^2\left( \epsilon \sqrt{G_{\nu_{j-}}} \right) - \sin^2\left( \epsilon \sqrt{G_{\nu_{j+}}} \right) & \hspace*{-18pt} 2 \epsilon \left( \sinc\left( 2 \epsilon \sqrt{G_{\nu_{j-}}} \right) V_{\nu_{j-}} - \sinc\left( 2 \epsilon \sqrt{G_{\nu_{j+}}} \right) V_{\nu_{j+}} \right)
        \\ \notag
        \\
        2 \epsilon \left( \sinc\left( 2 \epsilon \sqrt{G_{\nu_{j-}}} \right) V_{\nu_{j-}} - \sinc\left( 2 \epsilon \sqrt{G_{\nu_{j+}}} \right) V_{\nu_{j+}} \right) & \hspace*{-18pt} \sin^2\left( \epsilon \sqrt{H_{\nu_{j+}}} \right) - \sin^2\left( \epsilon \sqrt{H_{\nu_{j-}}} \right)
        \end{bmatrix}
        \\ \notag
        \\
        =& \frac{1}{r} \begin{bmatrix}
            \sin^2\left( \epsilon \sqrt{G_{\nu_{j-}}} \right) - \sin^2\left( \epsilon \sqrt{G_{\nu_{j+}}} \right) & 0
            \\
            0 & \sin^2\left( \epsilon \sqrt{H_{\nu_{j+}}} \right) - \sin^2\left( \epsilon \sqrt{H_{\nu_{j-}}} \right)
        \end{bmatrix} \label{eqn::matrix_sum_1}
        \\ \notag
        \\ 
        &+ \frac{1}{r} \begin{bmatrix}
            0 & 2 \epsilon \left( \sinc\left( 2 \epsilon \sqrt{G_{\nu_{j-}}} \right) V_{\nu_{j-}} - V_{\nu_{j-}}\right)
            \\
            2 \epsilon \left( V_{\nu_{j-}}^t \sinc\left( 2 \epsilon \sqrt{G_{\nu_{j-}}} \right)  - V_{\nu_{j-}}^t \right) & 0
        \end{bmatrix} \label{eqn::matrix_sum_2}
        \\ \notag
        \\ 
        &+ \frac{1}{r} \begin{bmatrix}
            0 & 2 \epsilon \left( \sinc\left( 2 \epsilon \sqrt{G_{\nu_{j+}}} \right) V_{\nu_{j+}} - V_{\nu_{j+}}\right)
            \\
            2 \epsilon \left( V_{\nu_{j+}}^t \sinc\left( 2 \epsilon \sqrt{G_{\nu_{j+}}} \right)  - V_{\nu_{j+}}^t \right) & 0
        \end{bmatrix} \label{eqn::matrix_sum_3}
        \\ \notag
        \\
        &+ \frac{1}{r} \begin{bmatrix}
            0 & 2\epsilon \left(V_{\nu_{j-}} - V_{\nu_{j+}} \right)
            \\
            2\epsilon \left(V_{\nu_{j-}}^t - V_{\nu_{j+}}^t \right) & 0
        \end{bmatrix} \label{eqn::matrix_sum_4}
    \end{align}
    Let us now first consider the matrix in~\eqref{eqn::matrix_sum_1}. Here, we note again that the maximal eigenvalue of $G_{\nu}$ is bounded from above by $\lambda_{max}(G_{\nu}) \leq r(d-r)$, which allows us to bound
    \begin{align*}
        \norm{\sin^2\left( \epsilon \sqrt{G_{\nu_{j-}}} \right) - \sin^2\left( \epsilon \sqrt{G_{\nu_{j+}}} \right)}_{op} &\leq \sin^2\left( \epsilon \sqrt{\lambda_{max}(G_{\nu_{j+}})}\right) + \sin^2\left( \epsilon \sqrt{\lambda_{max}(G_{\nu_{j-}})}\right) 
        \\
        &\leq 2\sin^2 \left( \epsilon \sqrt{r(d-r)} \right) \\
        &\leq 2 \epsilon^2 r (d-r).
    \end{align*}
    The same argument shows that
    \begin{equation*}
        \norm{\sin^2\left( \epsilon \sqrt{H_{\nu_{j+}}} \right) - \sin^2\left( \epsilon \sqrt{H_{\nu_{j-}}} \right)}_{op} \leq 2 \epsilon^2 r (d-r)
    \end{equation*}
    which shows that every entry in~\eqref{eqn::matrix_sum_1} is bounded by $2\epsilon^2 (d-r)$. We can therefore write the matrix in~\eqref{eqn::matrix_sum_1} in the following form
    \begin{align}
        M^{(1)} = 2\epsilon^2 (d-r) \begin{bmatrix}
            M_{11} & 0
            \\
            0 & M_{22}
        \end{bmatrix},
    \end{align}
    where $M_{11}$ and $M_{22}$ are self-adjoint matrices with every entry bounded by $1$. Let us now turn to~\eqref{eqn::matrix_sum_2}. Here, we note that we can write 
    \begin{equation*}
        2 \epsilon \left( \sinc\left( 2 \epsilon \sqrt{G_{\nu_{j-}}} \right) V_{\nu_{j-}} - V_{\nu_{j-}}\right) = 2 \epsilon \left( \sinc\left( 2 \epsilon \sqrt{G_{\nu_{j-}}} \right) - \mathbbm{1}_r\right) V_{\nu_{j-}}
    \end{equation*}
    and with the same argument as in~\eqref{eqn::bound_on_entries_of_AV} we can show that the each entry of the matrix in~\eqref{eqn::matrix_sum_2} is bounded in absolute value from above by $2 \epsilon^2 (d-r) \sqrt{r}$. The same argument applies to the entries of~\eqref{eqn::matrix_sum_3} which shows that we can write the sum of~\eqref{eqn::matrix_sum_2} and~\eqref{eqn::matrix_sum_3} as
    \begin{equation}
        M^{(2)} = 4\epsilon^2 \sqrt{r} (d-r) \begin{bmatrix}
            0 & M_{12}
            \\
            M_{12}^t & 0
        \end{bmatrix},
    \end{equation}
    where $M_{12} \in \mathbb{R}^{r \times (d-r)}$ has entries bounded by 1 in absolute value. Let us now consider the term
    \begin{align*}
        \Tr \left[ A \Delta_{\nu_j} \right]^2
    \end{align*}
    for a self-adjoint positive semi-definite $A$ with $\norm{A}_{op} \leq 1$, which we may bound using $(a + b + c)^2 \leq 3 (a^2 + b^2 + c^2)$ by
    \begin{align}
    \label{eqn::matrix_sum_split}
        \Tr \left[ A \Delta_{\nu_j} \right]^2 \leq 3 \Tr\left[ A M^{(1)} \right]^2 + 3 \Tr\left[ A M^{(2)} \right]^2 + 3 \Tr\left[ A M^{(3)} \right]^2,
    \end{align}
    where $M^{(3)}$ is the matrix in~\eqref{eqn::matrix_sum_4}. For the first summand we use the Hölder-inequality for the Schatten-$p$-norms to get
    \begin{align}
        \Tr\left[ A M^{(1)} \right]^2 &\leq 4 \epsilon^4 (d-r)^2 \norm{A}_{op} \norm{ \begin{bmatrix}
            M_{11} & 0
            \\
            0 & M_{22}
        \end{bmatrix}}_1 \notag
        \\
        &\leq 4 \epsilon^4 (d-r)^2 \norm{A}_{op} d \notag
        \\
        \label{eqn::bound_matrix_sum_1}
        &\leq 4 \epsilon^4 d^3 \norm{A}_{op}.
    \end{align}
    For the second summand, we again use Hölder's inequaliy to bound
    \begin{align}
        \Tr\left[ A M^{(2)} \right]^2 &\leq 16 \epsilon^4 r (d-r)^2 \norm{A}_{op} \norm{ \begin{bmatrix}
            0 & M_{12}
            \\
            M_{12}^t & 0
        \end{bmatrix}}_1 \notag
        \\
        &\leq 16 \epsilon^4 r (d-r)^2 \norm{A}_{op} r \sqrt{d-r} \notag
        \\
        \label{eqn::bound_matrix_sum_2}
        &\leq 16 \epsilon^4 r^2 d^{5/2} \norm{A}_{op}. 
    \end{align}
    Finally, let us consider the last summand. By definition of $V_{\nu_{j}}$, we have that $M^{(3)} = \frac{2\epsilon}{r} T_{k,l}$ for some $(k,l) \in \{r+1,...,d\} \times \{1,...,r\}$ corresponding to the index $j \in \{1,...,D\}$. As such, we have
    \begin{align*}
        &\sum_{j = 1}^D \frac{1}{2^{D-1}} \sum_{\nu \in \{-1,1\}^{D-1}} \Tr \left[ A (\rho_{\nu_{j+}} - \rho_{\nu_{j-}})\right]^2 
        \\
        =& \sum_{l = 1}^r \sum_{k = r+1}^d \frac{1}{2^{D-1}} \sum_{\nu \in \{-1,1\}^{D-1}} \Tr \left[ A (\rho_{\nu_{(k,l)+}} - \rho_{\nu_{(k,l)-}})\right]^2
        \\
        =& \sum_{l = 1}^r \sum_{k = r+1}^d \frac{1}{2^{D-1}} \sum_{\nu \in \{-1,1\}^{D-1}} \frac{4\epsilon^2}{r^2} 4 \Re\left( \bra{k}A\ket{l} \right)^2.
        \intertext{We see that this term is independent of the exact $\nu$, such that we can write this as}
        =& \sum_{l = 1}^r \sum_{k = r+1}^d  \frac{4\epsilon^2}{r^2} 4 \Re\left( \bra{k}A\ket{l} \right)^2
        \\
        \leq& \frac{16\epsilon^2}{r^2} \sum_{l = 1}^r \sum_{k = r+1}^d  \left| \bra{k}A\ket{l} \right|^2.
    \end{align*}
    Now we note that for a self-adjoint positive matrix
    \begin{equation*}
        A = \begin{bmatrix}
            A_{11} & A_{12}
            \\
            A_{12}^* & A_{22}
        \end{bmatrix}, \hspace{20pt} \text{with } \norm{A}_{op} \leq 1 
    \end{equation*}
    we have $\sum_{l = 1}^r \sum_{k = r+1}^d  \left| \bra{k}A\ket{l} \right|^2 = \norm{A_{12}}_F^2$. Let us now bound the Frobenius norm of $A_{12}$. To do this, note that, since $A$ is contractive, $A - A^2$ is also a positive matrix, where 
    \begin{equation*}
        A^2 = \begin{bmatrix}
            A_{11}^2 + A_{12}A_{12}^* & *
            \\
            * & *
        \end{bmatrix}.
    \end{equation*}
    Here we denote by $*$ blocks in the matrix $A^2$ whose exact value is not of interest to us for the calculation. The top left block of $A - A^2$ is therefore given by $A_{11} - A_{11}^2 - A_{12}A_{12}^*$. Since this sub-matrix is also positive, for its trace we have
    \begin{equation*}
        \Tr[A_{11}] - \Tr[A_{11}^2] - \Tr\left[ A_{12}A_{12}^* \right] \geq 0.
    \end{equation*}
    Rearranging this formula gives
    \begin{align*}
        \norm{A_{12}}_F^2 = \Tr\left[ A_{12}A_{12}^* \right] &\leq \Tr[A_{11}] - \Tr[A_{11}^2] = \sum_{m = 1}^r \lambda_m - \lambda_m^2 \leq \frac{r}{4}
    \end{align*}
    where $(\lambda_m)_{m = 1}^r \subseteq [0,1]$ are the eigenvalues of $A_{11}$. This gives the bound
    \begin{equation}
    \label{eqn::bound_matrix_sum_3}
        \sum_{j = 1}^D \frac{1}{2^{D-1}} \sum_{\nu \in \{-1,1\}^{D-1}} \Tr \left[ A (\rho_{\nu_{j+}} - \rho_{\nu_{j-}})\right]^2 \leq \frac{4\epsilon^2}{r}
    \end{equation}
    for any self-adjoint positive $A$ such that $\norm{A}_{op} \leq 1$. In order to proof the final bound we combine the results of~\eqref{eqn::matrix_sum_split} together with~\eqref{eqn::bound_matrix_sum_1},~\eqref{eqn::bound_matrix_sum_2} and~\eqref{eqn::bound_matrix_sum_3} to obtain
    \begin{align*}
        &\sup_{\norm{A}_{op} \leq 1} \sum_{j = 1}^D \frac{1}{2^{D-1}} \sum_{\nu \in \{-1,1\}^{D-1}} \Tr \left[ A (\rho_{\nu_{j+}} - \rho_{\nu_{j-}})\right]^2 
        \\
        \leq& \sup_{\norm{A}_{op} \leq 1} 3  \sum_{j = 1}^D \left( 4 \epsilon^4 d^3 \norm{A}_{op} + 16 \epsilon^4 r^2 d^{5/2} \norm{A}_{op}\right) + 12 \frac{\epsilon^2}{r}
        \\
        \leq& 12 \epsilon^4 d^4 r + 48 \epsilon^4 r^3 d^{7/2} + 12 \frac{\epsilon^2}{r}
        \\
        \leq& \frac{72 \epsilon^2}{r}
    \end{align*}
    for $\epsilon \leq \frac{1}{d^2r^2}$.
\end{proof}

\subsection{Pure states}

\begin{Restatementlemma*}{\scshape~\ref{lem::Hamming_distance_pure}.\hspace{7pt}}
    Let $D = d/2$ and $\mathcal{V} = \{-1,1\}^D$. The set of local binary hypotheses $\mathcal{F} = \{\rho_{\nu}\}_{\nu \in \mathcal{V}}$ defined by~\eqref{eqn::rho_nu_pure} induce a $2\eta$-Hamming separation for $\norm{\cdot}_F^2$ on $\mathcal{S}_{pure}(\mathbb{C}^d)$ where $\eta = \frac{\epsilon^2}{2d}$.
\end{Restatementlemma*}

\begin{proof}
    We can write any pure qudit as $\rho = \ket{\psi}\bra{\psi}$ with
    \begin{equation*}
        \ket{\psi} = \sum_{k = 1}^d \gamma_k \ket{k} \hspace{20pt} \text{and} \hspace{20pt} \ket{\psi'} = \sum_{k = 1}^d |\gamma_k| \ket{k}
    \end{equation*}
    we have
    \begin{align*}
        \norm{\rho - \rho_{\nu}}_{Fr}^2 &= 2 \left( 1 - \left| \bra{\psi}\ket{\psi_{\nu}} \right|^2 \right).
    \end{align*}
    For the inner product it holds
    \begin{align*}
        \left| \bra{\psi}\ket{\psi_{\nu}} \right|^2 &= \left| \sum_{k = 1}^D \gamma_k \frac{1 + \epsilon \nu_k}{\sqrt{d(1+\epsilon^2)}} + \gamma_{k + \frac{d}{2}} \frac{1 - \epsilon \nu_k}{\sqrt{d(1+\epsilon^2)}} \right|^2
        \\
        &\leq \left( \sum_{k = 1}^D \left| \gamma_k \right| \frac{1 + \epsilon \nu_k}{\sqrt{d(1+\epsilon^2)}} + \left| \gamma_{k + \frac{d}{2}} \right| \frac{1 - \epsilon \nu_k}{\sqrt{d(1+\epsilon^2)}} \right)^2 \leq \left| \bra{\psi'}\ket{\psi_{\nu}} \right|^2.
    \end{align*}
    Since the inner product between two pure states is smaller than one, we can bound the Frobenius-norm between the two states by the $\ell_2$-norm of their parameter vectors, i.e.
    \begin{equation*}
        \norm{\rho - \rho_{\nu}}_{Fr}^2 \geq 2 \left(1 - \left| \bra{\psi'}\ket{\psi_{\nu}} \right|^2  \right) \geq 2 \left(1 - \left| \bra{\psi'}\ket{\psi_{\nu}} \right|  \right) = \norm{\ket{\psi'} - \ket{\psi_{\nu}}}_{\ell_2}^2
    \end{equation*}
    and it holds
    \begin{align*}
        \norm{\ket{\psi'} - \ket{\psi_{\nu}}}_{\ell_2}^2 &= \sum_{k = 1}^D \left(  \left| \gamma_k \right| - \frac{1 + \epsilon \nu_k}{\sqrt{d(1+\epsilon^2)}}\right)^2 + \left(  \left| \gamma_{k + \frac{d}{2}} \right| - \frac{1 - \epsilon \nu_k}{\sqrt{d(1+\epsilon^2)}}\right)^2
        \intertext{By defining $\mathbf{v}(\rho)_k = \sgn\left( \left| \gamma_k\right| - \frac{1}{\sqrt{d(1+\epsilon^2)}}  \right)$ we can lower bound this term by}
        \norm{\ket{\psi'} - \ket{\psi_{\nu}}}_{\ell_2}^2 &= \frac{\epsilon^2}{d(1+\epsilon^2)} \sum_{k = 1}^D \mathbbm{1}_{\{ \mathbf{v}(\rho)_k \neq \nu_k\}} \geq \frac{\epsilon^2}{d} \sum_{k = 1}^D \mathbbm{1}_{\{ \mathbf{v}(\rho)_k \neq \nu_k\}},
    \end{align*}
    which proves the lemma.
\end{proof}

\label{sec::Appendix_pure_states}
\begin{Restatementlemma*}{\scshape~\ref{lem::upper_bound_distances_pure_states}.\hspace{7pt}}
    Let $\epsilon \leq 1$, $D = \frac{d}{2}$ and $(\rho_{\nu})_{\nu \in \mathcal{V}}$ as in~\eqref{eqn::rho_nu_pure}. Then, for the states defined in Definition~\ref{defn::mixed_distribution states} it holds
    \begin{equation*}
        \sup_{\norm{A}_{op} \leq 1} \sum_{j = 1}^D \frac{1}{2^{D-1}} \sum_{\nu \in \{-1,1\}^{D-1}} \Tr \left[ A (\rho_{\nu_{j+},i} - \rho_{\nu_{j-}, i})\right]^2 \leq \frac{392 \epsilon^2}{d}.
    \end{equation*}
\end{Restatementlemma*}

\begin{proof}
Let $A$ be positive semi-definite such that $\lambda_{max}(A) \leq 1$. We then consider the term
\begin{align*}
    &\sum_{j = 1}^D \frac{1}{2^{D-1}} \sum_{\nu \in \{-1,1\}^{D-1}} \Tr\left[ A (\rho_{\nu_{j+}} - \rho_{\nu_{j-}}) \right]^2 
    \\
    = &\sum_{j = 1}^D \frac{1}{2^{D-1}} \sum_{\nu \in \{-1,1\}^{D-1}} \left( \bra{\psi_{\nu_{j+}}}A\ket{\psi_{\nu_{j+}}} -  \bra{\psi_{\nu_{j-}}}A\ket{\psi_{\nu_{j-}}}\right)^2.
\end{align*}
For any $\nu \in \{-1,1\}^D$ we rewrite $\ket{\psi_{\nu_{j+}}} = \ket{\psi_{\nu_{j+}, 1}} + \ket{\psi_{\nu_{j+}, 2}}$ and $\ket{\psi_{\nu_{j-}}} = \ket{\psi_{\nu_{j-}, 1}} + \ket{\psi_{\nu_{j-}, 2}}$, where 
\begin{align*}
    \ket{\psi_{\nu_{j+}, 1}} &= \sum_{k \neq j} \frac{1 + \epsilon \nu_k}{\sqrt{d (1+\epsilon^2)}} \ket{k} + \frac{1 - \epsilon \nu_k}{\sqrt{d (1+\epsilon^2)}} \ket{k + \frac{d}{2}}
    \\
    \ket{\psi_{\nu_{j+}, 2}} &= \frac{1 + \epsilon}{\sqrt{d(1 + \epsilon^2)}} \ket{j} + \frac{1 - \epsilon}{\sqrt{d(1 + \epsilon^2)}} \ket{j + \frac{d}{2}}
    \intertext{and}
    \ket{\psi_{\nu_{j-}, 1}} &= \sum_{k \neq j} \frac{1 + \epsilon \nu_k}{\sqrt{d (1+\epsilon^2)}} \ket{k} + \frac{1 - \epsilon \nu_k}{\sqrt{d (1+\epsilon^2)}} \ket{k + \frac{d}{2}} = \ket{\psi_{\nu_{j+}, 1}}
    \\
    \ket{\psi_{\nu_{j-}, 2}} &= \frac{1 - \epsilon}{\sqrt{d(1 + \epsilon^2)}} \ket{j} + \frac{1 + \epsilon}{\sqrt{d(1 + \epsilon^2)}} \ket{j + \frac{d}{2}}.
\end{align*}
This allows us to rewrite 
\begin{align}
    \Tr\left[ A (\rho_{\nu_{j+}} - \rho_{\nu_{j-}}) \right]^2 =& \Big( \bra{\psi_{\nu_{j+}, 1}}A\ket{\psi_{\nu_{j+}, 1}} + \bra{\psi_{\nu_{j+}, 2}}A\ket{\psi_{\nu_{j+}, 1}} \notag
    \\
    & + \bra{\psi_{\nu_{j+}, 2}}A\ket{\psi_{\nu_{j+}, 2}} + \bra{\psi_{\nu_{j+}, 1}}A\ket{\psi_{\nu_{j+}, 2}} \notag
    \\
    & - \bra{\psi_{\nu_{j-}, 1}}A\ket{\psi_{\nu_{j-}, 1}} - \bra{\psi_{\nu_{j-}, 2}}A\ket{\psi_{\nu_{j-}, 1}} \notag
    \\
    & - \bra{\psi_{\nu_{j-}, 2}}A\ket{\psi_{\nu_{j-}, 2}} - \bra{\psi_{\nu_{j-}, 1}}A\ket{\psi_{\nu_{j-}, 2}} \Big)^2 \notag
    \\
    =& \Big( \bra{\psi_{\nu_{j+}, 2}}A\ket{\psi_{\nu_{j+}, 2}} - \bra{\psi_{\nu_{j-}, 2}}A\ket{\psi_{\nu_{j-}, 2}} \notag
    \\
    &  + \bra{\psi_{\nu_{j+}, 2}}A\ket{\psi_{\nu_{j+}, 1}} - \bra{\psi_{\nu_{j-}, 2}}A\ket{\psi_{\nu_{j-}, 1}} \notag
    \\
    &  + \bra{\psi_{\nu_{j+}, 1}}A\ket{\psi_{\nu_{j+}, 2}} - \bra{\psi_{\nu_{j-}, 1}}A\ket{\psi_{\nu_{j-}, 2}} \Big)^2 \notag
    \\
    &= \Big( \bra{\psi_{\nu_{j+}, 2}}A\ket{\psi_{\nu_{j+}, 2}} - \bra{\psi_{\nu_{j-}, 2}}A\ket{\psi_{\nu_{j-}, 2}} \notag
    \\
    & + 2 \Re\big( \bra{\psi_{\nu_{j+}, 1}}A\ket{\psi_{\nu_{j+}, 2}} - \bra{\psi_{\nu_{j-}, 1}}A\ket{\psi_{\nu_{j-}, 2}} \big) \Big)^2 \notag
    \\
    \leq& \Big( \bra{\psi_{\nu_{j+}, 2}}A\ket{\psi_{\nu_{j+}, 2}} - \bra{\psi_{\nu_{j-}, 2}}A\ket{\psi_{\nu_{j-}, 2}} \Big)^2 \notag
    \\
    \label{eqn::sum_over_differences_in_scalar_products}
    & + 4 \Re\big( \bra{\psi_{\nu_{j+}, 1}}A\ket{\psi_{\nu_{j+}, 2}} - \bra{\psi_{\nu_{j-}, 1}}A\ket{\psi_{\nu_{j-}, 2}} \Big)^2.
\end{align}
We now consider both summands separately. Firstly, we have
\begin{align*}
    \bra{\psi_{\nu_{j+}, 2}}A\ket{\psi_{\nu_{j+}, 2}}
    =& \frac{1}{d (1 + \epsilon^2)} \Bigg( (1+\epsilon)^2 \bra{j}A\ket{j} + (1-\epsilon)^2 \bra{j + \frac{d}{2}}A\ket{j + \frac{d}{2}} 
    \\
    & \qquad \qquad \quad+ 2 (1- \epsilon^2) \Re \bigg(\bra{j}A\ket{j + \frac{d}{2}} \bigg) \Bigg)
\end{align*}
Analogously, we have
\begin{align*}
    \bra{\psi_{\nu_{j-}, 2}}A\ket{\psi_{\nu_{j-}, 2}}
    =& \frac{1}{d (1 + \epsilon^2)} \Bigg( (1-\epsilon)^2 \bra{j}A\ket{j} + (1+\epsilon)^2 \bra{j + \frac{d}{2}}A\ket{j + \frac{d}{2}} 
    \\
    & \qquad \qquad \quad + 2 (1- \epsilon^2) \Re \bigg(\bra{j}A\ket{j + \frac{d}{2}} \bigg) \Bigg).
\end{align*}
Combining the two gives
\begin{align*}
    &\bra{\psi_{\nu_{j+}, 2}}A\ket{\psi_{\nu_{j+}, 2}} - \bra{\psi_{\nu_{j-}, 2}}A\ket{\psi_{\nu_{j-}, 2}}
    \\
    =& \frac{1}{d(1+\epsilon^2)} \left( \bra{j}A\ket{j} \left( (1+ \epsilon)^2 - (1-\epsilon)^2 \right) + \bra{j + \frac{d}{2}}A\ket{j + \frac{d}{2}} \left( (1-\epsilon)^2 - (1+\epsilon)^2 \right) \right)
    \\
    =& \frac{1}{d(1+\epsilon^2)} \left( 4 \epsilon \bra{j}A\ket{j} - 4 \epsilon \bra{j + \frac{d}{2}}A\ket{j + \frac{d}{2}} \right)
\end{align*}
Since this term is independent of the actual $\nu$, we have
\begin{align}
\label{eqn::upper_bound_first_summand}
    2 \sum_{j = 1}^D \frac{1}{2^{D-1}} \sum_{\nu \in \{-1,1\}^{D-1}} \Big( \bra{\psi_{\nu_{j+}, 2}}A\ket{\psi_{\nu_{j+}, 2}} - \bra{\psi_{\nu_{j-}, 2}}A\ket{\psi_{\nu_{j-}, 2}} \Big)^2 &\leq \sum_{j = 1}^D \frac{16 \epsilon^2}{d^2(1+\epsilon^2)^2} \notag
    \\
    &= \frac{8 \epsilon^2}{d(1+\epsilon^2)^2}
\end{align}
using the fact that $A$ is positive semi-definite with largest eigenvalue being smaller than 1. Let us now consider the second term in~\eqref{eqn::sum_over_differences_in_scalar_products}. This is the quantum part, that only appears due to the non-commutative nature of matrices. Here, we use the fact that $\ket{\psi_{\nu_{j-}, 1}} = \ket{\psi_{\nu_{j+}, 1}}$ to rewrite
\begin{equation*}
    \Big( \Re\big( \bra{\psi_{\nu_{j+}, 1}}A\ket{\psi_{\nu_{j+}, 2}} - \bra{\psi_{\nu_{j-}, 1}}A\ket{\psi_{\nu_{j-}, 2}} \big) \Big)^2 \leq \left| \bra{\psi_{\nu_{j+}, 1}} A \left( \ket{\psi_{\nu_{j+}, 2}} - \ket{\psi_{\nu_{j-}, 2}} \right) \right|^2.
\end{equation*}
Now it holds
\begin{align*}
    \ket{\psi_{\nu_{j+}, 2}} - \ket{\psi_{\nu_{j-}, 2}} =& \frac{1}{\sqrt{d (1+\epsilon^2)}} \bigg( (1+\epsilon)\ket{j} + (1-\epsilon) \ket{j + \frac{d}{2}} 
    \\
    &\qquad \qquad \quad- (1-\epsilon) \ket{j} - (1+\epsilon)\ket{j + \frac{d}{2}} \bigg)
    \\
    =& \frac{1}{\sqrt{d (1+\epsilon^2)}} \left( 2 \epsilon \ket{j} - 2\epsilon \ket{j + \frac{d}{2}} \right)
    \\
    =& \frac{2\epsilon}{\sqrt{d(1 + \epsilon^2)}} \left( \ket{j} - \ket{j + \frac{d}{2}} \right),
\end{align*}
which gives
\begin{alignat*}{2}
    & \bra{\psi_{\nu_{j+}, 1}} A && \left( \ket{\psi_{\nu_{j+}, 2}} - \ket{\psi_{\nu_{j-}, 2}} \right)
    \\
    =& \frac{2\epsilon}{d (1+\epsilon^2)} && \sum_{k \neq j} \left( (1+ \epsilon \nu_k) \bra{k} + (1-\epsilon \nu_k) \bra{k + \frac{d}{2}} \right) A \left( \ket{j} - \ket{j + \frac{d}{2}} \right)
    \\
    =& \frac{2\epsilon}{d (1+\epsilon^2)} \bigg( && \sum_{k \neq j} \left( \bra{k} + \bra{k + \frac{d}{2}} \right) A \left( \ket{j} - \ket{j + \frac{d}{2}} \right) 
    \\
    & && + \epsilon \sum_{k \neq j} \nu_k \left( \bra{k} - \bra{k + \frac{d}{2}} \right) A \left( \ket{j} - \ket{j + \frac{d}{2}} \right) \bigg)
    \\
    =& \frac{4\epsilon}{d (1+\epsilon^2)} \bigg( && \bra{\Psi_1} A \ket{\phi_j} - \left(\bra{j} + \bra{j + \frac{d}{2}}\right) A \ket{\phi_j} + \epsilon \bra{\Psi_{2, \nu}} A \ket{\phi_j} \bigg)
\end{alignat*}
where we define
\begin{equation*}
    \ket{\Psi_1} = \sum_{k = 1}^{D} \ket{k} + \ket{k + \frac{d}{2}} \hspace{20pt} \text{and} \hspace{20pt} \ket{\Psi_{2, \nu}} = \sum_{k \neq j} \nu_k \left( \ket{k} - \ket{k + \frac{d}{2}} \right).
\end{equation*}
and the orthonormal vectors.
\begin{equation*}
    \ket{\phi_j} = \frac{\ket{j} - \ket{j + \frac{d}{2}}}{\sqrt{2}}.
\end{equation*}
We can therefore write
\begin{align}
    &\sum_{j = 1}^D \frac{1}{2^{D-1}} \sum_{\nu \in \{-1,1\}^{D-1}} \left| \bra{\psi_{\nu_{j+}, 1}} A \left( \ket{\psi_{\nu_{j+}, 2}} - \ket{\psi_{\nu_{j-}, 2}} \right) \right|^2 \notag
    \\
    =& \frac{8\epsilon^2}{d^2(1+\epsilon^2)^2} \sum_{j = 1}^D  \frac{1}{2^{D-1}} \sum_{\nu \in \{-1,1\}^{D-1}} \bigg| \bra{\Psi_1} A \ket{\phi_j} + \left(\bra{j} + \bra{j + \frac{d}{2}}\right) A \ket{\phi_j} \notag
    \\
    & \qquad \qquad \qquad \qquad \qquad \qquad \qquad + \epsilon \bra{\Psi_{2, \nu}} A \ket{\phi_j} \bigg|^2 \notag
    \\
    \leq& \frac{32 \epsilon^2}{d^2(1+\epsilon^2)^2} \sum_{j = 1}^D  \bigg( \frac{1}{2^{D-1}} \sum_{\nu \in \{-1,1\}^{D-1}} \left| \bra{\Psi_1} A \ket{\phi_j} \right|^2 \notag
    \\
    & \qquad \qquad \qquad \quad + \frac{1}{2^{D-1}} \sum_{\nu \in \{-1,1\}^{D-1}} \left|\Big(\bra{j} +  \bra{j + \frac{d}{2}}\Big) A \ket{\phi_j} \right|^2 \notag
    \\
    \label{eqn::cross_terms_average}
    & \qquad \qquad \qquad \quad+ \epsilon^2 \frac{1}{2^{D-1}} \sum_{\nu \in \{-1,1\}^{D-1}} \left |\bra{\Psi_{2, \nu}} A \ket{\phi_j} \right|^2 \bigg).
\end{align}
Since $\ket{\Psi_1}$ is independent of $\nu$, we have
\begin{equation*}
    \frac{1}{2^{D-1}} \sum_{\nu \in \{-1,1\}^{D-1}} \left| \bra{\Psi_1} A \ket{\phi_j} \right|^2  =  \left| \bra{\Psi_1} A \ket{\phi_j} \right|^2.
\end{equation*}
The second summand is also independent of $\nu$ such that
\begin{equation*}
    \frac{1}{2^{D-1}} \sum_{\nu \in \{-1,1\}^{D-1}} \left|\Big(\bra{j} +  \bra{j + \frac{d}{2}}\Big) A \ket{\phi_j} \right|^2 = \left|\Big(\bra{j} +  \bra{j + \frac{d}{2}}\Big) A \ket{\phi_j} \right|^2 \leq 2 .
\end{equation*}
For the third summand, we have
\begin{align*}
    \frac{1}{2^{D-1}} \sum_{\nu \in \{-1,1\}^{D-1}} \left|\bra{\Psi_{2, \nu}} A \ket{\phi_j} \right|^2 =& \frac{1}{2^{D-1}} \sum_{\nu \in \{-1,1\}^{D-1}} \bra{\phi_j} A \ket{\Psi_{2, \nu}}\bra{\Psi_{2, \nu}} A \ket{\phi_j}
    \\
    =& \bra{\phi_j}A \left( \frac{1}{2^{D-1}} \sum_{\nu \in \{-1,1\}^{D-1}} \ket{\Psi_{2, \nu}}\bra{\Psi_{2, \nu}} \right) A \ket{\phi_j}.
\end{align*}
The inner sum may then be bounded as follows by making use of the independence structure of the Rademacher random variable $\nu_k$
\begin{align*}
    &\frac{1}{2^{D-1}} \sum_{\nu \in \{-1,1\}^{D-1}} \ket{\Psi_{2, \nu}}\bra{\Psi_{2, \nu}} 
    \\
    =& \frac{1}{2^{D-1}} \sum_{\nu \in \{-1,1\}^{D-1}} \sum_{k \neq j} \sum_{l \neq j} \nu_k \nu_l \left( \ket{k} - \ket{k + \frac{d}{2}} \right) \left( \bra{l} - \bra{l + \frac{d}{2}} \right)
    \\
    =& \sum_{k  \neq j} \left( \ket{k} - \ket{k + \frac{d}{2}} \right) \left( \bra{k} - \bra{k + \frac{d}{2}} \right) \frac{1}{2^{D-1}} \sum_{\nu \in \{-1,1\}^{D-1}} \nu_k^2 
    \\
    &+ \sum_{\genfrac{}{}{0pt}{}{k \neq l}{k,l \neq j}} \left( \ket{k} - \ket{k + \frac{d}{2}} \right) \left( \bra{l} - \bra{l + \frac{d}{2}} \right) \frac{1}{2^{D-1}} \sum_{\nu \in \{-1,1\}^{D-1}} \nu_k \nu_l
\end{align*}
Since $\nu_k$ and $\nu_l$ are independent Rademacher random variables, we have
\begin{equation*}
    \frac{1}{2^{D-1}} \sum_{\nu \in \{-1,1\}^{D-1}} \nu_k^2  = 1 \hspace{20pt} \text{and} \hspace{20pt} \frac{1}{2^{D-1}} \sum_{\nu \in \{-1,1\}^{D-1}} \nu_k \nu_l = 0,
\end{equation*} 
allowing us to rewrite
\begin{align*}
    \frac{1}{2^{D-1}} \sum_{\nu \in \{-1,1\}^{D-1}} \ket{\Psi_{2, \nu}}\bra{\Psi_{2, \nu}} =& \sum_{k  \neq j} \left( \ket{k} - \ket{k + \frac{d}{2}} \right) \left( \bra{k} - \bra{k + \frac{d}{2}} \right)
    \\
    =& 2 \sum_{k  \neq j} \frac{\ket{k} - \ket{k + \frac{d}{2}}}{\sqrt{2}} \frac{\bra{k} - \bra{k + \frac{d}{2}}}{\sqrt{2}} =: A_2
\end{align*}
Note that the Matrix $A_2$ is diagonal with maximal eigenvalue $2$. As such, its matrix square root is given by
\begin{equation*}
    \sqrt{A_2} = \sqrt{2} \sum_{k  \neq j} \frac{\ket{k} - \ket{k + \frac{d}{2}}}{\sqrt{2}} \frac{\bra{k} - \bra{k + \frac{d}{2}}}{\sqrt{2}}
\end{equation*}
We can therefore rewrite equation~\eqref{eqn::cross_terms_average} as
\begin{align}
    &4 \Re\big( \bra{\psi_{\nu_{j+}, 1}}A\ket{\psi_{\nu_{j+}, 2}} - \bra{\psi_{\nu_{j-}, 1}}A\ket{\psi_{\nu_{j-}, 2}} \Big)^2
    \\
    \leq& 4 \sum_{j = 1}^D \frac{1}{2^{D-1}} \sum_{\nu \in \{-1,1\}^{D-1}} \left| \bra{\psi_{\nu_{j+}, 1}} A \left( \ket{\psi_{\nu_{j+}, 2}} - \ket{\psi_{\nu_{j-}, 2}} \right) \right|^2 \notag
    \\
    \leq& \frac{128 \epsilon^2}{d^2(1+\epsilon^2)^2} \sum_{j = 1}^D \left( \left| \bra{\Psi_1} A \ket{\phi_j} \right|^2 + 2 +  \bra{\phi_j} A \sqrt{A_2} \sqrt{A_2} A \ket{\phi_j} \right) \notag
    \\
    \leq& \frac{128 \epsilon^2}{d^2(1+\epsilon^2)^2}  \left( \sum_{j = 1}^D \left| \bra{\Psi_1} A \ket{\phi_j} \right|^2 + 2D + \sum_{j = 1}^D \norm{\sqrt{A_2} A \ket{\phi_j}}^2 \right) \notag
    \\
    \leq& \frac{128 \epsilon^2}{d^2(1+\epsilon^2)^2} \left( \norm{A \ket{\Psi_1}}^2 + 2D + \sum_{j = 1}^D \norm{\sqrt{A_2} A \ket{\phi_j}}^2 \right) \notag
    \intertext{using Bessel's inequality as the $\ket{\phi_j}$ are an orthonormal system. Since the maximal eigenvalue of $A$ and $\sqrt{A_2}$ are $1$ and $\sqrt{2}$ respectively and $\norm{\Psi_1}^2 = d$ for $D = \frac{d}{2}$, we can further bound this by}
    \leq& \frac{128 \epsilon^2}{d^2(1+\epsilon^2)^2} \left( d + d + d \right) \notag
    \\
    \label{eqn::cross_terms_average_final}
    \leq& \frac{384 \epsilon^2}{d(1+\epsilon^2)^2}.
\end{align}
Combining the results of~\eqref{eqn::upper_bound_first_summand} and~\eqref{eqn::cross_terms_average_final} together with~\eqref{eqn::sum_over_differences_in_scalar_products} gives 
\begin{equation*}
    \sup_{\norm{A}_{op} \leq 1} \sum_{j = 1}^D \frac{1}{2^{D-1}} \sum_{\nu \in \{-1,1\}^{D-1}} \Tr\left[ A (\rho_{\nu_{j+}} - \rho_{\nu_{j-}}) \right]^2 \leq \frac{8 \epsilon^2}{d(1+\epsilon^2)^2} + \frac{384 \epsilon^2}{d (1+\epsilon^2)^2} \leq \frac{392 \epsilon^2}{d}.
\end{equation*}
\end{proof}

\subsection{Probability vector of a pure state}
\begin{lemma}
\label{lem::Hamming_distance_probability_vector}
    Let $D = d/2$ and $\mathcal{V} = \{-1,1\}^D$ and $p(\cdot)$ be the functioin that maps a pure state on its corresponding probability vector. The set of local binary hypotheses $\mathcal{F} = \{\rho_{\nu}\}_{\nu \in \mathcal{V}}$ defined by~\eqref{eqn::definition_pure_states_lower_bound_prop_amplitude} induce a $2\eta$-Hamming separation for $\norm{p(\cdot)}_2^2$ on $\mathcal{S}_{pure}(\mathbb{C}^d)$ where $\eta = \frac{\epsilon^2}{2d^2}$.
\end{lemma}

\begin{proof}
    Let $\rho = \ket{\psi}\bra{\psi}$ be a pure qudit with
    \begin{equation*}
        \ket{\psi} = \sum_{k = 1}^d  \gamma_{k} \ket{k}.
    \end{equation*}
    For the probability amplitude vector $p = [|\gamma_1|,....,|\gamma_d|^2]^t$ we have the following result: Let $\mathbf{v}(\rho): \mathcal{S}_{pure}(\mathbb{C}^d) \to \{-1,1\}, \mathbf{v}(\rho)_k = \sgn\left( |\gamma_k(\rho)|^2 - \frac{1}{d} \right)$, then it holds
    \begin{align*}
        \norm{p - |\gamma_{\nu}|^2}_2^2 &= \sum_{k = 1}^{D} \left(|\gamma_k|^2 - \left(\frac{1 + \epsilon\nu_k}{d}\right)\right)^2 + \left(|\gamma_k|^2 - \left(\frac{1 - \epsilon\nu_k}{d}\right)\right)^2
        \\
        &= \sum_{k = 1}^D \left( \mathbf{v}(\rho)_k + \frac{\epsilon}{d} \nu_k \right)^2 + \left( \mathbf{v}(\rho)_k - \frac{\epsilon}{d} \nu_k \right)^2 \geq \frac{\epsilon^2}{d^2} \sum_{k = 1}^D \mathbbm{1}_{\{ \mathbf{v}(\rho)_k \neq \nu_k \}},
    \end{align*}
    which shows the claim.
\end{proof}

\begin{lemma}
\label{lem::upper_bound_distances_pure_states_for_amplitudes}
    Let $\epsilon \leq 1$, $D = \frac{d}{2}$ and $(\rho_{\nu})_{\nu \in \mathcal{V}}$ as in~\eqref{eqn::rho_nu_pure}. Then, for the states defined in Definition~\ref{defn::mixed_distribution states} it holds
    \begin{equation*}
        \sup_{\norm{A}_{op} \leq 1} \sum_{j = 1}^D \frac{1}{2^{D-1}} \sum_{\nu \in \{-1,1\}^{D-1}} \Tr \left[ A (\rho_{\nu_{j+}} - \rho_{\nu_{j-}})\right]^2 \leq \frac{\epsilon^2}{d}
    \end{equation*}
\end{lemma}

\begin{proof}
    We can express the trace of the positive part of any operator $\rho$ as (see \cite{audenaert_asymptotic_2008})
    \begin{equation}
    \label{eqn::Trace_variational_formulation}
        \Tr[X_+] = \max_A \left\{ \Tr[A X] \; \middle| \; 0 \leq A \leq \mathbbm{1} \right\}
    \end{equation}
    As such, using $\norm{X}_{Tr} = \frac{1}{2} \Tr[X_+] + \frac{1}{2} \Tr[X_-]$, we have
    \begin{equation*}
        \Tr \left[ A (\rho_{\nu_{j+}} - \rho_{\nu_{j-}})\right]^2 \leq \frac{1}{4} \norm{\rho_{\nu_{j+}} - \rho_{\nu_{j-}}}_{Tr}^2 = \frac{1}{4} \left( 1 - \left| \bra{\psi_{\nu+}\ket{\psi_{\nu-}}} \right|^2 \right).
    \end{equation*}
    For the inner product it holds
    \begin{align*}
        \left| \bra{\psi_{\nu+}\ket{\psi_{\nu-}}} \right|^2 &= \left|\sum_{k \neq j}^D \left( \frac{1 + \epsilon\nu_k}{d} + \frac{1 - \epsilon\nu_k}{d} + 2 \sqrt{\frac{1 + \epsilon}{d}}\sqrt{\frac{1-\epsilon}{d}} \right)\right|^2
        \\
        &= \left| 1 - \frac{2}{d} + 2 \sqrt{\frac{1}{d^2} - \frac{\epsilon^2}{d^2}} \right|^2
        \\
        &= \left| 1 - \frac{2 (1-\sqrt{1-\epsilon^2})}{d} \right|^2
    \end{align*}
    which shows
    \begin{align*}
        \Tr \left[ A (\rho_{\nu_{j+}} - \rho_{\nu_{j-}})\right]^2 &\leq \frac{1}{4}\left( 1 - \left( 1 - \frac{4}{d}\left( 1 - \sqrt{1-\epsilon^2} \right) + \frac{4}{d^2} \left( 1 - \sqrt{1 - \epsilon^2} \right)^2 \right) \right)
        \\
        &= \frac{1}{d^2}\left( 1 - \sqrt{1 - \epsilon^2} \right)\left( d - \left( 1 - \sqrt{1 - \epsilon^2} \right) \right)
        \\
        &\leq \frac{1}{d}\left( 1 - \sqrt{1 - \epsilon^2} \right)
        \\
        &\leq \frac{\epsilon^2}{d}.
    \end{align*}
\end{proof}

\end{appendix}

\end{document}